\documentclass[11pt]{amsart}

\usepackage{fullpage}
\usepackage{amssymb}
\usepackage{amsmath}
\usepackage{amsxtra}
\usepackage{amscd}
\usepackage{graphicx}
\usepackage{amsfonts}
\usepackage{pb-diagram}
\usepackage{float}
\usepackage{enumerate}
\usepackage{dsfont}
\usepackage{multirow}
\usepackage{color}
\usepackage{rotating}
\usepackage{url}
\usepackage{hyperref}
\definecolor{darkgreen}{rgb}{0,0.4,0.1}
\hypersetup{
   colorlinks=true,       	
   linkcolor=red,          
   citecolor=darkgreen,    
   filecolor=magenta,      
   urlcolor=blue			
}

\def\a{\alpha}
\def\b{\beta}
\def\g{\gamma}
\def\d{\delta}

\def\s{\sigma}
\def\leq{\leqslant}
\def\geq{\geqslant}

\numberwithin{equation}{section}
\newtheorem{nono-thm}{Theorem}
\newtheorem{thm}{Theorem}[section]
\newtheorem{lemma}[thm]{Lemma}
\newtheorem{cor}[thm]{Corollary}
\newtheorem{prop}[thm]{Proposition}

\newtheorem{defn}[thm]{Definition}

\newtheorem{rem}[thm]{Remark}

\def\wt{\widetilde}
\def\Z{{\mathbb Z}}

\def\C{{\mathbb C}}
\def\N{{\mathbb N}}
\def\Jtr{{\rm tr}_d}
\def\Jtru{\wt{\rm tr}_d}
\def\YH{{\rm Y}_{d,n}}

\def\Jtrs{{\rm tr}_{d,D}}
\def\Jtrsu{\wt{\rm tr}_{d,D}}
\def\Jtrm{{\rm tr}_{d,m}}
\def\Jtrmu{\wt{\rm tr}_{d,m}}

\def\zet{z}

\begin{document}

\newsavebox{\claspA}
\savebox{\claspA}(55,12)[bl]{
\put(0,0){\line(3,1){14.8}}
\put(0,5){\line(1,0){14.4}}
\put(14.8,5){\line(-3,1){14.8}}
\put(11.5,2.3){$\bullet$}

\put(16,5){\line(1,0){22}}
\put(23,8){\tiny $r_1$}

\put(37.5,2.3){$\bullet$}
\put(55.2,0){\line(-3,1){14.8}}
\put(40.5,5){\line(1,0){14.4}}
\put(40.5,5){\line(3,1){14.8}}
}

\newsavebox{\claspB}
\savebox{\claspB}(55,12)[bl]{
\put(0,0){\line(3,1){14.8}}
\put(0,5){\line(1,0){14.4}}
\put(14.8,5){\line(-3,1){14.8}}
\put(11.5,2.3){$\bullet$}

\put(16,5){\line(1,0){22}}
\put(15,9){\tiny $r_1-1$}

\put(37.5,2.3){$\bullet$}
\put(55.2,0){\line(-3,1){14.8}}
\put(40.5,5){\line(1,0){14.4}}
\put(40.5,5){\line(3,1){14.8}}
}

\newsavebox{\claspC}
\savebox{\claspC}(55,12)[bl]{
\put(0,0){\line(3,1){14.8}}
\put(0,5){\line(1,0){14.4}}
\put(14.8,5){\line(-3,1){14.8}}
\put(11.5,2.3){$\bullet$}

\put(16,5){\line(1,0){22}}
\put(25,7){\tiny $1$}

\put(37.5,2.3){$\bullet$}
\put(55.2,0){\line(-3,1){14.8}}
\put(40.5,5){\line(1,0){14.4}}
\put(40.5,5){\line(3,1){14.8}}
}

\newsavebox{\claspD}
\savebox{\claspD}(40,12)[bl]{
\put(0,0){\line(3,1){14.8}}
\put(0,5){\line(1,0){14.4}}
\put(14.8,5){\line(-3,1){14.8}}
\put(11.5,2.3){$\bullet$}

\put(22.5,2.3){$\bullet$}
\put(40.2,0){\line(-3,1){14.8}}
\put(25.5,5){\line(1,0){14.4}}
\put(25.5,5){\line(3,1){14.8}}
}

\newsavebox{\claspE}
\savebox{\claspE}(30,12)[bl]{
\put(0,0){\line(3,1){14.8}}
\put(0,5){\line(1,0){14.4}}
\put(14.8,5){\line(-3,1){14.8}}
\put(11.5,2.3){$\bullet$}

\put(29.2,0){\line(-3,1){14.8}}
\put(14.5,5){\line(1,0){14.4}}
\put(14.5,5){\line(3,1){14.8}}
}

\newsavebox{\claspF}
\savebox{\claspF}(32,10)[bl]{
\put(0,0.3){$\bullet$}
\put(4.5,3){\line(1,0){22}}
\put(14.5,4){\tiny $r$}
\put(26,0.3){$\bullet$}
}

\title[Identifying the invariants for classical knots and links from the Yokonuma--Hecke algebras]
  {Identifying the invariants for classical knots and links from the Yokonuma--Hecke algebras}

\author{M.~Chlouveraki}
\address{Laboratoire de Math\'{e}matiques, UVSQ, B\^{a}timent Fermat, 45 avenue des \^{E}tats--Unis, 78035 Versailles cedex, France}
\email{maria.chlouveraki@uvsq.fr}

\author{J.~Juyumaya}
\address{Instituto de Matem\'aticas, Universidad de Valpara\'{\i}so \\
Gran Breta\~na 1091, Valpara\'{\i}so, Chile.}
\email{juyumaya@uvach.cl}

\author{K.~Karvounis}
\address{Institut f\"{u}r Mathematik,
Universit\"{a}t Z\"{urich},
Winterthurerstrasse 190, CH-8057 Z\"{u}rich, Switzerland.}
\email{konstantinos.karvounis@math.uzh.ch}

\author{S.~Lambropoulou\\with an appendix by W.B.R. Lickorish}
\address{Department of Mathematics,
National Technical University of Athens,
Zografou campus, GR-157 80 Athens, Greece.}
\email{sofia@math.ntua.gr}
\urladdr{http://www.math.ntua.gr/~sofia}

\address{Department of Pure Mathematics and Mathematical Statistics,
Wilberforce Road,
Cambridge, CB3 0WB, UK.}
\email{wbrl@dpmms.cam.ac.uk}

\thanks{This research  has been co-financed by the European Union (European Social Fund - ESF) and Greek national funds through the Operational Program "Education and Lifelong Learning" of the National Strategic Reference Framework (NSRF) - Research Funding Program: THALES: Reinforcement of the interdisciplinary and/or inter-institutional research and innovation.}

\keywords{Framed braids, Yokonuma--Hecke algebras, Markov trace, E--condition, framed links, classical links, Homflypt polynomial, skein relation, skein invariants, generalization of the Homflypt polynomial.}

\subjclass[2010]{57M27, 57M25, 20F36, 20F38, 20C08}

\begin{abstract}
In this paper we announce the existence of a family of new $2$-variable polynomial invariants for oriented classical links defined via a Markov trace on the Yokonuma--Hecke algebra of type $A$. Yokonuma--Hecke algebras are generalizations of Iwahori--Hecke algebras, and this family contains the Homflypt polynomial, the famous $2$-variable invariant for classical links arising from the Iwahori--Hecke algebra of type $A$. 
We show that these invariants are topologically equivalent to the Homflypt polynomial on \textit{knots}, but not on \textit{links}, by providing pairs of Homflypt-equivalent links that are distinguished by our invariants. In order to do this, we prove that 
our invariants can be defined diagrammatically via a special skein relation involving {\it only crossings between different components}. We further generalize this family of invariants to a new $3$-variable skein link invariant which is stronger than the Homflypt polynomial. Finally, we present a closed formula for this invariant, by W.B.R. Lickorish, which uses Homflypt polynomials of sublinks and linking numbers of a given oriented link.
\end{abstract}

\maketitle
\setcounter{tocdepth}{1}
{\hypersetup{linkcolor=black}
\tableofcontents
}
\section*{Introduction} \label{s:intro}

One of the greatest accomplishments in knot theory and low-dimensional topology is the pioneering construction of the Jones polynomial by V.~F.~R.~Jones in 1984.
This is true for many reasons: The Jones polynomial was a new strong invariant of classical knots and links, so it advanced spectacularly the tabulation of knots. It made use for the first time of the Artin braid groups and their Markov equivalence via a Markov trace on the Temperley--Lieb algebras. It was easily computable thanks to the new diagrammatic skein methods developed by L.~H.~Kauffman. It made unexpected connections with the area of statistical mechanics. Finally, it led to the Witten--Reshetikhin--Turaev invariants of $3$-manifolds.

The Jones polynomial was immediately generalized to a $2$-variable invariant of classical knots and links, the Homflypt polynomial, specializations of which yield both the Jones polynomial and the Alexander polynomial. This invariant was constructed with the use of the Ocneanu trace  defined on the Iwahori--Hecke algebras of type $A$ \cite{jo}. The Iwahori--Hecke algebra  ${\rm H}_n(q)$ of type $A$, where $q$ is a non-zero complex number, is a quotient of the group algebra of the classical braid group $B_n$ of type $A$ over a quadratic relation. Every classical knot or link can be represented as a braid, and so each knot and link has an image in an Iwahori--Hecke algebra.  The Ocneanu trace $\tau$, which depends on a second parameter $z$, can be thus applied to knots and links. Re-scaling according to the positive and negative stabilization of the braid equivalence yields the Homflypt polynomial $P(q,z)$. The polynomial $P(q,z)$ can be also defined diagrammatically via a skein relation.

Now, in the 1960's, T.~Yokonuma \cite{yo}  introduced some new algebras as generalizations of the Iwahori--Hecke algebras in the context of Chevalley groups. Instead of considering the endomorphism rings of the permutation representation with respect to a Borel subgroup, he considered the endomorphism rings of the permutation representation with respect to a maximal unipotent subgroup. These algebras, which are particular cases of unipotent Hecke algebras, are called Yokonuma--Hecke algebras. 
The interest in these algebras was rekindled in the past years
in relation to their presentation \cite{ju1,ju2},
their interesting representation theory \cite{thi, chpa} and their topological applications \cite{jula1}--\cite{jula5}.

Let $d,n \in \N$ and let $q$ be a non-zero complex number. The Yokonuma--Hecke algebra $\YH(q)$ of type $A$ can be obtained as a quotient of the group algebra over $\C$ of the modular framed braid group $(\Z/d\Z)^n \rtimes B_n$ by the quadratic relation
$$
g_i^2 = 1 +(q-q^{-1})e_i g_i \quad \text{for all } i=1,\ldots,n-1,
$$
where $g_1,\ldots,g_{n-1}$ are the images of the ``braiding'' generators of $B_n$, $t_1,\ldots,t_n$ denote the ``framing'' generators of $(\Z/d\Z)^n$ and
$$e_i:=\frac{1}{d}\sum_{s=0}^{d-1}t_i^st_{i+1}^{d-s}\,\,, \quad \text{ for all  }i=1,\ldots,n-1,$$ are idempotents in $\C[(\Z/d\Z)^n ]$. For $d=1$, we have $e_i=1$ and the algebra  ${\rm Y}_{1,n}(q)$ coincides with the Iwahori--Hecke algebra ${\rm H}_n(q)$. For $q=\pm 1$,  ${\rm Y}_{d,n}(\pm 1)$ is the group algebra over $\C$ of the complex reflection group $(\Z/d\Z)^n \rtimes \mathfrak{S}_n$, where $\mathfrak{S}_n$ denotes the symmetric group on $n$ letters.
Note that the above presentation of the Yokonuma--Hecke algebra was introduced in \cite{chpa}, while
in the papers \cite{ju1,ju2,ju} and \cite{jula1}--\cite{jula5} a different presentation for the Yokonuma--Hecke algebra was used with the following quadratic relation:
$$
\wt{g_i}^2 = 1 + (u-1)e_i+(u-1)e_i \wt{g_i} \quad \text{for all } i=1,\ldots,n-1,
$$
where $u:=q^2$ and $\wt{g_i}:=g_i+(q-1)e_ig_i$. 
In order to differentiate between the two presentations of the algebra, we will denote the Yokonuma--Hecke algebra with the second presentation by $\YH(u)$.

In \cite{ju} J.~Juyumaya defined a Markov trace $\Jtru$ on  $\YH(u)$, thus making the Yokonuma--Hecke algebras natural candidates for the construction of invariants for framed knots and links. The trace $\Jtru$ depends on a parameter $\wt{z}$, similarly to the Ocneanu trace, but also on $d-1$ parameters  $x_1, \ldots, x_{d-1}$ corresponding to the framing generators. Trying to repeat a process similar to the construction of the Homflypt polynomial, it turned out \cite{jula2} that $\Jtru$ 
does not re-scale  according to the positive and negative stabilization of the framed braid equivalence, 
making it the only Markov trace known in the literature  that does not have straightforward re-scaling. The problem lies in the presence of the idempotents $e_i$ in the quadratic relations for the $\wt{g}_i$. In order for $\Jtru$ to re-scale, the framing parameters $x_1, \ldots, x_{d-1}$  had to satisfy a non-linear system of equations, the so-called ``{\rm E}--system'' \cite{jula2}; 
we call this the {\rm E}--condition.
As it was shown by P.~G\'{e}rardin, the solutions of the {\rm E}--system are parametrized by the non-empty subsets of $\Z/d\Z$ \cite[Appendix]{jula2}.
Hence, for each solution of the {\rm E}--system parametrized by a non-empty subset $D$ of $\Z/d\Z$, J.~Juyumaya and S.~Lambropoulou defined an invariant $\Gamma_{d,D}(u,\wt{z})$ for framed links \cite[Theorem~8]{jula2}.
Further, since $B_n$ embeds in  $(\Z/d\Z)^n \rtimes B_n$, and so classical links are contained in the set of framed links (they correspond to the framed links with all framings equal to $0$), the invariants $\Gamma_{d,D}(u,\wt{z})$ restrict to invariants $\Delta_{d,D}(u,\wt{z})$ for classical knots and links \cite[Theorem~6]{jula3}.
For $d=1$, we have ${\rm{\wt{tr}}}_1=\tau$ and $\Delta_{1,\{0\}}(u,\wt{z})=P(u,\wt{z})$.

For the past years, we have been trying to compare the invariants  $\Delta_{d,D}$, for $d>1$, with the Homflypt polynomial. In a first attempt, M.~Chlouveraki and S.~Lambropoulou showed in \cite{chla} that there is no suitable choice of parameters that will make $\Delta_{d,D}$ coincide with the Homflypt polynomial, unless $u=1$ or $\Jtru(e_i)=1$. Both conditions negate the consequences of the appearance of the idempotent $e_i$ in the quadratic relation for the braiding generator $\wt{g}_i$. Note that, even though classical knots and links do not involve framing generators when seen inside the modular framed braid group, computing the value of $\Jtru$ on their images in $\YH(u)$ will make use of the quadratic relation and thus make framing generators appear (in the form of the idempotents $e_i$). In \cite{chla}, we even showed that there is no algebra homomorphism between the algebra $\YH(u)$ and the algebra ${\rm H}_n(u)$ which respects the trace, unless again $\Jtru(e_i) = 1$. However, despite the results in \cite{chla}, the invariants $\Delta_{d,D}$ could still be topologically equivalent to the Homflypt polynomial, in the sense that they distinguish the same pairs of knots and links. Moreover, although we had a defining skein relation for the  framed link invariants $\Gamma_{d,D}$ \cite{jula2}, this skein relation could not apply to the invariants $\Delta_{d,D}$ since it contained framed links. This fact has rendered a diagrammatic comparison with the Homflypt  very difficult until now.

In \cite{chpa}, M.~Chlouveraki and L.~Poulain d'Andecy studied the representation theory of the Yokonuma--Hecke algebra and transformed its presentation to the one used in this paper.
On the algebra $\YH(q)$ we adapt the trace $\Jtru$ to a trace $\Jtr$ with parameters $z,x_1,\ldots,x_{d-1}$.
Invariants for framed and classical links, denoted by $\Phi_{d,D}(q,z)$ and $\Theta_{d,D}(q,z)$ respectively, can be then defined by imposing on $\Jtr$ the same E--condition as on $\Jtru$. One would expect that only a single invariant should arise, independently of the algebra presentation, similarly to the Homflypt polynomial, where the change of presentation corresponds to a simple change of variables. However, this is not the case for the invariants arising from the Yokonuma--Hecke algebra, and so
the invariants $\Gamma_{d,D}(u,\wt{z})$ (respectively $\Delta_{d,D}(u,\wt{z})$) and $\Phi_{d,D}(q,z)$ (respectively  $\Theta_{d,D}(q,z)$) may not be topologically equivalent. Again, for $d=1$, we have ${\rm{tr}}_1=\tau$ and $\Theta_{1,\{0\}}(q,z)=P(q,z)$.

In the mean time, computational packages were developed in order to compute the values of the invariants  on specific knots and links (see \cite{chjakala}). We note that the new quadratic relation reduced significantly the complexity of the problem.
Computational data on several Homflypt-equivalent pairs of knots and links indicated that the invariants $\Theta_{d,D}$  do not distinguish those pairs either, leading us to believe that the invariants $\Theta_{d,D}$  are topologically equivalent 
to the Homflypt polynomial. This belief was strengthened by the fact that Yokonuma--Hecke algebras are  natural generalizations of Iwahori--Hecke algebras and the invariants  $\Theta_{d,D}$ include the Homflypt polynomial as a particular case. Consequently, in the case of {\it knots}, S.~Jablan and K.~Karvounis were able to formulate a concrete conjecture (cf.~\cite{chjakala}), which is now Theorem~\ref{conjHomflypt} in this paper. Namely, 

\begin{nono-thm} \label{theorem1}
If $K$ is a \textit{knot}, then
\begin{center}
$\Theta_{d,D} (q,z)(K) = \Theta_{1,\{0\}} (q,z|D|) (K)= P(q,z|D|)(K).$
\end{center}
\end{nono-thm}

The proof of the above theorem requires the comparison of the Ocneanu trace with the specialized trace  $\Jtrs$, where $\Jtrs$ is the notation we use for $\Jtr$ when the parameters $x_1,\ldots,x_{d-1}$ are specialized to the solution of  the {\rm E}--system  parametrized by the non-empty subset $D$ of $\Z/d\Z$. Since we are only interested in classical knots and links, we need to compute $\Jtrs$ only on the images of the elements of $B_n$ in the algebra $\YH(q)$.  This process, as mentioned above, makes the framing generators $t_1,\ldots,t_n$ appear only in the form of the idempotents $e_i$.
This led to  Theorem~\ref{specialtrace}, another important result of this paper, which was conjectured by J.~Juyumaya (cf.~\cite{AJ}).
This theorem states the following:

\begin{nono-thm}\label{theorem2}
When computing $\Jtrs$ on images of classical braids, the trace rule involving the framing generators,
$$
\Jtrs (a \,t_{n+1}^k) = x_k\,\Jtrs (a) \qquad a\in\YH(q) \quad (1\leq k\leq d-1),
$$
can be replaced by two rules involving the idempotents $e_i$,
 $$\begin{array}{rcll}
\Jtrs (a e_n) &=& E_D\, \Jtrs (a) &\quad a\in \YH(q)
\\
\Jtrs (a e_n g_n) &=& z\, \Jtrs (a) &\quad a\in \YH(q),
\end{array}
$$
where $E_D:=\Jtrs(e_i) = 1/|D|$. 
\end{nono-thm}
So $\Jtrs$ depends only on parameters $q$, $z$ and $E_D=1/|D|$ when computed on images of classical braids.
As a consequence, we obtain that the invariants $\Theta_{d,D}$ are in fact parametrized by the natural numbers, so they can be simply denoted as $\Theta_d$,
with $\Theta_d:=\Theta_{d,\Z/d\Z}$ and $E_{\Z/d\Z}:=1/d$.
This result is also valid for the invariants $\Delta_{d,D}$.  

Theorem~\ref{theorem2} enabled the development of a program for computing the invariants $\Theta_d$ with much lower complexity \cite{ka} and it is available at \url{http://www.math.ntua.gr/~sofia/yokonuma}.

We next investigate the behaviour of the invariants $\Theta_d$ on \textit{links}.  Surprisingly, Theorem~\ref{theorem1} (Theorem~\ref{conjHomflypt} in the paper)
does not hold for links, except in the case of disjoint unions of knots, where an analogous result holds (Theorem~\ref{disj_links}), namely,

\begin{nono-thm}
If $L$ is a disjoint union of $k$ knots, we have 
\begin{center}
$\Theta_d (q,z)(L) = E_D^{1-k} \Theta_1 (q,{z}/{E_D})(L) = E_D^{1-k} P(q,{z}/{E_D})(L)$.
\end{center}
\end{nono-thm}

Now, we have already mentioned that the framed link invariants $\Gamma_{d,D}$ satisfy a defining skein relation. With the change of the quadratic relation for the braiding generators of $\YH(q)$, the analogous skein relation for the framed link invariants $\Phi_{d,D}$ is given by Equation \eqref{framed_skein}. Using this, we were able to prove that the invariants $\Theta_d$ satisfy a special skein relation, which can \emph{only} be applied on \emph{crossings of different components}. Namely, 

\begin{nono-thm}\label{intro_skein}
The following special skein relation holds for $\Theta_d$:
$$
\frac{1}{\sqrt{\lambda_D}} \Theta_d( \includegraphics[scale=0.5]{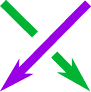} ) - \sqrt{\lambda_D} \Theta_d(\includegraphics[scale=0.5]{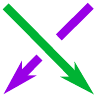} ) = (q-q^{-1})\, \Theta_d(\includegraphics[scale=0.5]{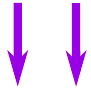}),
$$
where different colors represent different components of a link and $\lambda_D:=\frac{z - (q-q^{-1})E_D}{z}$.
\end{nono-thm}

We note that the above skein relation (Proposition~\ref{skein_ei}) could not have been derived for the invariants $\Delta_d$ from the skein relation of the invariants $\Gamma_{d,D}$ with the old quadratic relation. Note also that it is identical to the skein relation of the Homflypt polynomial $P$ \eqref{skein of hom} considered at variables $(q, \lambda_D)$. The above result led to the reveal of the behaviour of the invariants $\Theta_d$ on links. Namely, it led to Theorem~\ref{links}, which states the following:

\begin{nono-thm} \label{theorem5}
The value of $\Theta_d$ on a link $L$ is a linear combination of the Homflypt polynomials of disjoint unions of knots obtained by the skein relation. 
\end{nono-thm}

As a consequence of the above results, the invariants $\Theta_d$ can be computed on any classical link diagram $L$ via skein relations 
by applying the following procedure:
\begin{enumerate}[Step 1.]
\item Apply the skein relation of Theorem~\ref{intro_skein} (Proposition~\ref{skein_ei}) on crossings between different components until the link $L$ is decomposed into disjoint unions of knots.
 \smallbreak
\item Following Theorem~\ref{theorem5} (Theorem~\ref{links}), write $\Theta_d(L)$ as a linear combination of the Homflypt polynomials of these disjoint unions of knots. \smallbreak
\item Apply the skein relation of the Homflypt polynomial \eqref{skein of hom} to calculate the latter. \smallbreak
\end{enumerate}

The intrinsic difference from the Homflypt polynomial on a link lies in the different values of $\Theta_d$ on the unlinks with more than one component, in which the value $E_D$ appears. Theorem~\ref{theorem5} (Theorem~\ref{links}) was a strong indication that the invariants $\Theta_d$ might not be topologically equivalent to the Homflypt polynomial after all. 

Finally, we reach the end of our quest, which is not the one that we expected when we started this paper.

\begin{nono-thm}
The classical link invariants $\Theta_d$ for $d \geq 2$ are \textit{not} topologically equivalent to the Homflypt polynomial.
\end{nono-thm}

Indeed, in \cite{chli} one can find all 4.188 links with up to 11 crossings and the values of the polynomial $P$ on them. We singled out 89 pairs with the same Homflypt value which are not the same as unoriented links. We computed the invariants $\Theta_d$ on all these pairs and we found that they distinguish six of them, given in Table~\ref{pairs}. For one of these pairs, namely,
\begin{table}[H]
\centering
{\renewcommand{\arraystretch}{1.2}
\begin{tabular}{|c|c|}
\hline
Link notation & Braid word\\\hline
$L11n358\{0,1\}$ & $\s_1 \s_2^{-1} \s_3^{-1} \s_4^{-1} \s_3^2 \s_5^{-1} \s_4 \s_3^{-1} \s_2 \s_1^{-1} \s_3^{-1} \s_2^{-1} \s_4^{-1} \s_3 \s_2^{-3} \s_5 \s_4 \s_3^{-1}$\\\hline
$L11n418\{0,0\}$ & $\s_1^{-1} \s_2^{-1} \s_3 \s_2^{-1} \s_3^{-1} \s_2 \s_1^{-1} \s_3^{-2} \s_2 \s_3^{-1}$\\\hline
\end{tabular}}
\end{table}
\noindent
we give a diagrammatic proof using the special skein relation. Similar diagrammatic proofs can be given for the remaining $5$ pairs.

\smallbreak 
To summarize, the family of invariants $\{\Theta_d(q,\lambda_D)\}_{d\in\N}$ include the Homflypt polynomial $P$ for $d=1$ and they are different from $P$ for $d > 1$.
The invariants $\Theta_d$ are also different from the Kauffman polynomial, since they coincide with the Homflypt polynomial on knots, and there exists at least one pair of Kauffman-equivalent knots which are distinguished by $P$. 
Furthermore, the invariants $\Theta_d$ can be defined diagrammatically via a skein theoretical approach and this fact is very important, since there are very few link invariants defined through skein relations. Regarding properties, these invariants behave similarly to $P$ under reversing orientation, split links, connected sums, mirror imaging (see \cite{chjakala}) and mutation (see Proposition~\ref{mutants}).

\smallbreak
Letting further $E$ be a parameter in place of $E_D$ (which so far equals $1/d$), we show that our family of invariants $\{\Theta_d(q,\lambda_D)\}_{d \in \N}$ generalizes to a \emph{new 3-variable skein link invariant} $\Theta(q,\lambda,E)$ (Theorem~\ref{invE}):

\begin{nono-thm}\label{theorem7}
Let $q,\,\lambda,\, E$ be indeterminates.
There exists a unique isotopy invariant  of classical oriented links $\Theta: \mathcal{L} \rightarrow \C[q^{\pm 1}, \lambda^{\pm 1}, E^{\pm 1}]$ defined by the following rules:
\begin{enumerate}
\item For a disjoint union $L$ of $k$ knots, with $k \geq 1$, it holds that:
$$\Theta(L) =  E^{1-k} \,P(L).$$
\item On crossings involving different components the following skein relation holds:
$$
\frac{1}{\sqrt{\lambda}} \, \Theta(L_+) - \sqrt{\lambda} \, \Theta(L_-) = (q-q^{-1}) \, \Theta(L_0),
$$
where $L_+$, $L_-$, $L_0$ is a Conway triple. \smallbreak
\end{enumerate}
\end{nono-thm}

In particular, the invariant $\Theta(q,\lambda,E)$ specializes to $P$ and is stronger than $P$ on \textit{links}. Moreover, $\Theta$ retains all the properties satisfied by the invariants $\Theta_d$, since it is defined diagrammatically by the special skein relation of $\Theta_d$ (Theorem~\ref{intro_skein}) and by its values on disjoint unions of knots. We prove the well-definedness of $\Theta$ by comparing it to an invariant $\overline{\Theta}$ of tied links. F.~Aicardi and J.~Juyumaya have defined in \cite{AJ2} an invariant $\overline{\Delta}$ of tied links via a Markov trace on the algebra of braids and ties. In this paper we construct the invariant $\overline{\Theta}$ analogously by changing the presentation of the algebra of braids and ties using the new quadratic relation (see Section~\ref{threeVar}). The invariant $\overline{\Delta}$ is related to the invariants $\Delta_d$ and there are computational indications \cite{Ai} that these invariants may be distinct from the invariants $\overline{\Theta}$ and $\Theta_d$ respectively. A diagrammatic skein-theoretic proof of Theorem~\ref{theorem7} is given in \cite{kala}.

Finally, in the Appendix~\ref{lickorish}, W.B.R.~Lickorish proves a closed formula for the 3-variable invariant $\Theta$, showing that it is in fact a complicated mixture of linking numbers and the Homflypt polynomial of sublinks (Theorem~\ref{theta_linking_P}). Namely:

\begin{nono-thm}[W.B.R. Lickorish]\label{thm_lickorish}
Let $L$ be an oriented link with $n$ components.  Then
$$
\Theta (L) = \sum_{k=1}^n \mu^{k-1}E_k \sum_\pi \lambda^{\nu(\pi)}P(\pi L)
$$
where the second summation is over all partitions $\pi$ of the components of $L$ into $k$ (unordered) subsets and $P(\pi L)$ denotes the product of the Homflypt polynomials of the $k$ sublinks of $L$ defined by $\pi$. Furthermore, $\nu(\pi)$ is the sum of all linking numbers of pairs of components of $L$ that are in distinct sets of $\pi$, $E_k = (E^{-1} - 1)(E^{-1} - 2) \dots (E^{-1} - k + 1)$, with $E_1 =1$, and $\mu = \frac{\lambda^{-{1/2}} - \lambda^{{1/ 2}}}{q - q^{-1}}$.
\end{nono-thm}

From the above, the strength of $\Theta$ over the Homflypt polynomial lies on the fact that $\Theta$ takes into account the Homflypt polynomials of sublinks and also linking numbers.  Note that Theorem~\ref{thm_lickorish} provides us with a new way of defining the invariant $\Theta$ (see Theorem~\ref{theta_new_def}). Concluding, the invariant $\Theta$ can be defined algebraically via the trace $\Jtrs$ on $\YH$, skein-theoretically using the special skein relation \cite{kala} and combinatorially using Theorem~\ref{thm_lickorish}. An immediate consequence of Theorem~\ref{thm_lickorish} is that the 2-variable invariants $\Theta_d(L)$, where $L$ is an $n$-component link, are topologically equivalent for all $d \geq n$ (Theorem~\ref{theta_d_n}). The same result has been proved independently in \cite{pawa} using representation theory techniques.

\smallbreak
Our results lead to various interesting research directions. To mention some: one could explore further how the invariants $P$, $\Theta_d$, $\Theta$ and $\Delta_d$, $\overline{\Delta}$ compare among themselves. Moreover, the corresponding Jones-type invariants related to the framization of the Temperley--Lieb algebra are introduced in \cite{gojukolaf}, where they are explored in comparison to the Jones polynomial and proved to be topologically non-equivalent to it. In another direction, in \cite{kala} state sum models are constructed using the skein theoretical methods of this paper. Further, under the light of the new results for classical link invariants arising from the Yokonuma--Hecke algebra, it makes sense to revisit the framed and the $p$-adic framed link invariants of \cite{jula2}, as well as those that are related to the framization of the Temperley--Lieb algebra \cite{gojukolaf}, and to explore the possibility of obtaining $3$-manifold invariants and possibly new results in the domain of transverse links \cite{chjakala}. Finally, starting from a knot algebra and the link invariant related to it, one could study further the corresponding framization algebra \cite{jula5,chpa2} with the objective to construct new stronger link invariants.

\smallbreak
The paper is organized as follows. Section~\ref{yh} contains introductory material on the Yokonuma--Hecke algebras $\YH(q)$. We use the new quadratic relation introduced in \cite{chpa}, which is different from the one that we have been using so far and more computation-friendly, and we adjust all equations we need. 
Section~\ref{traces} is devoted to the three traces, the Ocneanu trace $\tau$, the trace  $\Jtr$ and the specialized trace  $\Jtrs$. The passage from $\Jtr$ to $\Jtrs$ is via the E--system, which is also presented in this section.
 We continue with Section~\ref{yhinvts}, where we define the invariants for framed and classical knots and links that we obtain from the Yokonuma--Hecke algebras; the formulae are again adapted to the new quadratic relation. Section~\ref{spclassical} is about Theorem~\ref{specialtrace}, whose proof requires some technical lemmas and some properties of the trace $\Jtrs$. The focus of Section~\ref{DeltaHomflypt} is Theorem~\ref{conjHomflypt}, which comprises one of the main results of this paper, that is, the topological equivalence of the invariants $\Theta_d$ and the Homflypt polynomial on knots. We then investigate in Section~\ref{slinks} the behaviour of the invariants $\Theta_d$ on arbitrary links. We start with simple cases of links, such as disjoint unions of knots and $2$-component links. Then, in Proposition~\ref{skein_ei}, we produce the special skein relation for the invariants $\Theta_d$ for classical links on crossings of different components. Next, we formulate the immediate consequences of this special skein relation, leading to one of the main results of Section~\ref{slinks}, Theorem~\ref{links}, which is the concrete expression of the invariants $\Theta_d$ with respect to the Homflypt polynomial. We further show that any invariant $\Theta_{d'}$ has a similar concrete expression with respect to any other invariant $\Theta_d$. As a consequence of the results of Section~\ref{slinks}, we obtain that the invariants  $\Theta_d$ can be defined diagrammatically via skein relations. 
Section~\ref{distinguish} contains the most important result of this paper, Theorem~\ref{ThetaNotP}, which states that the invariants $\Theta_d$ are  not topologically equivalent to the Homflypt polynomial. This is proved by providing computational data and an explicit diagrammatic proof via the special skein relation. In Section~\ref{distinguish}, we also show Theorem~\ref{2comp}, which is a derivation of Theorem~\ref{links}, about $2$-component Homflypt-equivalent links: a pair of  links on two components with the same Homflypt value is distinguished by an invariant $\Theta_d$ for some $d$ with $d \geq 2$ if and only if it is distinguished by all invariants $\Theta_d$ (excluding the Homflypt). 
We conclude Section~\ref{distinguish} by proving the invariance of $\Theta_d$ under mutation. 
In Section~\ref{threeVar} we generalize the invariants $\Theta_d$ to the $3$-variable skein invariant $\Theta$ (Theorem~\ref{invE}), using the algebra of braids and ties and adapting the results of \cite{AJ2}. We also present the closed formula of W.B.R. Lickorish for the invariant $\Theta$ and we discuss some consequences of this result.
Finally, in Section~\ref{new research directions} we discuss new research directions. 
Throughout the paper, we make remarks about the algebras $\YH(u)$ with the old quadratic relation, the trace $\Jtru$ defined on them and the invariants $\Delta_d$ (see Remarks~\ref{oldquadr}, \ref{oldquadr2}, \ref{oldquadr3}, \ref{oldquadr4}, \ref{oldquadr5}, \ref{invs_nosame}, \ref{E as a param u}, \ref{oldquadr6}, \ref{oldquadr7}, \ref{oldquadr8}, \ref{tiedinvs_nosame} and \ref{oldquadr9}).

\smallbreak
The last author acknowledges with pleasure inspiring discussions with Louis H. Kauffman about the importance of the skein approach to our invariants. We are also indebted to W.B.R. Lickorish for his contribution to this paper. Finally, it gives us pleasure to mention a list of places and meetings where we worked on these matters. These are: Valpara\'{i}so, Athens, Oberwolfach (MFO), Belgrade, Z\"{u}rich, Paris/Versailles, Lausanne, and Moscow.

\section{The Yokonuma--Hecke algebra}\label{yh}

In this section we recall the definition of the Yokonuma--Hecke algebra of type $A$ as a quotient of the (modular) framed braid group algebra and of its canonical basis.

\subsection{The framed braid group  and the modular framed braid group}
Let $n \in \N$ and let $B_n$ denote the classical braid group on $n$ strands. 
The {\it framed braid group}, ${\mathcal F}_n \cong \Z^n\rtimes B_n$, is the group defined by the standard {\it braiding} generators  $\s_1, \ldots, \s_{n-1}$ of $B_n$ 
together with the {\it framing}  generators  $t_1, \ldots, t_n$,
 subject to the relations:
\begin{equation}\label{framedgp}
\begin{array}{ccrclcl}
\mathrm{(b_1)}& & \s_i\s_j\s_i & = & \s_j\s_i\s_j && \mbox{for $ \vert i-j\vert = 1$}\\
\mathrm{(b_2)}& & \s_i\s_j & = & \s_j\s_i & & \mbox{for $\vert i-j\vert > 1$}\\
\mathrm{(f_1)}& & t_i t_j & =  &  t_j t_i &&  \mbox{for all $ i,j$}\\
\mathrm{(f_2)}& & t_j \s_i & = & \s_i t_{s_i(j)} && \mbox{for all $ i,j$}
\end{array}
\end{equation}
where $s_i$ denotes the transposition $(i,i+1)$.
Relations $\mathrm{(b_1)}$ and $\mathrm{(b_2)}$ are the usual braid relations, while relations $\mathrm{(f_1)}$ and $\mathrm{(f_2)}$ involve the framing generators. Further, for any natural number $d$, we can define the {\it modular framed braid group}, ${\mathcal F}_{d,n} \cong (\Z/d\Z)^n\rtimes B_n$, as the group generated by the elements $\s_1, \ldots, \s_{n-1}, t_1,\ldots,t_n$ satisfying relations (\ref{framedgp}) together with the extra ``modular'' relations: 
\begin{equation}\label{modular}
\begin{array}{ccrclcl}
& & t_j^d   & =  &  1 && \mbox{for all $j$.}
\end{array}
\end{equation}

Thanks to relations $\mathrm{(f_1)}$ and $\mathrm{(f_2)}$, every element $\alpha$ of  ${\mathcal F}_{n}$ (respectively ${\mathcal F}_{d,n}$) can be written in the form
$t_1^{k_1} \ldots t_n^{k_n} \s$, where $k_1,\ldots,k_n \in \Z$ (respectively  $k_1,\ldots,k_n \in \Z/d\Z$) and   $\s$ involves only the generators $\s_1,\ldots,\s_{n-1}$ (that is, $\s \in B_n$). This is called the {\it split form} of $\alpha$, with $t_1^{k_1} \ldots t_n^{k_n}$ being the {\it framing part} of $\alpha$ and $\s$ being the {\it braiding part} of $\alpha$. 
Diagrammatically, $\alpha$ can be pictured as the classical braid on $n$ strands corresponding to $\s$ with an integer (respectively an integer modulo $d$), the \emph{framing}, attached to each strand: $k_j$ is the integer attached to the $j$-th strand, for all $j=1,\ldots,n$.

For a fixed $d \in \N$, we define the following elements $e_i$ in the group algebra ${\C}{\mathcal F}_{d,n}$:
\begin{equation}\label{ei}
e_i := \frac{1}{d}\sum_{s=0}^{d-1}t^s_i t^{d-s}_{i+1} \qquad (1\leq i\leq n-1) \ .
\end{equation}
One can easily check that $e_i$ is an idempotent, \emph{i.e.,} $e_i^2=e_i$,  and that $e_i\s_i=\s_ie_i$ for all $i$.

\subsection{The Yokonuma--Hecke algebra}
Let $d\in \N$ and let $q \in \C \backslash \{0\}$  fixed. The {\it Yokonuma--Hecke algebra} (of type $A$), denoted by $\YH(q)$, is defined as the quotient of
$\C {\mathcal F}_{d,n}$ by
the ideal generated by the expressions:
$ \s_i^2 - 1 - (q - q^{-1})e_i \s_i$ for $1\leq i\leq n-1$.
We shall denote by $g_i$ the element in the algebra $\YH(q)$ corresponding to $\s_i$, while we keep the same notation for the $t_j$
(even though, in the following sections, we will sometimes identify ``algebra monomials'', that is, products of the generators $g_i$ and $t_j$, 
 with the corresponding framed braid words).
So, in $\YH(q)$ we have the following quadratic relations:
\begin{equation}\label{quadr}
g_i^2 = 1 + (q - q^{-1}) \, e_i \, g_i \qquad (1\leq i\leq n-1).
\end{equation}
The elements $g_i\in\YH(q)$ are invertible, with
\begin{equation}\label{invrs}
g_i^{-1} = g_i - (q - q^{-1}) \, e_i \qquad (1\leq i\leq n-1).
\end{equation}
Further, the elements $g_i \in \YH(q)$ satisfy the following relations: 

\begin{lemma}\label{gipower}
Let $i \in \{1,\ldots,n-1\}$ and let $r \in \Z$.
\begin{enumerate}[(a)]
\item If $r$ is odd, we have
$$g_i^r = (1-e_i)\,g_i + \left( \frac{q^r+q^{-r}}{q+q^{-1}} \right)e_ig_i +\left(  \frac{q^{r-1}-q^{-r+1}}{q+q^{-1}}\right) e_i \, .$$
\item If $r$ is even, we have
$$g_i^r = 1-e_i + \left( \frac{q^r-q^{-r}}{q+q^{-1}} \right)e_ig_i +\left(  \frac{q^{r-1}+q^{-r+1}}{q+q^{-1}}\right) e_i \, .$$
\end{enumerate}
\end{lemma}

\begin{proof}
For $r \in \Z_{>0}$, we prove the above formulas with the use of induction and the quadratic relation \eqref{quadr}  for the generators $g_i$.

Now, it is easy to check that the inverses of the elements $g_i$ satisfy the following quadratic relation:
$$(g_i^{-1})^2 = 1 + (\kappa - \kappa^{-1}) e_ig_i^{-1},$$
where $\kappa:=q^{-1}$. Replacing $g_i$ by $g_i^{-1}$ and $q$ by $\kappa$ in the above formulas, and then using \eqref{invrs}, yields the desired result for $r \in \Z_{<0}$.
\end{proof}

Yokonuma--Hecke algebras were originally introduced  by T.~Yokonuma \cite{yo}  in the context of finite Chevalley groups as natural generalizations of Iwahori--Hecke algebras. 
Note that, for $d=1$, the  algebra ${\rm Y}_{1,n}(q)$ coincides with the   Iwahori--Hecke algebra ${\rm H}_n(q)$ of type $A$. Indeed all framings are zero, so the corresponding elements of ${\mathcal F}_n$ are identified with elements in $B_n$; moreover we have $e_i=1$, so the quadratic relation~\eqref{quadr} becomes the well--known quadratic relation of the algebra ${\rm H}_n(q)$:
$$
g_i^2 = 1 + (q - q^{-1}) \, g_i \qquad (1\leq i\leq n-1).
$$

Following \cite[\S 3]{ju}, 
the algebra $\YH(q)$ has linear dimension $d^nn! $ and the set
$$\mathcal{B}_n^{\rm can}=\left\{t_1^{k_1}\ldots t_n^{k_n}(g_{i_1}\ldots g_{i_1-r_1})(g_{i_2}\ldots g_{i_2-r_2})\cdots (g_{i_p}\ldots g_{i_p-r_p})\,\left|\,\begin{array}{l} k_1,\ldots,k_n \in \Z /d \Z\\
1\leq i_1 <\cdots<i_p \leq n-1\end{array}\right\}\right.
$$ is a $\C$-linear basis for $\YH(q)$. This basis is called the \emph{canonical} basis of $\YH(u)$. Note that, in each element of the standard basis, the highest index generator $g_{n-1}$ appears at most once.

Now,  the natural inclusions ${\mathcal F}_n \subset {\mathcal F}_{n+1}$  give rise to the algebra inclusions $\C {\mathcal F}_n \subset \C {\mathcal F}_{n+1}$, which in turn induce the algebra inclusions $\YH(q) \subset {\rm Y}_{d,n+1}(q)$ for $n \in \N$ (setting $\C {\rm Y}_{d,0}(q): = \C$). We can construct an inductive basis $\mathcal{B}_n^{\rm ind}$ for $\YH(q)$ in the following way: we set
$B_0^{\rm ind}:=\{1\}$\ and $$\mathcal{B}_{n+1}^{\rm ind}:= \{w_n g_n g_{n-1} \ldots g_it_i^k,\,\, w_nt_{n+1}^k \,\,|\,\, 1 \leq i \leq n,\,\, k \in \Z/d\Z,\,\,w_n \in 
\mathcal{B}_{n}^{\rm ind} \},$$
for all $n \in \N$. 

\begin{rem}\label{oldquadr} \rm 
In the papers \cite{ju}, \cite{jula1}--\cite{jula5} and \cite{chla}, the Yokonuma--Hecke algebra is defined with a parameter $u$ instead of $q$. The algebra $\YH(u)$ is generated by the elements $\wt{g}_1, \ldots, \wt{g}_{n-1}$ and $t_1, \ldots, t_n$, satisfying relations \eqref{framedgp} (with $\wt{g}_i$ corresponding to $\sigma_i$), (\ref{modular}) and the quadratic relations:
\begin{equation}\label{THE oldquadr}
\wt{g_i}^2 = 1 + (u - 1) \, e_i + (u - 1) \, e_i \wt{g_i} \qquad (1\leq i\leq n-1).
\end{equation}
The new presentation of $\YH(q)$ used in this paper was obtained in \cite{chpa} by taking $u:=q^2$ and $g_i := \wt{g_i}+ (q^{-1}-1)\, e_i \wt{g_i}$ (or, equivalently, $\wt{g_i} := g_i + (q-1) \, e_i g_i$).

The above results on the bases have been proved with the old quadratic relations, but the proofs work exactly the same with the new ones. For an alternative proof using directly the new quadratic relations, the reader may refer to \cite[\S 4]{chpa2}, where bases are constructed for cylotomic Yokonuma--Hecke algebras, which include $\YH(q)$ as a particular case. 
\end{rem}

\section{Markov traces}\label{traces}

We will now discuss a Markov trace defined on the Yokonuma--Hecke algebra $\YH(q)$. This trace generalizes the Ocneanu trace defined on the Iwahori--Hecke algebra ${\rm H}_n(q)$. Further, we recall the conditions on the trace parameters for defining invariants for framed and classical links.

\subsection{The Ocneanu trace}
As stated in the previous section,  the Iwahori--Hecke algebra ${\rm H}_n(q)$ of type $A$ is isomorphic to  ${\rm Y}_{1,n}(q)$.  We will denote by $G_1,\ldots,G_{n-1}$ the braiding generators of  $ {\rm H}_n(q)$.
The natural inclusions $B_n  \subset B_{n+1}$ give rise to  the algebra inclusions ${\C}B_n \subset {\C}B_{n+1}$ (setting ${\C}B_0:={\C}$), which in turn induce the algebra inclusions
$ {\rm H}_n(q)  \, \subset {\rm H}_{n+1}(q)$ 
(setting ${\rm H}_0(q):=\C$), for $n \in  \N$. 
We then have the following result (cf.~\cite[Theorem 5.1]{jo}):

\begin{thm}
Let $\zet$ be an indeterminate over $\C$. There exists a unique linear Markov trace
$$
\tau :  \bigcup_{n \geq 0} {\rm H}_{n}(q) \longrightarrow  \C [\zet]
$$
defined inductively on ${\rm H}_n(q)$, for all $n \geq 0$, by the following rules:
$$
\begin{array}{rcll}
\tau(ab) & = &\tau(ba)  & \qquad a,b \in {\rm H}_{n}(q) \\
\tau(1) & = & 1 & \qquad 1 \in {\rm H}_{n}(q)  \\
\tau(a G_n) & = & \zet \, \tau(a)  & \qquad a\in {\rm H}_{n}(q) \quad (\text{Markov  property} ).
\end{array}
$$
\end{thm}

The trace $\tau$ is the \emph{Ocneanu trace} with parameter $\zet$. 
Using the natural $\C$-algebra epimorphism from $\C B_n$ onto ${\rm H}_{n}(q)$ given by $\sigma_i \mapsto G_i$, and abusing notation, we can define the trace $\tau$ on the elements of $\C B_n$, and thus, in particular, on the elements of $B_n$.
Then, diagrammatically, in the second rule, $1$ corresponds to the identity braid for any number of strands. The third rule is the Markov property of the trace, which corresponds to adding an extra strand with a positive crossing to a braid in $B_n$.
The trace $\tau$  was used by V.~F.~R.~Jones \cite{jo} for constructing the $2$-variable Jones or Homflypt polynomial for classical knots and links, see \S\ref{classic}.

In the rest of this paper, when we need to specify the values of $q$ and $\zet$, we will write $\tau(q,\zet)$ instead of simply $\tau$.

\subsection{The trace $\Jtr$}

An important property of the Yokonuma--Hecke algebra is that it also supports a Markov trace defined for all values of $n$. More precisely,  due to the inclusions $\YH(q) \subset {\rm Y}_{d,n+1}(q)$ and with the use of the inductive bases of the algebras $\YH(q)$, we obtain (cf.~\cite[Theorem 12]{ju}):

\begin{thm}\label{trace}
Let $z$, $x_1, \ldots, x_{d-1}$ be indeterminates over $\C$.
There exists a unique linear Markov trace
$$
\Jtr : \bigcup_{n \geq 0} \YH(q) \longrightarrow \C [z, x_1, \ldots, x_{d-1}]
$$
defined inductively on $\YH(q)$, for all $n \geq 0$, by the following rules:
$$
\begin{array}{crcll}
(1) & \Jtr(a b) & = & \Jtr(b a) & \qquad a, b\in\YH(q) \\
(2) & \Jtr(1) &  = & 1 & \qquad 1 \in\YH(q) \\
(3) & \Jtr(a g_n) & = & z\, \Jtr( a) & \qquad a\in\YH(q) \quad (\text{Markov  property} ) \\
(4) & \Jtr(a t_{n+1}^k) & = & x_k\, \Jtr( a) & \qquad a\in\YH(q) \quad (1\leq k\leq d-1).
\end{array}
$$
\end{thm}

Using the natural $\C$-algebra epimorphism from $\C \mathcal{F}_n$ onto $\YH(q)$ given by $\sigma_i \mapsto g_i$ and $t_j^k  \mapsto t_j^{k( {\rm mod}\, d)}$, and abusing notation, we can define the trace $\Jtr$ on the elements of $\C \mathcal{F}_n$, and thus, in particular, on the elements of $\mathcal{F}_n$.

Note that, for $d=1$, the trace ${\rm tr}_1$ is defined by only the first three rules. Thus,   ${\rm tr}_1$ coincides with the  Ocneanu trace $\tau$ on the Iwahori--Hecke algebra ${\rm H}_n(q) \cong {\rm Y}_{1,n}(q)$.

In the rest of this paper, when we need to specify the values of $q$ and $z$, we will write $\Jtr(q,\zet)$ instead of simply $\Jtr$.

\begin{rem} \rm \label{oldquadr2}
As mentioned in Remark \ref{oldquadr}, in \cite{ju}, J.~Juyumaya works on the Yokonuma--Hecke algebra $\YH(u)$, with $u=q^2$ and generators $\wt{g}_1, \ldots, \wt{g}_{n-1}$ satisfying the old quadratic relations \eqref{THE oldquadr}. He proves that
 there exists a unique linear Markov trace $\Jtru$ on 
$\bigcup_{n \geq 0} \YH(u)$ defined inductively by the four rules of Theorem \ref{trace}, where rule (3) is replaced by:
 $$
\begin{array}{crcll}
(\wt{3}) & \Jtru(a \wt{g}_n) & = & \wt{z}\, \Jtru( a) & \qquad a\in\YH(u) \quad (\text{\it Markov  property} ) 
\end{array}
$$
for some indeterminate $\wt{z}$ over $\C$. Since his proof uses the inductive bases of $\YH(u)$, it also works with the new quadratic relations, thus yielding Theorem \ref{trace}. For an alternative proof of Theorem \ref{trace} using directly the new quadratic relations, the reader may refer to \cite[\S 5]{chpa2}, where Markov traces are constructed on cylotomic Yokonuma--Hecke algebras, which include $\YH(q)$ as a particular case. 

Now, using the natural $\C$-algebra epimorphism from $\C \mathcal{F}_n$ onto $\YH(u)$ given by $\sigma_i \mapsto \wt{g}_i$ and $t_j^k  \mapsto t_j^{k( {\rm mod}\, d)}$, and abusing notation, we can define the trace $\Jtru$ on the elements of $\C \mathcal{F}_n$, and thus, in particular, on the elements of $\mathcal{F}_n$. 
\end{rem}

\subsection{The {\rm E}--system} 

Let ${\mathcal L}_f$ denote the set of oriented  framed links. For any braid $\a$, we will denote by $\widehat{\a}$ the link obtained as the closure of $\a$. By the Alexander Theorem, we have ${\mathcal L}_f = \cup_n\{\widehat{\a} \,|\, \alpha \in \mathcal{F}_n\}$. Further, by the Markov theorem, isotopy of framed links is generated by conjugation in $\mathcal{F}_n$ and by positive and negative stabilization and destabilization ($\a \sim \a \s_n^{\pm 1}$), for any $n$ (see, for example, \cite{KS}). In view of this and using the natural epimorphism of the framed braid group $\mathcal{F}_n$ onto $\YH(u)$, in \cite{jula2} the authors obtained  a topological invariant for framed knots and links after the method of V.~F.~R.~Jones \cite{jo}. 
Here we will do the same thing using the natural epimorphism of $\mathcal{F}_n$ onto $\YH(q)$.
This means that $\Jtr$  has to be normalized, so that the closures of the framed braids $\a$ and $\a \s_n$ $(\a\in {\mathcal F}_n)$ be assigned the same value of the invariant, and re-scaled, so that the closures of the framed braids $\a \s_n^{-1}$ and $\a \s_n$  $(\a\in {\mathcal F}_n)$ be  assigned the same value  of the invariant.  However, $\Jtr(\a \s_n^{-1})$ does not factor through $\Jtr(\a)$ as in the classical case, that is,
\begin{center}
$
\Jtr(\a \s_n^{-1}) 
\stackrel{(\ref{invrs})}{=}
\Jtr(\a \s_n) - (q - q^{-1})\, \Jtr(\a e_n) \neq \Jtr(\s_n^{-1}) \Jtr(\a).
$
\end{center}
The reason is that, although $\Jtr(\a \s_n) = z\, \Jtr(\a)= \Jtr(\s_n) \Jtr(\a)$, 
 $\Jtr(\a e_{n})$ does not always factor through $\Jtr(\a)$, that is,
\begin{center}
$\Jtr(\a e_{n}) \neq \Jtr(e_{n}) \Jtr(\a),$
\end{center}
which in turn, is due to the fact that,
\begin{center}
$\Jtr(\a t_n^k) \neq \Jtr(t_n^k) \Jtr(\a) \quad k=1,\ldots,d-1.$
\end{center}
Forcing the so-called \emph{${\rm E}$--condition}
\begin{equation}\label{aeneq}
\Jtr(\a e_{n}) = \Jtr(e_{n}) \Jtr(\a)
\end{equation}
yields that the trace parameters $x_1, \ldots, x_{d-1}$ have to satisfy the following non-linear system of equations in $\C$, called the ${\rm E}$--\emph{system}:
\begin{equation}\label{Esystem}
\sum_{s=0}^{d-1}{x}_{k+s}{x}_{d-s}  = {x}_k \,\sum_{s=0}^{d-1}{x}_{s}{x}_{d-s} \qquad (1\leq k \leq d-1)\ ,
\end{equation}
where the sub-indices on the ${x}_j$'s are regarded modulo $d$ and ${x}_0:=1$. 
Note that
$$
\frac{1}{d}\sum_{s=0}^{d-1}{x}_{s}{x}_{d-s} = \Jtr(e_{i}) \quad \text{for all }i.
$$

\begin{rem} \rm \label{oldquadr3}
The ${\rm E}$--condition and the ${\rm E}$--{system} as presented above are first defined and used in 
\cite{jula2} in order to re-scale $\Jtru$.
\end{rem}

As it was shown  by P.~G\'{e}rardin (in the Appendix of \cite{jula2}), the solutions of the {\rm E}--system are parametrized by the non-empty subsets of $\Z/d\Z$. For example, for every singleton subset $\{ m \}$ of $\Z/d\Z$, we have a solution of the ${\rm E}$--system given by:
\begin{equation}\label{Solution singleton}
{\rm x}_1 ={\rm exp}({{2 \pi m \sqrt{-1}}/{d}}) \qquad \text{ and } \qquad {\rm x}_k = {\rm x}_1^k  \,\,\,\,\, \text{ for }  k=2,\ldots, d-1.
\end{equation}
At the other extreme, the whole set $\Z/d\Z$ parametrizes the ``trivial'' solution of the ${\rm E}$--system, which is given by
\begin{equation}\label{trivial solution}
 {\rm x}_1={\rm x}_2 = \cdots={\rm x}_{d-1}=0.
\end{equation}

Finally, we note that the solutions of the ${\rm E}$--system can be interpreted as a generalization of the Ramanujan sum. Namely, by considering the subset $P$ of  $\mathbb{Z}/d{\mathbb{Z}}$ consisting  
of the numbers coprime to $d$, then the solution of the ${\rm E}$--system parametrized by $P$ is, up to the factor $\vert P\vert$,  the Ramanujan sum $c_d(k)$ \cite{ra}.

\subsection{The specialized trace $\Jtrs$}\label{strace} Let  $X_D := ({\rm x}_1, \ldots , {\rm x}_{d-1})$ be a solution of the {\rm E}--system parametrized by the non-empty subset $D$ of $\Z /d \Z$.

\begin{defn}\label{trdD} {\rm
We shall call {\it specialized trace} with parameter $z$, and denote by  $\Jtrs$,  the trace $\Jtr$ with the parameters $x_1,\ldots,x_{d-1}$ specialized to the complex numbers ${\rm x}_1, \ldots , {\rm x}_{d-1}$. More precisely,
$$
\Jtrs : \bigcup_{n \geq 0}\YH(q) \longrightarrow \C [z] 
$$
is a Markov trace defined inductively on $\YH(q)$, for all $n \geq 0$, by the following rules:
$$
\begin{array}{crcll}
(1) & \Jtrs(a b) & = & \Jtrs(b a) & \qquad a,b\in\YH(q) \\
(2) & \Jtrs(1) &  = & 1 & \qquad 1 \in\YH(q) \\
(3) & \Jtrs(a g_n) & = & z\, \Jtrs(a) & \qquad a\in\YH(q) \quad (\text{\it Markov  property} ) \\
(4^\prime) & \Jtrs (a \,t_{n+1}^k) & = & {\rm x}_k\,\Jtrs (a) & \qquad a\in\YH(q) \quad (1\leq k\leq d-1).
\end{array}
$$}
\end{defn}

The rules (1)--(3) are the same as in Theorem~\ref{trace}, while rule (4) is replaced by the rule ($4^\prime$). As it turns out \cite{jula3}:  
\begin{equation}\label{ED}
E_D := \Jtrs(e_{i}) = \frac{ 1}{\vert D\vert} \quad \text{ for all }i,
\end{equation}
where $\vert D\vert$ is the cardinality of the subset $D$. We also have (cf.~ \cite[Theorem 7 \& Lemma 8]{jula2}):
\begin{equation}\label{later for classical}
\Jtrs(a e_n)  =  E_D\, \Jtrs(a) \quad \text{and} \quad
\Jtrs(a e_ng_n)  =  z\, \Jtrs(a)  \quad \text{for all }a\in\YH(q).
\end{equation}

 Note that, for $d=1$, the specialized trace ${\rm tr}_{1, \{ 0\}}$  coincides with ${\rm tr}_1$, which in turn coincides with the Ocneanu trace $\tau$.

In the rest of this paper, when we need to specify the values of $q$ and $z$, we will write $\Jtrs(q,\zet)$ instead of simply $\Jtrs$.

\begin{rem}\rm
Following (\ref{Solution singleton}), we have
\begin{equation}
E_D=1 \Leftrightarrow |D|=1 \Leftrightarrow {\rm x}_1^d=1 \text{ and } {\rm x}_k= {\rm x}_1^k\,\,\, (1\leq k\leq d-1).
\end{equation}
On the other hand, following (\ref{trivial solution}), we have
\begin{equation}
E_D=\frac{1}{d} \Leftrightarrow D= \Z/d\Z \Leftrightarrow {\rm x}_k= 0\,\,\, (1\leq k\leq d-1).
\end{equation}
\end{rem}

\begin{rem} \rm  \label{oldquadr4}
In \cite[Definition 3]{chla} the specialized trace $\Jtrsu$ with parameter $\wt{z}$ is defined on 
$\bigcup_{n \geq 0}\YH(u)$, satisfying the analogous rules: ($1$), ($2$), ($\wt{3}$) and ($4^\prime$).
\end{rem}

\section{Framed and classical link invariants from the Yokonuma--Hecke algebras}\label{yhinvts}

Let $d \in \N$. Let  $D$ be a non-empty subset of $\Z/d\Z$ and let $X_D := ({\rm x}_1, \ldots , {\rm x}_{d-1})$ be the corresponding solution of the {\rm E}--system. With the use of the specialized  trace $\Jtrsu$, J.~Juyumaya and S.~Lambropoulou have defined invariants for various types of knots and links, such as framed, classical and singular \cite{jula2,jula3,jula4}.
We shall now recall briefly the definition of these invariants, in view also of the new presentation of $\YH(q)$ used in this paper.
We set
\begin{equation}\label{CapitalLambda}
\lambda_D := \frac{z - (q-q^{-1})E_D}{z} \quad \text{and} \quad \Lambda_D:=\frac{1}{z \sqrt{\lambda_D}}.
\end{equation} 
Recall that $E_D= 1/|D|$. We also set $\mathcal{R}_D:=\C[z^{\pm 1}, \sqrt{\lambda_D}^{\pm 1}]$.

\subsection{Framed links} Let ${\mathcal L}_f$ denote the set of oriented  framed links.
By the Alexander Theorem, we have ${\mathcal L}_f = \cup_n\{\widehat{\a} \,|\, \alpha \in \mathcal{F}_n\}$. Denote by $\g: \C {\mathcal F}_n \rightarrow \YH(q)$ the natural surjection defined by $\s_i \mapsto g_i$ and  $t_j \mapsto t_j$. 
Normalizing and re-scaling the specialized trace $\Jtrs$ (which is possible due to the E--condition) yields the following (cf.~\cite[Theorem 8]{jula2}):

\begin{thm}\label{invariant Phi}
For any framed braid $\a \in {\mathcal  F}_{n}$, we define
$$
\Phi_{d,D} (\widehat{\a}) := \Lambda_D^{n-1} (\sqrt{\lambda_D})^{\epsilon(\a)}\, (\Jtrs \circ \g) (\a)\ ,
$$
where $\epsilon(\a)$ is the sum of the exponents of the braiding generators $\s_i$ in the word  $\a$.
Then the map  
$$\Phi_{d,D}(q,z): {\mathcal L}_f \rightarrow \mathcal{R}_D, \,\,L \mapsto \Phi_{d,D}(L)$$  is a $2$-variable isotopy invariant of oriented framed links.
\end{thm}

In this paper we  write $\Jtrs(\a)$ instead of $\Jtrs(\g(\a))$, where $\a \in \mathcal{F}_n$, by abusing the notation.

\begin{rem}\label{2^d}
{\rm Note that, for every $d\in \mathbb{N}$, we have $2^d-1$ distinct solutions of the E--system, so the above construction yields $2^d -1$ distinct isotopy invariants for framed links.}\end{rem}

\begin{rem}\label{change to lambda}
\rm Using \eqref{CapitalLambda}, we can obtain defining equations for $z$ and $\Lambda_D$ with respect to $\lambda_D$, namely, 
\begin{equation}\label{eq change to lambda}
z:= \frac{(q-q^{-1})\,E_D}{1-\lambda_D}  \qquad \text{ and } \qquad \Lambda_D: =\frac{1}{z \sqrt{\lambda_D}}.
\end{equation}
Accordingly, we can use the notation $\Phi_{d,D}(q,\lambda_D)$ instead of $\Phi_{d,D}(q,z)$ . 
\end{rem}

\begin{rem}\rm \label{oldquadr5}
Using the natural surjection of $\C {\mathcal F}_n$ onto $\YH(u)$ defined by $\s_i \mapsto \wt{g_i}$ and $t_j \mapsto t_j$, and normalizing and re-scaling the specialized trace $\Jtrsu$, invariants $\Gamma_{d,D}(u,\wt{z})$ for oriented framed links are defined in \cite[Theorem 8]{jula2}. The proof of existence and invariance of $\Phi_{d,D}$ is standard and completely analogous to the proof for $\Gamma_{d,D}$. For an alternative proof of Theorem \ref{invariant Phi} using directly the new quadratic relations, the reader may refer to \cite[\S 6]{chpa2}, where invariants for framed links, and more generally framed links in the solid torus, are constructed with the use of cyclotomic Yokonuma--Hecke algebras.

In \cite[Proposition~7]{jula2} a skein relation is found for the invariant $\Gamma_{d,D}$, involving the braiding and the framing generators. It reads:
\begin{equation}\label{framed_skein_old} 
\frac{1}{\sqrt{\wt{\lambda}_D}} \Gamma_{d,D}(L_+) - \sqrt{\wt{\lambda}_D} \Gamma_{d,D}(L_-) = \frac{1-u^{-1}}{d} \sum_{s=0}^{d-1} \Gamma_{d,D}(L_s) + \frac{1-u^{-1}}{d \sqrt{\wt{\lambda}_D}} \sum_{s=0}^{d-1} \Gamma_{d,D}(L_{s\times})\,
\end{equation}
where 
$$\wt{\lambda}_D = \frac{\wt{z} - (u-1)E_D}{u\,\wt{z}}$$ 
and  the links $L_+$, $L_-$, $L_s$ and $L_{s\times}$ are illustrated in Figure~\ref{fig_framed_skein}.

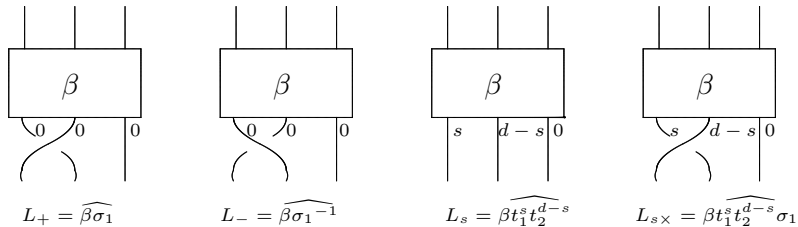
\begin{figure}[H]
\begin{center}
\begin{picture}(296,80)
\put(20,53){$\beta$}

\qbezier(6,70)(6,78)(6,86) 
\qbezier(25,70)(25,78)(25,86)
\qbezier(44,70)(44,78)(44,86)

\qbezier(5,20)(4,22)(5,24) 
\qbezier(25,20)(26,22)(25,24) 

\qbezier(0,44)(0,59)(0,70)
\qbezier(50,44)(50,59)(50,70)

\qbezier(0,70)(25,70)(50,70)
\qbezier(0,44)(25,44)(50,44)

\qbezier(5,24)(5,29)(15,34)
\qbezier(15,34)(25,39)(25,44)
\qbezier(25,24)(25,28)(20,31)
\qbezier(10,37)(5,40)(5,44)

\qbezier(44,20)(44,32)(44,44)

\put(100,53){$\beta$}

\qbezier(86,70)(86,78)(86,86) 
\qbezier(105,70)(105,78)(105,86)
\qbezier(124,70)(124,78)(124,86)

\qbezier(85,20)(84,22)(85,24) 
\qbezier(105,20)(106,22)(105,24) 

\qbezier(80,44)(80,59)(80,70)
\qbezier(130,44)(130,59)(130,70)

\qbezier(80,70)(105,70)(130,70)
\qbezier(80,44)(105,44)(130,44)

\qbezier(85,24)(85,28)(90,31)
\qbezier(100,37)(105,40)(105,44)
\qbezier(105,24)(105,29)(95,34)
\qbezier(95,34)(85,39)(85,44)

\qbezier(124,20)(124,32)(124,44)

\put(180,53){$\beta$}

\qbezier(166,70)(166,78)(166,86) 
\qbezier(185,70)(185,78)(185,86)
\qbezier(204,70)(204,78)(204,86)

\qbezier(160,44)(160,59)(160,70)
\qbezier(210,44)(210,59)(210,70)

\qbezier(160,70)(185,70)(210,70)
\qbezier(160,44)(215,44)(210,44)

\qbezier(166,20)(166,32)(166,44) 
\qbezier(185,20)(185,32)(185,44)
\qbezier(204,20)(204,32)(204,44)

\put(260,53){$\beta$}

\qbezier(246,70)(246,78)(246,86) 
\qbezier(265,70)(265,78)(265,86)
\qbezier(284,70)(284,78)(284,86)

\qbezier(245,20)(244,22)(245,24) 
\qbezier(265,20)(266,22)(265,24) 

\qbezier(240,44)(240,59)(240,70)
\qbezier(290,44)(290,59)(290,70)

\qbezier(240,70)(265,70)(290,70)
\qbezier(240,44)(265,44)(290,44)

\qbezier(245,24)(245,29)(255,34)
\qbezier(255,34)(265,39)(265,44)
\qbezier(265,24)(265,28)(260,31)
\qbezier(250,37)(245,40)(245,44)

\qbezier(284,20)(284,32)(284,44)

\put(10,37){\tiny{$0$}}
\put(25,37){\tiny{$0$}}
\put(46,37){\tiny{$0$}}

\put(90,37){\tiny{$0$}}
\put(105,37){\tiny{$0$}}
\put(125,37){\tiny{$0$}}

\put(168,37){\tiny{$s$}}
\put(185,37){\tiny{$d-s$}}
\put(206,37){\tiny{$0$}}

\put(250,37){\tiny{$s$}}
\put(265,37){\tiny{$d-s$}}
\put(286,37){\tiny{$0$}}

\put(5,5){\tiny{$L_{+}=\widehat{\beta \sigma_1}$}}
\put(80,5){\tiny{$L_{-}=\widehat{\beta {\sigma_1}^{-1}}$}}
\put(165,5){\tiny{$L_s=\widehat{\beta t_1^st_2^{d-s}}$}}
\put(237,5){\tiny{$L_{s\times}=\widehat{\beta t_1^st_2^{d-s}  \sigma_1}$}} 

\end{picture}
\caption{The framed links in the skein relation in open braid form.}
\label{fig_framed_skein}
\end{center}
\end{figure}
\end{rem}

Analogously, with the new quadratic relations, the invariant $\Phi_{d,D}$ satisfies the following skein relation:
\begin{equation}\label{framed_skein}
\frac{1}{\sqrt{\lambda_D}} \Phi_{d,D}(L_+) - \sqrt{\lambda_D} \Phi_{d,D}(L_-) = \frac{q-q^{-1}}{d} \sum_{s=0}^{d-1} \Phi_{d,D}(L_s) \,
\end{equation}
where the links $L_+$, $L_-$ and $L_s$ are illustrated in Figure~\ref{fig_framed_skein}.
The fact that the skein relation for $\Phi_{d,D}$ is very different from the skein relation for $\Gamma_{d,D}$ is 
an indication that the two invariants may not be topologically equivalent.

\subsection{Classical links}\label{classic}

Let  ${\mathcal L}$ denote the set of oriented classical links.  By the Alexander Theorem, we have ${\mathcal L} = \cup_{n} \{\widehat{\a} \,|\, \alpha \in B_n\}$. The classical braid group $B_n$ injects into the framed braid group ${\mathcal F}_n \cong \Z^n \rtimes B_n$, whereby elements of $B_n$ are viewed as framed braids with all framings equal to zero. So, by the classical Markov braid equivalence, comprising conjugation in the groups $B_n$ and positive and negative stabilizations and destabilizations, and by treating the $t_j$'s  as formal generators,
$\Phi_{d,D}(q,z)$ becomes an isotopy invariant of oriented classical links when restricted to ${\mathcal L}$ (see also \cite{jula3}).
This invariant of classical links
 will be denoted by $\Theta_{d,D}(q,z)$. 
Accordingly, for any classical braid $\a \in B_{n}$, we have
$$
\Theta_{d,D} (\widehat{\a}) := \Lambda_D^{n-1} (\sqrt{\lambda_D})^{\epsilon(\a)}\, (\Jtrs \circ \d) (\a)\ ,
$$
where  $\lambda_D$ and $\Lambda_D$ are given by \eqref{CapitalLambda},
$\epsilon(\a)$ is the sum of the exponents of the braiding generators $\s_i$ in the word  $\a$ and
$\d: \C B_n \rightarrow \YH(q)$ denotes the restriction of $\g$ to $\C B_n$. 
As in the case of framed links, we write $\Jtrs(\a)$ instead of $\Jtrs(\d(\a))$.

\begin{rem}
{\rm Following Remark \ref{2^d}, the above construction yields $2^d -1$ seemingly distinct isotopy invariants for classical links. 
However, we shall prove in Section~\ref{spclassical}  (Proposition~\ref{noinvts}) that, for classical links,  we only obtain one invariant for every $d \in \N$.}
\end{rem}

The invariants  $\Theta_{d,D}(q,z)$ need to be compared with known invariants of classical links, especially with the Homflypt polynomial. The Homflypt (or $2$-variable Jones) polynomial $P (q,\zet)$ is a $2$-variable isotopy invariant of oriented classical links that was constructed from the Iwahori--Hecke algebras ${\rm H}_n(q)$ and the Ocneanu trace $\tau$ after re--scaling and normalizing $\tau$ \cite{jo}. In this paper, we define $P (q,\zet)$ via the invariants $\Theta_{d,D}(q,z)$, since, for $d=1$, the algebras  ${\rm H}_n(q)$ and ${\rm Y}_{1,n}(q)$ coincide, while the traces $\tau$, ${\rm tr}_1$ and ${\rm tr}_{1, \{ 0\}}$ also coincide. For any classical braid $\a \in B_n$, we define
$$
P(\widehat{\a}): = \Theta_{1,\{0\}}(\widehat{\a}) = \left( \frac{1}{z \sqrt{\lambda_{\rm H}}} \right)^{n-1} (\sqrt{\lambda_{\rm H}})^{\epsilon(\a)}\, ({\rm tr}_{1, \{ 0\}} \circ \d)\,(\a) \ ,
$$
where $$\lambda_{\rm H} := \frac{z-(q-q^{-1})}{z} = \lambda_{\{0\}},$$
$\epsilon(\a)$ is the sum of the exponents of the braiding generators $\s_i$ in the word  $\a$ and $\d$ is the surjection $\C B_n \rightarrow {\rm Y}_{1,n}(q) \cong {\rm H}_n(q)$.
 Further, the Homflypt polynomial satisfies the following skein relation \cite{jo}:
\begin{equation}\label{skein of hom}
\frac{1}{\sqrt{\lambda_{\rm H}}} \, P(L_+) - \sqrt{\lambda_{\rm H}} \, P(L_-) = (q-q^{-1}) \, P(L_0)
\end{equation}
where $L_+, L_-, L_0$ is a Conway triple.

Contrary to the case of framed links, the skein relation of the invariant $\Phi_{d,D}(q,z)$ has no topological interpretation in the case of classical links, since they introduce framings.  This makes it very difficult to compare the invariants $\Theta_{d,D} (q,z)$ with the Homflypt polynomial using diagrammatic methods.  
On the algebraic level, there are no algebra homomorphisms connecting the algebras and the traces, see \cite{chla}.  Further, for generic values of the parameters $q,z$ the invariants $\Theta_{d,D}(q,z)$ do not coincide with the Homflypt polynomial; they only coincide in  the trivial cases where $q= \pm 1$ or when $E_D=1$ (cf.~\cite[Theorem 5]{chla}).

\begin{rem}\label{invs_nosame} \rm 
Similarly to $\Phi_{d,D}(q,z)$, the invariant $\Gamma_{d,D}(u,\wt{z})$ becomes an isotopy invariant of oriented classical links when restricted to ${\mathcal L}$. This invariant  of classical links is denoted by $\Delta_{d,D}(u,\wt{z})$ and it is the one studied in \cite{jula3} and \cite{chla}.
Again, there is no reason that the invariants $\Delta_{d,D}(u,\wt{z})$ and $\Theta_{d,D}(q,z)$ are topologically equivalent. In fact, as we shall see in Section~\ref{threeVar}, it seems that they are not. Note though that, for $d=1$, $\Delta_{1,\{0\}}(u,\wt{z}) = \Theta_{1,\{0\}}(q,z)$ if we take $u=q^2$ and $\wt{z}=qz$ (this can be easily seen by writing down the skein relation for $\Delta_{1,\{0\}}$
and for $\Theta_{1,\{0\}}$), so both families of invariants include the Homflypt polynomial as a special case.

Everything that we have said above concerning the comparison of $\Theta_{d,D}(q,z)$ and the Homflypt polynomial also applies
to $\Delta_{d,D}(u,\wt{z})$,
Further, as shown in \cite{jula3}, in $\YH(u)$ a ``closed'' cubic relation is satisfied, which  involves only the braiding generators, and which factors through the quadratic relation of the Iwahori--Hecke algebra ${\rm H}_n(u) \cong {\rm Y}_{1,n}(u)$. Namely, 
$\wt{g_i}^3 - u \wt{g_i}^2 - \wt{g_i} + u = (\wt{g_i} - 1) \big(\wt{g_i}^2 - (u - 1) \wt{g_i} - u\big)=0$. However, the skein relation for classical links coming from the cubic relation is not sufficient for determining the invariants $\Delta_{d,D}(u,\wt{z})$ diagrammatically with a simple set of initial conditions. Moreover, the above factoring does not give information toward the comparison of the invariants either.
\end{rem}

\section{The specialized trace $\Jtrs$ for classical links}\label{spclassical}

In the rest of the paper we restrict our attention to classical knots and links. In this section we show that in this case, and assuming the E--condition, the specialized trace $\Jtrs$, which gives rise to the invariants $\Theta_{d,D}$, can be computed only via rules involving the traces of the elements $g_i$ and $e_i$.

\subsection{A new set of defining rules for $\Jtrs$ and its consequences} 
Assume that we are given an oriented classical link in the form of the closure of a braid $\alpha \in B_{n}$. 
Let $\d:  \C B_n \rightarrow  \YH(q)$ be the natural $\C$-algebra homomorphism given by $\sigma_i \mapsto g_i$.
Then $\d(\a)$ involves only the braiding generators $g_1,\ldots,g_{n-1}$. In fact, 
we have the following:

\begin{prop} The image of $\d$ is the 
  subalgebra $\YH(q)^{\rm (br)}$ of $\YH(q)$ generated by $g_1,\ldots,g_{n-1}$.
\end{prop}

\begin{proof} The equality $\d(\C B_n)=\YH(q)^{\rm (br)}$ 
derives 
from the fact that $$\delta(\sigma_{j_1}^{k_1}\sigma_{j_2}^{k_2}\ldots \sigma_{j_r}^{k_r}) =g_{j_1}^{k_1}g_{j_2}^{k_2}\ldots g_{j_r}^{k_r}
\in \YH(q)^{\rm (br)}$$ for all $j_1,j_2,\ldots,j_r \in \{1,\ldots,n-1\}$ and $k_1,k_2,\ldots,k_r \in \Z$. For this, it is enough to show that $g_i^{-1} \in \YH(q)^{\rm (br)}$ for all $i=1,\ldots,n-1$. Indeed, we have
\begin{equation}\label{ties1}
(q-q^{-1}) \,e_i g_i = g_i^{2} -1 \in  \YH(q)^{\rm (br)},
\end{equation}
which yields that
\begin{equation}\label{ties2}
 (q-q^{-1}) \,e_i = g_i^{3} -g_i -  (q-q^{-1})^2 e_i g_i \in  \YH(q)^{\rm (br)},
 \end{equation}
and so,
$$g_i^{-1} = g_i - (q-q^{-1}) e_i  \in  \YH(q)^{\rm (br)},$$
as desired.
\end{proof}

\begin{rem}\label{conj braids and ties}
 \rm If $q \neq \pm 1$, then combining \eqref{ties1} and \eqref{ties2} yields
\begin{equation}
e_i =  \frac{1}{q-q^{-1}}\left(g_i^{3} -g_i\right) -  \left(g_i^2-1\right)
\end{equation}
for all $i=1,\ldots,n-1$. So, in this case, $\YH(q)^{\rm (br)}$ coincides with the subalgebra of $\YH(q)$ generated by the elements $g_1,\ldots,g_{n-1}, e_1,\ldots,e_{n-1}$, which in turn should be isomorphic, when $d \geq n$, to the \emph{algebra of braids and ties} $\mathcal{E}_n(q)$ studied in \cite{Ry, AJ}\footnote{A bit after the completion of this paper, J.~Espinoza and S.~Ryom-Hansen proved that the natural homomorphism from $\mathcal{E}_n(q)$ to $\YH(q)$ is injective for $d \geq n$ \cite[Theorem 8]{EsRy}. Since the image of this monomorphism is $\YH(q)^{\rm (br)}$, their result implies that $\mathcal{E}_n(q)$ is isomorphic to $\YH(q)^{\rm (br)}$ for $d \geq n$.}.
\end{rem}

Now, during the calculation of the specialized trace $\Jtrs$ on $\a \in B_n$, the framing generators appear only when the quadratic relation~\eqref{quadr} and the inverse relation~\eqref{invrs} are applied, and then only in the form of the idempotents $e_i$. So, it would make sense in this setting to substitute rule
$(4')$ of Definition~\ref{trdD} by rules involving only the $e_i$'s, such as \eqref{later for classical}.

Indeed, our aim in this section will be to prove the following result:

\begin{thm}\label{specialtrace}
Let $m \in \{1,\ldots,d\}$ and set $E_m:=1/m$. 
Let $z$ be an indeterminate over $\C$.
There exists a unique linear Markov trace
$$
\Jtrm : \bigcup_{n \geq 0} \YH(q)^{\rm (br)} \longrightarrow \C [z]
$$
defined inductively on $\YH(q)^{\rm (br)}$, for all $n \geq 0$, by the following rules:
$$
\begin{array}{crcll}
{\rm (i)}  & \Jtrm(a b) & = & \Jtrm(b a) & \qquad a, b\in\YH(q)^{\rm (br)} \\
{\rm (ii)}  & \Jtrm(1) &  = & 1 & \qquad 1 \in\YH(q)^{\rm (br)} \\
{\rm (iii)}  & \Jtrm(a g_n) & = & z\, \Jtrm( a) & \qquad a\in\YH(q)^{\rm (br)} \quad (\text{Markov  property} ) \\
{\rm (iv)}  & \Jtrm(a e_n) & = & E_m\, \Jtrm( a) & \qquad a\in\YH(q)^{\rm (br)}\\
{\rm (v)}  & \Jtrm(a e_ng_n) & = & z\, \Jtrm( a) & \qquad a\in\YH(q)^{\rm (br)} .
\end{array}
$$
For all $a \in  \bigcup_{n \geq 0} \YH(q)^{\rm (br)}$, we  have $\Jtrm(a) =\Jtrs(a)$ where $D$ is any subset of $\Z/d\Z$ such that $|D|=m$. Note that, in this case, $E_m = E_D$.
\end{thm}

For $\alpha \in B_n$, we set $\Jtrm(\a):=\Jtrm(\d(\a))$. Theorem \ref{specialtrace} implies that the specialized trace $\Jtrs$ on classical knots and links depends only on $|D|$ and not on the solution $X_D$ of the E--system. It does not even depend on $d$, since we have the following:

\begin{cor} \label{noinvts1} Let $d,d'$ be positive integers with $d \leq d'$.
For all $\a \in B_n$, we have $${\rm tr}_{d',d}(\alpha) = {\rm tr}_{d,d}(\alpha) = {\rm tr}_{d,\,\Z/d\Z}(\alpha).$$ 
\end{cor}

\begin{rem}\rm
Recall that $\Z/d\Z$ parametrizes the trivial solution of the E--system, that is, the one given by ${\rm x}_1 = {\rm x}_2=\cdots ={\rm x}_{d-1}=0$.
\end{rem}

The above way of calculating the trace $\Jtrs$ allows us to prove new results for the invariants of classical knots and links $\Theta_{d,D}$ and to compare them with the Homflypt polynomial.   Here is the first observation. The construction of the invariants $\Theta_{d,D}$ yielded seemingly $2^d-1$ invariants for every choice $d$. However, an immediate consequence of Theorem \ref{specialtrace} and Corollary \ref{noinvts1}  is the following:

\begin{prop} \label{noinvts} 
The values of the isotopy invariants $\Theta_{d,D}$ for classical links depend only on the cardinality $\vert D \vert$ of $D$. 
Hence, 
for a fixed $d$, we only obtain $d$ invariants. Further, for $d,d'$ positive integers with $d \leq d'$, we have $\Theta_{d,D} = \Theta_{d',D'}$ as long as $|D|=|D'|$.
We deduce that, if $|D'|=d$, then  $\Theta_{d',D'} = \Theta_{d,\,\Z/d\Z}$. 
Therefore, the invariants $\Theta_{d,D}$ can be parametrized by the natural numbers, setting
$\Theta_d := \Theta_{d,\,\Z/d\Z}$ for all $d \in \Z_{>0}$.
\end{prop}

\begin{rem} \rm
Note that, in general, $\Jtrs$ is sensitive to the choice of the set $D$, since for two sets $D_1$ and $D_2$ of the same cardinality, we get different solutions of the E--system ($X_{D_1} \neq X_{D_2}$). Theorem~\ref{specialtrace} implies that the trace $\Jtrs$ on $\YH(q)^{\rm (br)}$ depends only on the cardinality $\vert D \vert$, since rules ${\rm (iv)}$ and ${\rm (v)}$ substitute completely rule $(4')$ of Definition~\ref{trdD}. However,  this may not be the case on $\YH(q)$ (\emph{i.e.,} for framed links), since rules ${\rm (iv)}$ and ${\rm (v)}$ are not sufficient to calculate $\Jtrs$ on any word.
\end{rem}

\begin{rem} \rm 
Note that there is no analogue of Theorem~\ref{specialtrace} for the trace $\Jtr$, since condition \eqref{aeneq} is not satisfied by $\Jtr$.
\end{rem}

\subsection{Some useful relations in $\YH(q)$}
For the proof of Theorem~\ref{specialtrace}, we will need some lemmas. Let $i,k \in \{1,\ldots,n\}$. We consider the elements $e_{i,k} \in \YH(q)$ defined as follows:
$$
e_{i,k} = \frac{1}{d} \sum_{s=0}^{d-1} t_i^s t_k^{d-s}.
$$
The elements $e_{i,k}$ are idempotents and clearly, $e_{i,i} = 1$ and $e_{i,i+1} = e_i$. Moreover, $e_{i,k} = e_{k,i}$, so we may assume from now on that $i < k$. Furthermore, the following relations hold in $\YH(q)$, as a result of defining relations $\mathrm{(f_1)}$ and $\mathrm{(f_2)}$:
\begin{equation}\label{id_1}
e_{i,k} e_{i',k'}  = e_{i',k'} e_{i,k}
\end{equation}
and 
\begin{equation}\label{id_2}
e_{i,k}\, g_j = g_j  \,e_{s_j(i),s_j(k)}
\end{equation}
The following two lemmas are direct consequences of the above relations.

\begin{lemma} \label{ei_conj}
The following relations hold in $\YH(q)$:
\begin{center}
$
e_{i-1} = g_i g_{i-1} e_i (g_i g_{i-1})^{-1} \quad \quad \mbox{for all \quad $ 1 < i < n$.} 
$
\end{center}
In particular, $e_i$ is conjugate to $e_k$ for all $1 \leq i < k \leq n-1$. 
\end{lemma}

\begin{lemma} \label{eik_conj} Let $1 \leq i < j<k \leq n$. The following relations hold in $\YH(q)$:
\begin{equation}\label{eik_conj_ei}
e_{i,k} = (g_{k-1} \ldots g_{j})\, e_{i,j} \,(g_{k-1} \ldots g_{j})^{-1} , 
\end{equation}
and 
\begin{equation}\label{eik_conj_ek}
e_{i,k} = (g_i \ldots g_{j-1}) \,e_{j,k}\, (g_i \ldots g_{j-1})^{-1}  .
\end{equation}
In particular, the element $e_{i,k}$ is conjugate to $e_{i,j}$ and $e_{j,k}$. 
\end{lemma}

For $j=i+1$ and $k>j$, Lemma \ref{eik_conj} yields that the element $e_{i,k}$ is conjugate to $e_i$. 
Following Lemma \ref{ei_conj}, we deduce that $e_{i,k}$ is conjugate to $e_j$, for all $i,k,j \in \{1,\ldots,n\}$.
Rule (1) of the definition of $\Jtrs$ implies the following:

\begin{cor} 
 Given a solution $X_D$ of the {\rm E}--system, we have
$$\Jtrs (e_{i,k}) = \Jtrs (e_i) = E_D = \frac{1}{|D|}$$
for all $1 \leq i < k \leq n$.
\end{cor}

In the next lemma a relation in $\YH(q)$ involving the elements $e_{i,k}$ is stated. This relation helps us to prove some results about $\Jtrs$.

\begin{lemma}\label{many_ei}
Let $p \in \N$ and let $1 \leq i_1 < i_2 < \ldots < i_p < k \leq n$. The following relation holds in $\YH(q)$:
\begin{equation}\label{many_ei_eq}
e_{i_1,k} e_{i_2,k} \ldots e_{i_p,k} = e_{i_1,i_2} e_{i_2,i_3} \ldots e_{i_{p-1},i_p} e_{i_p,k}.
\end{equation}
\end{lemma}

\begin{proof} One can easily check that \cite[Lemma 1]{chla}:
\begin{equation}\label{good to know}
t_k e_{i,k} = t_i e_{i,k}  \quad \quad \mbox{for all \quad $ 1  \leq  i < k \leq n$.} 
\end{equation}
Equation \eqref{many_ei_eq} is a direct consequence of \eqref{good to know}.
\end{proof}

\begin{cor}
For $1 \leq i \leq k  \leq n$, we have
$$
e_{i,k+1} e_{i+1,k+1} \ldots e_{k-1,k+1}e_k  = e_i e_{i+1}\ldots e_k.
$$
\end{cor}

Let $a \in \YH(q)$. By \eqref{later for classical}, we have 
$$\Jtrs(a e_n)  =  E_D\, \Jtrs(a) \quad \text{and} \quad
\Jtrs(a e_ng_n)  =  z\, \Jtrs(a).$$
 This result can be extended to the elements $e_{i,n+1}$, with $1 \leq i \leq n$, as follows:

\begin{lemma} \label{trace_one_ei}
Let $a \in \YH(q)$ and let $X_D$ be a solution of the {\rm E}--system. Then we have
$$\Jtrs(a \, e_{i,n+1})  =  E_D\, \Jtrs(a) \quad \text{and} \quad
\Jtrs(a \, e_{i,n+1}g_n)  =  z\, \Jtrs(a \, e_{i,n})$$
for all $1\leq i \leq n$ (recall that $e_{n,n}=1$).
\end{lemma}

\begin{proof}
By Lemma~\ref{eik_conj} and rule (1) of the definition of the $\Jtrs$, we obtain
\begin{center}
$\Jtrs (a \, e_{i,n+1}) \stackrel{\eqref{eik_conj_ek}}{=} \Jtrs (a (g_i \ldots g_{n-1}) e_n (g_i \ldots g_{n-1})^{-1}) = \Jtrs ((g_i \ldots g_{n-1})^{-1} a (g_i \ldots g_{n-1}) e_n).$
\end{center}
Since $(g_i \ldots g_{n-1})^{-1} a (g_i \ldots g_{n-1}) \in \YH(q)$, by \eqref{later for classical}, we have
\begin{eqnarray*}
\Jtrs (a \, e_{i,n+1}) &=& \Jtrs ((g_i \ldots g_{n-1})^{-1} a (g_i \ldots g_{n-1}) e_n)
\\
&=&E_D \, \Jtrs ((g_i \ldots g_{n-1})^{-1} a (g_i \ldots g_{n-1}))
\\
&=&E_D \, \Jtrs ((g_i \ldots g_{n-1}) (g_i \ldots g_{n-1})^{-1} a )
\\
&=& E_D \, \Jtrs (a).
\end{eqnarray*}

Now, for $i=n$, we have $\Jtrs(a e_ng_n)  =z\, \Jtrs(a)=  z\, \Jtrs(ae_{n,n})$. For $i<n$, by 
\eqref{id_2} and  rules (1) and (3) of the definition of the $\Jtrs$, we obtain
\begin{center}
$\Jtrs(a \, e_{i,n+1}g_n)  = \Jtrs(a \, g_n e_{i,n}) = z\, \Jtrs(a \,e_{i,n})$.
\end{center}

\end{proof}

A similar result for words of the form $a e_{i_1,n+1} \ldots e_{i_p,n+1}$, with $1\leq i_1 < \ldots <i_p \leq n$, can be proved. Namely, we have:

\begin{lemma} \label{trace_many_ei}
Let $a \in \YH(q)$, and and let $X_D$ be a solution of the {\rm E}--system. 
Let $p \in \N$   and let $1\leq i_1 < \ldots <i_p \leq n$.
Then we have
\begin{equation} \label{trace_many_ei_eq1}
\Jtrs (a \, e_{i_1,n+1} \ldots e_{i_p,n+1}) = E_D \, \Jtrs (a \, e_{i_1,i_2} \ldots e_{i_{p-1},i_p}) 
\end{equation}
and
\begin{equation} \label{trace_many_ei_eq2}
\Jtrs (a \, e_{i_1,n+1} \ldots e_{i_p,n+1} g_n) = z \, \Jtrs (a \, e_{i_1,i_2} \ldots e_{i_{p-1},i_p} e_{i_p,n}) .
\end{equation}
\end{lemma}

\begin{proof}
First note that, by Lemma \ref{many_ei}, we have
$$ e_{i_1,n+1} \ldots e_{i_p,n+1} = e_{i_1,i_2} \ldots e_{i_{p-1},i_p}e_{i_{p},n+1}.$$
Then, by Lemma \ref{trace_one_ei}, we obtain
$$\Jtrs (a \, e_{i_1,n+1} \ldots e_{i_p,n+1}) =  \Jtrs (a \, e_{i_1,i_2} \ldots e_{i_{p-1},i_p}e_{i_{p},n+1})=E_D \, \Jtrs (a \, e_{i_1,i_2} \ldots e_{i_{p-1},i_p})$$ 
and
$$\Jtrs (a e_{i_1,n+1} \ldots e_{i_p,n+1}g_n) =  \Jtrs (a \, e_{i_1,i_2} \ldots e_{i_{p-1},i_p}e_{i_{p},n+1}g_n)=z \, \Jtrs (a \, e_{i_1,i_2} \ldots e_{i_{p-1},i_p}e_{i_p,n}).$$ 
\end{proof}

\begin{rem}\label{all good}
{\rm Note that, in the proofs of Lemmas \ref{trace_one_ei} and \ref{trace_many_ei}, we never used rule $(4')$ of Definition~\ref{trdD}. We just used the equalities given by (\ref{later for classical}). Since these equalities correspond to rules (iv) and (v) of the definition of $\Jtrm$, we deduce that the properties of $\Jtrs$ described in these lemmas are also valid for $\Jtrm$.}
\end{rem}
	
\subsection{Proof of Theorem~\ref{specialtrace}} 
We shall now prove the main result of this section, Theorem~\ref{specialtrace}. Let $m \in \{1,\ldots,d\}$ and set $E_m:=1/m$. Let $D$ be a subset of $\Z/d\Z$ such that $|D|=m$.
We will show that the linear map $\Jtrm$ is uniquely defined on $\YH(q)^{\rm (br)} = \delta(\C B_n)$ by rules (i)--(v) and that we have
$\Jtrm=\Jtrs$. We will proceed by induction on $n$.

For $n=1$, the statement is obvious. Now assume that the statement holds on $\YH(q)^{\rm (br)}$ and let $a \in  {\rm Y}_{d,n+1}(q)^{\rm (br)}$. We will show that
$\Jtrm(a)$ is uniquely defined by  rules (i)--(v) and that $\Jtrm(a)=\Jtrs(a)$. Since $\Jtrm$ is a linear map, it is enough to prove the statement when 
$a$ is a product of the braiding generators $g_1,\ldots,g_{n}$.

In order to calculate $\Jtrm(a)$, we apply an algorithm similar to the one described in \cite{chjakala} for the calculation of $\Jtrs$ on any element of the Yokonuma--Hecke algebra.
First, all negative exponents are eliminated through the inverse relation \eqref{invrs}. This produces a $\C$-linear combination of new elements 
in ${\rm Y}_{d,n+1}(q)^{\rm (br)}$ which are products of the generators $g_1,\ldots,g_{n}$ with only positive exponents and the idempotents $e_1,\ldots,e_{n}$.
Then all exponents greater than one in these elements are eliminated with the use of the quadratic relation \eqref{quadr}. The result is a $\C$-linear combination of elements which are products of the generators $g_1,\ldots,g_{n-1}$ with all exponents equal to $1$ and the idempotents $e_1,\ldots,e_{n-1}$.
Thanks to \eqref{id_1} and \eqref{id_2}, each such element can be brought to its \emph{adapted split form}
\begin{equation}\label{adapted}
e_{i_1,k_1}e_{i_2,k_2} \ldots e_{i_p,k_p} g_{j_1}g_{j_2} \ldots g_{j_r} \ ,
\end{equation}
with $i_1,i_2,\ldots,i_p   \in \{1, \ldots,n\}$, $2 \leq k_1 \leq k_2 \leq \cdots \leq k_p \leq n+1$,   $i_l < i_{l+1}$ whenever $k_l =k_{l+1}$ and $j_1,j_2,\ldots,j_r \in \{1,\ldots,n\}$.
If more squares of the braiding generators appear during this process, we apply again the quadratic relation. We repeat the above procedure until $a$ is expressed as a   $\C$-linear combination of elements which are in their  adapted split form and do not contain any braiding generator with exponent other than $1$.

Next, by applying the braid relations $(\mathrm{b}_1)$ and $(\mathrm{b}_2)$, and breaking any squares that appear in the process with the use of the quadratic relation, 
$a$ is eventually expressed as a $\C$-linear combination of elements in ${\rm Y}_{d,n+1}(q)^{\rm (br)}$ of the form $\eqref{adapted}$, where
$j_l \neq j_{l+1}$ and $g_{{\rm max}\{j_1,\dots,j_r\}}$ appears only once.

We conclude that $a$ can be expressed as  a $\C$-linear combination of elements in ${\rm Y}_{d,n+1}(q)^{\rm (br)}$ with each such element $y$ being in its adapted form
\begin{equation}\label{adapted plus}
y=e_{i_1,k_1}e_{i_2,k_2} \ldots e_{i_p,k_p} g_{j_1}g_{j_2} \ldots g_{j_r} \ ,
\end{equation}
where 
$i_1,i_2,\ldots,i_p   \in \{1, \ldots,n\}$, $2 \leq k_1 \leq k_2 \leq \cdots \leq k_p \leq n+1$,  $i_l < i_{l+1}$ whenever $k_l =k_{l+1}$,  
$j_1,j_2,\ldots,j_r \in \{1,\ldots,n\}$,
$j_l \neq j_{l+1}$ and $g_{{\rm max}\{j_1,\dots,j_r\}}$ appears only once. 
Therefore, in order to compute $\Jtrm(a)$, it is enough to know the value of the trace $\Jtrm$ on all elements $y$
as in \eqref{adapted plus}.

Let $y$ be as above.
Let $s$ be the unique element in $\{1,\ldots,r\}$ with the property $j_s={\rm max}\{j_1,\dots,j_r\}$
and let $\ell$ be the unique element in $\{1,\ldots,p\}$ with the property $k_\ell = k_p$ and $k_{\ell-1}<k_p$.
If $k_p < n+1$ and $j_s < n$, then we are covered by the induction hypothesis.
If not, then we distinguish three cases:
\begin{itemize}
\item $k_p <n+1$ and $j_s=n$. Set $y':=g_{j_{s+1}} \ldots g_{j_r}e_{i_1,k_1}e_{i_2,k_2} \ldots e_{i_p,k_p} g_{j_1}g_{j_2}\ldots g_{j_{s-1}} \in \YH(q)^{\rm (br)}$.
Combining rules (i) and (iii) yields:  
\begin{center}
$\Jtrm(y)= \Jtrm (y'g_n)= z \, \Jtrm (y')$
\end{center}
The induction hypothesis  implies that $\Jtrm (y')$ is uniquely defined by rules (i)--(v). Further, we have
$\Jtrm(y')=\Jtrs(y')$, and so, 
\begin{center}
$\Jtrm(y)=z \, \Jtrm(y') = z\,\Jtrs(y')=\Jtrs(y)$.
\end{center}
 \smallbreak

\item $k_p = n+1$ and $j_s<n$.  
Set $y':=g_{j_1}g_{j_2} \ldots g_{j_r}e_{i_1,k_1}e_{i_2,k_2} \ldots e_{i_{\ell-1},k_{\ell-1}} \in \YH(q)^{\rm (br)}$.
Following Remark \ref{all good}, we have, by rule (i) and Lemma \ref{trace_many_ei}, that
\begin{center}
$\Jtrm(y)=\Jtrm(y'e_{i_\ell,n+1} \ldots e_{i_p,n+1}) = E_m\, \Jtrm(y'e_{i_\ell,i_{\ell+1}} \ldots e_{i_{p-1},i_p})$,
\end{center}
since we have $i_\ell < i_{\ell+1} <\cdots < i_p$. Moreover, $i_p < k_p = n+1$, and so $y'':=y'e_{i_\ell,i_{\ell+1}} \ldots e_{i_{p-1},i_p}$ belongs to $\YH(q)^{\rm (br)}$.
By the induction hypothesis, $\Jtrm (y'')$ is uniquely defined by rules (i)--(v). Further, we have 
$\Jtrm(y'')=\Jtrs(y'')$, and so,  by  Lemma \ref{trace_many_ei}, 
\begin{center}
$\Jtrm(y)=E_m \, \Jtrm(y'') = E_D\,\Jtrs(y'')=\Jtrs(y)$.
\end{center}
\smallbreak

\item $k_p = n+1$ and $j_s=n$.  Following \eqref{id_2}, there exist $i_\ell',i_{\ell+1}',\ldots,i_p' \in \{1,\ldots,n\}$ such that 
$$e_{i_\ell,n+1} \ldots e_{i_{p},n+1} g_{j_1}g_{j_2}\ldots g_{j_{s-1}} = g_{j_1}g_{j_2}\ldots g_{j_{s-1}}e_{i_\ell',n+1} \ldots e_{i_{p}',n+1}.$$
Due to the commutativity of the idempotents 
$e_{i,k}$, we can rearrange the product $e_{i_\ell',n+1} \ldots e_{i_{p}',n+1}$ so that $i_\ell' < i_{\ell+1}' < \cdots < i_p'$.
Set $y':=g_{j_{s+1}} \ldots g_{j_r}e_{i_1,k_1}e_{i_2,k_2} \ldots e_{i_{\ell-1},k_{\ell-1}} g_{j_1}g_{j_2}\ldots g_{j_{s-1}} \in \YH(q)^{\rm (br)}$.
Following Remark \ref{all good}, we have, by rule (i) and Lemma \ref{trace_many_ei}, that
\begin{center}
$\Jtrm(y)=\Jtrm(y'e_{i_\ell',n+1} \ldots e_{i_{p}',n+1} g_n) =  z\, \Jtrm (y' e_{i_\ell',i_{\ell+1}'} \ldots e_{i_{p-1}',i_p'}). $
\end{center}
The element $y'':=y' e_{i_\ell',i_{\ell+1}'} \ldots e_{i_{p-1}',i_p'}$ belongs to $\YH(q)^{\rm (br)}$.
By the induction hypothesis, $\Jtrm (y'')$ is uniquely defined by rules (i)--(v). Moreover, we have
$\Jtrm(y'')=\Jtrs(y'')$, and so,  by  Lemma \ref{trace_many_ei}, 
\begin{center}
$\Jtrm(y)=z \, \Jtrm(y'') = z\,\Jtrs(y'')=\Jtrs(y)$.
\end{center}
\end{itemize}

We conclude that rules (i)--(v) suffice for the computation of $\Jtrm(a)$ and that $\Jtrm(a)=\Jtrs(a)$. Hence, we have proved that the specialized trace $\Jtrs$ can be 
defined inductively on $\YH(q)^{\rm (br)}$, and thus on $B_n$, by the following $5$ rules:
$$
\begin{array}{crcll}
{\rm (i)} & \Jtrs(a b) & = & \Jtrs(b a) & \qquad a, b\in\YH(q)^{\rm (br)} \\
{\rm (ii)} & \Jtrs(1) &  = & 1 & \qquad 1 \in\YH(q)^{\rm (br)} \\
{\rm (iii)} & \Jtrs(a g_n) & = & z\, \Jtrs( a) & \qquad a\in\YH(q)^{\rm (br)} \quad (\text{Markov  property} ) \\
{\rm (iv)} & \Jtrs(a e_n) & = & E_D\, \Jtrs( a) & \qquad a\in\YH(q)^{\rm (br)}\\
{\rm (v)} & \Jtrs(a e_ng_n) & = & z\, \Jtrs( a) & \qquad a\in\YH(q)^{\rm (br)} .
\end{array}
$$
\qed

Recall that, following Proposition~\ref{noinvts}, the invariants $\Theta_{d,D}$ are parametrized by the natural numbers. So from now on $D$ will always be $\Z/d\Z$, implying $E_D = 1/d$, and the invariants  will be simply denoted by $\Theta_d$. However, we will keep on using our initial notation for $E_D$, $\lambda_D$ and for the traces $\Jtr$ and $\Jtrs$ in order not to confuse the reader. Note that $\Theta_1$ is the Homflypt polynomial $P$.

\begin{rem} \label{new_program}\rm
Theorem~\ref{specialtrace} enabled the development of a program for computing the invariants $\Theta_d$ using the algorithm described in the proof, which is a modification of the algorithm used for the computation of the invariant $\Gamma_{d,D}$ (see \cite{chjakala}). This leads to much lower computational complexity, since we obtain only two new terms when applying the quadratic relation \eqref{quadr}.   Such a program has been developed by K. Karvounis \cite{ka} and it is available at \url{http://www.math.ntua.gr/~sofia/yokonuma}.
\end{rem}

\begin{rem}\label{E as a param}\rm
In Theorem~\ref{specialtrace}, we could have taken $E_m$ to be simply an indeterminate. However, since we do not have a basis for the algebra $\YH(q)^{\rm (br)}$, we have not been able to show that $\Jtrm$ is well-defined in this case.
One way to overcome this difficulty would be to use the result in \cite{EsRy} that, for $d \geq n$,
$\YH(q)^{\rm (br)}$ is isomorphic to the algebra of braids and ties $\mathcal{E}_n(q)$, 
for which we have a basis.
Then $\Jtrm$ would coincide with the Markov trace on $\mathcal{E}_n(q)$ defined in Section~\ref{threeVar} similarly to the one constructed in \cite{AJ} with the use of the old quadratic relations. Nevertheless, we will see later that the invariants $\Theta_d$
can be also defined via a skein relation, where $E_D$ can be taken to be an indeterminate without any problem (see Section~\ref{threeVar}). Note, though, that we do not expect that this generalization has much of an impact on the strength of the family of the invariants $\{\Theta_d\}_{d \in \N}$.
\end{rem}

\begin{rem} \label{E as a param u} \rm
Theorem~\ref{specialtrace}, and all its consequences, hold also for the specialized trace $\Jtrsu$. In particular, the values of the classical link invariants $\Delta_{d,D}$ depend only on the cardinality $\vert D \vert$ of $D$, so they can be parametrized by the natural numbers, setting $\Delta_d := \Delta_{d,\,\Z/d\Z}$ for all $d \in \Z_{>0}$. Moreover, in analogy to Remark~\ref{E as a param}, $E_m$ could be taken to be an indeterminate, and $\Jtrmu$  would be well-defined due to the isomorphism of $\YH(u)^{\rm (br)}$ with the algebra of braids and ties $\mathcal{E}_n(u)$ for $d \geq n$; then $\Jtrmu$ would coincide with the Markov trace on $\mathcal{E}_n(u)$ defined in \cite{AJ}.
\end{rem}

\section{Comparison of the invariants $\Theta_d$ with the Homflypt polynomial on classical knots}\label{DeltaHomflypt}

In this section, we will show that the invariants $\Theta_d$, constructed in \S \ref{classic},  are all topologically equivalent to the Homflypt polynomial when restricted to classical \textit{knots}, that is, links with only one component.

\subsection{Behaviour on knots}
While trying to compare computationally the invariants $\Theta_d$ with the Homflypt polynomial, we noticed that the values of the invariants on knots
were connected in the following way:
\begin{equation}
\Theta_d (q,z)(K) = \Theta_1 (q,{z}/{E_D})(K) = P(q,{z}/{E_D})(K) \qquad \text{where $K$ is a knot.}
\end{equation}
In this section, we will prove the above equality for all knots (Theorem \ref{conjHomflypt}).
In order to do this, we will first compare the specialized trace $\Jtrs:=\Jtrs(q,z)$ to the Ocneanu trace $\tau:=\tau(q,z/E_D)$. Let $n \in \N$. We will show that
\begin{equation}\label{CONJ}
\Jtrs(\alpha)  = E_D^{n-1} \,\tau (\alpha)
\end{equation} 
for all $\a \in B_n$ such that $\widehat{\a}$ is a knot.
Abusing the language, from now on, we will say that $\alpha \in B_{n}$ is a \emph{knot} if  $\widehat{\alpha}$ is a knot. 
Here are some useful remarks about knots:

\begin{rem}\label{Mark}
{\rm Let $\alpha, \alpha' \in B_n$. 
By the Markov theorem we have that ${\alpha\alpha'}$ is a knot if and only if ${\alpha'\alpha}$ is a knot.
Also, ${\alpha}$ is a knot if and only if ${\alpha\sigma_n^{\pm 1}}$ is a knot.}
\end{rem}

\begin{rem}\label{square}{\rm
We have that ${\alpha}$ is a knot if and only if $\pi(\alpha)$ is an $n$-cycle, where $\pi$ is the natural surjection from $B_{n}$ to $\mathfrak{S}_{n}$. 
As a consequence, ${\alpha}$ is a knot if and only if ${\alpha\sigma_i^{2k}}$ is a knot for all $i=1,\ldots,n-1$ and for all $k \in \mathbb{Z}$.}
\end{rem}

Thanks to the above remarks, we have the following lemma:

\begin{lemma}\label{LEMM}
Let $n \in \N$ and let $\alpha \in B_n$ be a knot.
If \eqref{CONJ} holds for all knots in $B_n$, then, for all $ i=1,\ldots,n-1$, we have
\begin{enumerate}[(a)]
\item $\Jtrs(\alpha \s_ie_i)  = E_D^{n-1} \,\tau (\alpha \s_i) $ ; \smallbreak
\item $\Jtrs(\alpha e_i)  = E_D^{n-1} \tau (\alpha) $ ;\smallbreak
\item $\Jtrs(\alpha \s_i^{-1}e_i)  = E_D^{n-1} \,\tau (\alpha \s_i^{-1}) .$  \smallbreak
\end{enumerate}
\end{lemma}

\begin{proof}
Let $i \in \{1,\ldots,n-1\}$.
If $\alpha \in B_n$ is a knot, then following Remark \ref{square}, we have that  $\alpha \sigma_i^2$ and $\alpha \sigma_i^4$ are knots.
So, using the assumption,
we obtain
\begin{equation}\label{first}
\Jtrs(\alpha)  = E_D^{n-1} \,\tau (\alpha),\,\, \Jtrs(\alpha \s_i^2)  = E_D^{n-1} \,\tau (\alpha \s_i^2),\,\,\Jtrs(\alpha \s_i^4)  = E_D^{n-1} \,\tau (\alpha \s_i^4).
\end{equation}

Now, by \eqref{quadr}, we have
 \begin{equation}\label{second}
 g_i^2 = 1 + (q-q^{-1}) g_ie_i \quad \text{ and } \quad G_i^2 = 1 +  (q-q^{-1}) G_i \,
 \end{equation} 
while, by Lemma \ref{gipower}, we have
 \begin{equation}\label{second2}
  g_i^4 = 1 + (q-q^{-1})^2 e_i + \left( \frac{q^4-q^{-4}}{q+q^{-1}} \right)g_ie_i  \quad \text{ and } \quad  G_i^4 = 1 + (q-q^{-1})^2 + \left( \frac{q^4-q^{-4}}{q+q^{-1}} \right)G_i.
 \end{equation} 
Combining (\ref{first}) with (\ref{second}) and (\ref{second2}) yields $(a)$ and $(b)$.

Finally, by \eqref{invrs}, we have
\begin{center}
$
g_i^{-1}e_i = g_ie_i - (q-q^{-1}) e_i \quad \text{ and } \quad G_i^{-1} = G_i -  (q-q^{-1}).
$
\end{center}
 Thus, $(c)$ is a direct consequence of $(a)$ and $(b)$.
\end{proof}

Note that the above lemma is also true if we replace everywhere $B_n$ by the \emph{braid monoid} $B_n^+$, that is, the set of braid words with non-negative exponents.
In order to prove \eqref{CONJ} for all knots in $B_n$, we will first prove it for all knots in $B_n^+$. For this, we will make use of the following result \cite[Lemma 3]{chla}.
Recall that, for any $\a \in B_n$, we denote by $\epsilon(\a)$ the sum of the exponents of the braiding generators $\s_i$ in the word  $\a$.

\begin{lemma}\label{lemma from chla}
Let $\a \in B_n^{+}$. Then one of the following hold:
\begin{enumerate}[(a)]
\item there exist $\a_1,\a_2 \in B_n^{+}$ such that $\s_n\a\s_n=\a_1 \s_n \a_2$ and $\epsilon(\s_n\a\s_n)=\epsilon(\a_1 \s_n \a_2)$, or \smallbreak
\item there exist $\beta_1,\beta_2 \in B_{n+1}^{+}$ and $1 \leq i \leq n$ such that $\s_n\a\s_n=\beta_1 \s_i^2 \beta_2$ and $\epsilon(\s_n\a\s_n)=\epsilon(\beta_1 \s_i^2 \beta_2)$.
\end{enumerate} 
\end{lemma}

We are now ready to prove the following.

\begin{prop}\label{step 1}
Let $\a \in B_n^{+}$ be a knot. Then
$$\Jtrs(\alpha)  = E_D^{n-1} \,\tau (\alpha).$$
\end{prop}

\begin{proof}
We will proceed by induction on $n$ and use the rules of definition of $\Jtrs$ (and thus, also of $\tau={\rm tr}_{1,\{0\}}$) described by Theorem \ref{specialtrace}.  For $n=1$, we have $\Jtrs(1)=1=\tau(1)$. 

Now assume that  \eqref{CONJ} holds for all knots in $B_{n}^+$. Note that, in this case, equalities $(a), (b), (c)$ of Lemma \ref{LEMM} hold for these knots. 
Our aim will be to prove that the relation holds for all knots $\alpha \in B_{n+1}^+$. 
In order to show this, we will use induction on $\epsilon(\a)$ to prove simultaneously the following four statements :
\begin{equation}\label{theconj}
\Jtrs(\alpha)  = E_D^n \,\tau (\alpha).
\end{equation}
and, for all $i=1,\ldots,n$, 
\begin{equation}\label{theconj2}
\Jtrs(\alpha \s_ie_i)  = E_D^n \,\tau (\alpha \s_i) 
\end{equation}
\begin{equation}\label{theconj3}
\Jtrs(\alpha e_i)  = E_D^n\,\tau (\alpha ) 
\end{equation}
\begin{equation}\label{theconj4}
\Jtrs(\alpha \s_i^{-1}e_i)  = E_D^n \,\tau (\alpha \s_i^{-1}) .
\end{equation}

As we have already seen, ${\alpha}$ is a knot if and only if $\pi(\alpha)$ is an $(n+1)$-cycle, where $\pi$ is the natural surjection from $B_{n+1}$ to $\mathfrak{S}_{n+1}$.
So the minimum value for $\epsilon(\alpha)$ is $n$, and it is achieved when all generators $\sigma_i$ appear in $\alpha$ with exponent $1$ exactly once each, that is,
$$\alpha = \sigma_{i_1} \ldots  \sigma_{i_n}, \quad \text{ where } \{i_1,\ldots,i_n\} = \{1,\ldots,n\}.$$
Because of rules (i) and (iii) in the definition of the traces, we may assume without loss of generality that $\alpha = \sigma_1\sigma_2\ldots\sigma_n$.
We have
$$\Jtrs(\alpha) = z^n = E_D^n \frac{z^n}{E_D^n} = E_D^n \,\tau(\alpha).$$
Moreover, we have
$$\begin{array}{rcl}
\Jtrs(\alpha \s_ie_i) &=& z^{n-i}\, \Jtrs(\s_1\ldots \s_{i-1}\s_i^2e_i)   \\ \smallbreak
& =& z^{n-i} \left( \Jtrs(\s_1\ldots \s_{i-1}e_i) +(q-q^{-1}) \Jtrs(\s_1\ldots \s_{i-1}\s_ie_i)\right) \\ \smallbreak
 &= &z^{n-i}  (E_D\, \Jtrs(\s_1\ldots \s_{i-1}) +(q-q^{-1}) z\, \Jtrs(\s_1\ldots \s_{i-1})) \\ \smallbreak
 &= &  E_D z^{n-1} + (q-q^{-1})\, z^n \\ \smallbreak
 &=  & E_D^n \left( \frac{z^{n-1}}{E_D^{n-1}} + (q-q^{-1})  \frac{z^{n}}{E_D^{n}}\right) \\ \smallbreak
 &=  & E_D^n \tau (\alpha \s_i)
 \end{array} $$
and
$$\Jtrs(\alpha e_i) = z^{n-i} \Jtrs(\s_1\ldots \s_{i-1}\s_ie_i) = z^{n-i+1}\Jtrs(\s_1\ldots \s_{i-1}) = z^n =E_D^n\, \tau(\alpha).$$
Finally, with the use of \eqref{invrs}, we obtain
$$\Jtrs(\a \s_i^{-1}e_i)=\Jtrs(\a \s_i e_i) - ( q-q^{-1}) \Jtrs(\a e_i) = E_D^n \tau (\alpha \s_i) -  ( q-q^{-1}) E_D^n\, \tau(\alpha) = E_D^n\, \tau(\alpha).$$

We will now prove the four statements for $\epsilon(\alpha)=m$ assuming that they hold for smaller values of $\epsilon(\alpha)$.
If there exists $\sigma_i$ such that $\alpha = \beta_1 \sigma_i^2 \beta_2$ for some $\beta_1,\beta_2 \in B_n^+$, then, following Remarks \ref{Mark} and \ref{square}, 
${\beta}:=\beta_2 \beta_1$ is a knot in $B_{n+1}^+$. Then $\epsilon(\beta) = \epsilon(\alpha)-2 < \epsilon(\alpha)$. By the induction hypothesis on $\epsilon(\alpha)$, we have
$$\Jtrs(\beta)  = E_D^n \,\tau (\beta)
\quad \text{ and } \quad \Jtrs(\beta \s_ie_i)  = E_D^n \,\tau (\beta \s_i).$$
We deduce that
$$\Jtrs(\alpha) = \Jtrs(\beta) + (q-q^{-1}) \Jtrs(\beta \s_ie_i) =  E_D^n( \tau (\beta) +(q-q^{-1})  \tau (\beta \s_i)) =E_D^n\,\tau(\alpha).$$
If there exists no such $\sigma_i$, then all generators appear in $\alpha$ with exponent $1$. 
If $\sigma_n$ appears only once in $\alpha$ and $\alpha = \beta_1 \sigma_n \beta_2$ for some $\beta_1,\beta_2 \in B_n^+$, then, by Remark \ref{Mark}, ${\beta}:=\beta_2 \beta_1$ is a knot in $B_n^+$. Following the induction hypothesis on $n$, we 
obtain
$$\Jtrs(\alpha) = \Jtrs(\beta \s_n) = z \, \Jtrs(\beta) = z \, E_D^{n-1} \, \tau(\beta) = E_D^n \, \frac{z}{E_D} \, \tau(\beta) =E_D^n\, \tau(\alpha).$$
If all generators appear with exponent $1$ and $\sigma_n$ appears more than once in $\alpha$, then, due to Lemma \ref{lemma from chla}, we can apply the braid relations and obtain an expression for $\alpha$ with the same $\epsilon(\a)$ such that either a generator appears with exponent greater than $1$ or  $\sigma_n$ appears only once.

We have thus proved  \eqref{theconj} for all knots $\alpha \in B_{n+1}^+$. Lemma \ref{LEMM} (restricted to the braid monoid $B_{n+1}^+$) yields (\ref{theconj2}), (\ref{theconj3}) and (\ref{theconj4}) for $\alpha \in B_{n+1}^+$. 
\end{proof}

Using Proposition \ref{step 1}, we will prove \eqref{CONJ} for all knots in $B_n$.

\begin{prop}\label{step 2}
Let $\a \in B_n$ be a knot. Then
$$\Jtrs(\alpha)  = E_D^{n-1} \,\tau (\alpha).$$
\end{prop}

\begin{proof}
Again, we will use  induction on $n$ and  the rules of definition of $\Jtrs$  described by Theorem \ref{specialtrace}.  For $n=1$, we have $\Jtrs(1)=1=\tau(1)$. 

Now assume that  \eqref{CONJ} holds for all knots in $B_{n}$. Note that, in this case, equalities $(a), (b), (c)$ of Lemma \ref{LEMM} hold for these knots. 
Our aim will be to prove that the relation holds for all knots $\alpha \in B_{n+1}$. 
In order to show this, we will use induction on
$$S(\alpha):= \text{Sum of positive exponents of } \alpha  -2 \cdot \text{Sum of negative exponents of } \alpha \, $$
to prove simultaneously statements \eqref{theconj},  \eqref{theconj2}, \eqref{theconj3} and \eqref{theconj4} for all knots $\a \in B_n$.

As we have already seen, ${\alpha}$ is a knot if and only if $\pi(\alpha)$ is an $(n+1)$-cycle, where $\pi$ is the natural surjection from $B_{n+1}$ to $\mathfrak{S}_{n+1}$.
So the minimum value for $S(\alpha)$ is $n$, and it is achieved when all generators $\sigma_i$ appear in $\alpha$ with exponent $1$ exactly once each, that is,
$$\alpha = \sigma_{i_1} \ldots  \sigma_{i_n}, \quad \text{ where } \{i_1,\ldots,i_n\} = \{1,\ldots,n\}.$$
This case has already  been  covered by Proposition \ref{step 1}.

We will now prove the four statements for $\epsilon(\alpha)=m$ assuming that they hold for smaller values of $\epsilon(\alpha)$.
We have already proved them  for $\alpha \in B_{n+1}^+$.
So let us assume that  $\alpha \in B_{n+1} \setminus B_{n+1}^+$. Then there exists $i \in \{1,\ldots,n\}$ such that 
$\alpha = \beta_1 \sigma_i^{-1} \beta_2$ for some $\beta_1,\beta_2 \in B_n$. Following Remarks \ref{Mark} and \ref{square}, 
${\beta}:= \beta_2\beta_1\sigma_i $ is a knot in $B_{n+1}$.  We have 
$S(\beta) = S(\alpha)-1 < S(\alpha)$. By induction hypothesis, we obtain
$$\Jtrs(\beta )  = E_D^n \,\tau (\beta) \quad \text{ and } \quad \Jtrs(\beta \sigma_i^{-1} e_i )  = E_D^n \,\tau (\beta \sigma_i^{-1}).$$ 
Note that $\beta \sigma_i^{-1}=\beta_2\beta_1$. By \eqref{invrs}, we have
$$\Jtrs(\alpha)= \Jtrs(\beta ) - (q-q^{-1}) \Jtrs(\beta\sigma_i^{-1} e_i )
= E_D^n (\tau(\beta) - (q-q^{-1}) \tau(\beta \s_i^{-1})) = E_D^n \,\tau(\alpha).$$

We have thus proved \eqref{theconj} for all knots $\alpha \in B_{n+1}$. Lemma \ref{LEMM} implies (\ref{theconj2}), (\ref{theconj3}) and (\ref{theconj4}) for $\alpha \in B_{n+1}$.
\end{proof}

The following result is an immediate consequence of Proposition \ref{step 2} and Lemma \ref{LEMM}.

\begin{cor}\label{CORR}
Let $\alpha \in B_n$ be a knot.
Then, for all $ i=1,\ldots,n-1$, we have
\begin{enumerate}[(a)]
\item $\Jtrs(\alpha \s_ie_i)  = E_D^{n-1} \,\tau (\alpha \s_i) $ ; \smallbreak
\item $\Jtrs(\alpha e_i)  = E_D^{n-1} \tau (\alpha) $ ;\smallbreak
\item $\Jtrs(\alpha \s_i^{-1}e_i)  = E_D^{n-1} \,\tau (\alpha \s_i^{-1}) .$  \smallbreak
\end{enumerate}
\end{cor}

Now, Proposition \ref{step 2} implies the main result of this section, which is the following.

\begin{thm} \label{conjHomflypt}
Given a solution $X_D$ of the {\rm E}--system, for any braid $\a \in B_{n}$ such that $\widehat{\a}$ is a knot, we have:
$$\Theta_d(q,z) (\widehat{\a}) = \Theta_1  (q,{z}/{E_D}) (\widehat{\a}) = P(q,{z}/{E_D}) (\widehat{\a}).$$
\end{thm}

\begin{proof} Let $\a \in B_n$ be a knot.
Applying the transformation $z \mapsto z/E_D$ on the Homflypt polynomial $P(q,z)$, we get
$$
P(q, z/E_D) (\widehat{\a})= {\left( \frac{E_D}{z \sqrt{\lambda_{\rm H}}} \right)}^{n-1}  \sqrt{\lambda_{\rm H}}^{\epsilon(\a)} \tau(q,z/E_D)(\a)
$$
where
$$\lambda_{\rm H}= \frac{z-(q-q^{-1})E_D}{z} = \lambda_D.$$
By Proposition \ref{step 2}, we have
$$ \Jtrs(q,z)(\a) = E_D^{n-1}\tau(q,z/E_D)(\a).$$
We deduce that
$$P(q, z/E_D) (\widehat{\a}) = {\left( \frac{1}{z \sqrt{\lambda_D}} \right)}^{n-1}  \sqrt{\lambda_D}^{\epsilon(\a)}  \Jtrs(q,z)(\a) = \Theta_d(q,z) (\widehat{\a}).$$
\end{proof}

\begin{rem} \rm
The transformation $z  \mapsto {z}/{E_D}$ corresponds to the transformation $\lambda_{\rm H} \mapsto \lambda_D$ on the Homflypt polynomial at variables $(q,\lambda_{\rm H})$. Hence, we have equivalently that $\Theta_d(q,\lambda_D) = P(q,\lambda_D)$, where the variable $\lambda_{\rm H}$ of the Homflypt polynomial has been substituted with the variable $\lambda_D$. Consequently, $E_D$ does not appear in the values of the invariants $\Theta_d$ at variables $(q,\lambda_D)$ when computed on knots.
\end{rem}

\begin{rem}\label{kauffman polynomial}\rm
It is worth adding here that the invariants $\Theta_d$ are not topologically equivalent to the Kauffman polynomial \cite{kau}, since there is at least one pair of knots which are distinguished by the Homflypt polynomial but not by the Kauffman polynomial. 
\end{rem}

\section{Obtaining the invariants $\Theta_d$ for classical links via the Homflypt skein relation}\label{slinks}

In this section we investigate the behaviour of the invariants $\Theta_d$ 
on classical links with at least two components. Let $L = L_1 \sqcup \ldots \sqcup L_m$ a split link, where $L_1, \ldots, L_m$ are links. By the multiplicative property of the invariants $\Theta_d$, we have (see \cite[Proposition 3.3]{chjakala}):
\begin{equation}\label{mult prop}
\Theta_d(L) = \Lambda_D^{m-1} \Theta_d(L_1) \ldots \Theta_d(L_m).
\end{equation}
So by this property, we only need to examine non-split links.

Again, we consider the Ocneanu trace $\tau$ and the Homflypt polynomial $P$ at variables $(q,{z}/{E_D})$ (or equivalently $q,\lambda_D$). 
In general, it does not seem that an analogue of Theorem~\ref{conjHomflypt} holds for links. For example, for the simplest possible non-split link, the Hopf link $H = \widehat{\s_1^2}$, we have:
$$
\tau(\s_1^2) = 1 + (q-q^{-1}) \frac{z}{E_D},
$$
and
\begin{center}
$\Jtrs(\s_1^2) = \Jtrs(1 + (q-q^{-1})g_1 e_1) = 1 + (q-q^{-1}) z = 1 - E_D + E_D \, \tau(\s_1^2).$
\end{center}
Therefore,
$$
\Theta_d(q,z)(H) \neq P(q,z/E_D)(H).
$$

However, we have an analogue of Theorem~\ref{conjHomflypt} for disjoint union of knots (Theorem~\ref{disj_links}). Furthermore, by a special skein relation (see Proposition~\ref{skein_ei}) we will see that we can reduce the general case of links to the case of disjoint unions of knots. We will conclude that the invariants $\Theta_d$ can be defined with the use of the Homflypt skein relation.

\subsection{Behaviour on disjoint unions of knots}
We shall now consider links whose components are not linked. Then a variant of Theorem~\ref{conjHomflypt} holds.
Indeed, this is clear from Theorem~\ref{conjHomflypt}, using the multiplicative property \eqref{mult prop} and the fact that a disjoint union of knots is a special case of a split link. However, we provide below a self-contained proof using the results of this paper. Indeed, for the trace $\Jtrs$ we have:

\begin{prop}\label{unlink_lemma}
Let $\a \in B_n$ be such that $\widehat{\a}$ is a link with $k$ components which are not linked with each other, that is, $\widehat{\alpha}$ is a disjoint union of $k$ knots. Then
$$
\Jtrs(\a) = E_D^{n-k} \tau(\a).
$$
\end{prop}

\begin{proof}
We prove the statement by induction on the number of components $k$. For $k=1$, the statement is Proposition \ref{step 2}. Assume now that the statement holds for links with up to  $k$ disjoint components and let $\a$ be a link with $k+1$ disjoint components $\a_1,\a_2,\ldots,\a_{k+1}$. We have $\a = \a_1\a_2\ldots \a_{k+1}$, and  the $\a_j$ commute with each other.
Suppose that $\a_1$ is the component expressed as a braid word in the generators $\s_1,\ldots, \s_{m-1}$ for some $m \leq n$.
Consequently, $\a_2\ldots \a_{k+1}$ is expressed as a braid word in the generators $\s_{m+1},\ldots, \s_{n-1}$ (we will also write $\a_2\ldots \a_{k+1} \in B_n \setminus B_{m+1}$).
Therefore, in computing $\Jtrs(\a)$ and $\tau(\a)$, we will first exhaust the word $\a_2\ldots \a_{k+1}$. We thus obtain
$$\Jtrs(\a)= \Jtrs( \a_1)\, \Jtrs(\a_2 \ldots \a_{k+1}) \quad \text{and} \quad 
\tau(\a)=\tau( \a_1) \,\tau(\a_2 \ldots \a_{k+1}).$$
By the induction hypothesis, we have
$$ \Jtrs(\a_1) = E_D^{m-1}\tau(\a_1) \quad \text{and} \quad \Jtrs(\a_2 \ldots \a_{k+1}) E_D^{n-m-k} \tau(\a_2 \ldots \a_{k+1}).$$
We deduce that
$$\Jtrs(\a)= E_D^{n-(k+1)} \tau(\a),$$
as required.
\end{proof}

Now we can extend Theorem \ref{conjHomflypt} to disjoint unions of knots as follows:

\begin{thm}\label{disj_links}
Given a solution $X_D$ of the {\rm E}--system, for any braid $\a \in B_{n}$ such that $\widehat{\a}$ is a disjoint union of $k$ knots, we have 
\begin{center}
$\Theta_d (q,z)(\widehat{\a}) = E_D^{1-k} \Theta_1 (q,{z}/{E_D})(\widehat{\a}) = E_D^{1-k} P(q,{z}/{E_D})(\widehat{\a})$.
\end{center}
\end{thm}

\begin{proof} Let $\a \in B_n$ be such that $\widehat{\a}$ is a disjoint union of $k$ knots.
Applying the transformation $z \mapsto z/E_D$ on the Homflypt polynomial $P(q,z)$, we get
$$
P(q, z/E_D) (\widehat{\a})= {\left( \frac{E_D}{z \sqrt{\lambda_{\rm H}}} \right)}^{n-1}  \sqrt{\lambda_{\rm H}}^{\epsilon(\a)} \tau(q,z/E_D)(\a)
$$
where
$$\lambda_{\rm H}= \frac{z-(q-q^{-1})E_D}{z} = \lambda_D.$$
By Proposition \ref{unlink_lemma}, we have
$$ \Jtrs(q,z)(\a) = E_D^{n-k}\tau(q,z/E_D)(\a).$$
We deduce that
$$P(q, z/E_D) (\widehat{\a}) = E_D^{k-1}\, {\left( \frac{1}{z \sqrt{\lambda_D}} \right)}^{n-1}  \sqrt{\lambda_D}^{\epsilon(\a)} \Jtrs(q,z)(\a) = E_D^{k-1}\,\Theta_d(q,z) (\widehat{\a}).$$\end{proof}

\subsection{Behaviour on the links $\widehat{\a \s_n^{2k} \b}$}\label{asn2kb}
The general case, which remains to be examined, is that of links  whose components are linked together (some of them, if not all).
Let us consider first this simple family of 2-component links which are only interconnected by $|k|$ 
consecutive clasps, and which includes the Hopf link as the basic case: let $\a \in B_n$ and $\b \in B_{m}$, for some $n, m \in \N$, such that $\widehat{\a}$ and $\widehat{\b}$ are disjoint unions of $s$ and $t$ knots respectively (see Figures~\ref{braid_clapsed} and \ref{3clasps}).

\begin{figure}[H]
\begin{center}
\scalebox{0.75}{
\begin{picture}(107,170)
\put(20,133){$\alpha$}

\qbezier(6,150)(6,158)(6,166) 
\qbezier(25,150)(25,158)(25,166)
\qbezier(44,150)(44,158)(44,166)

\qbezier(0,124)(0,139)(0,150)
\qbezier(50,124)(50,139)(50,150)

\qbezier(0,150)(25,150)(50,150)
\qbezier(0,124)(25,124)(50,124)

\qbezier(6,0)(6,112)(6,124)
\qbezier(25,0)(25,112)(25,124)
\color[RGB]{153,0,230} \qbezier(44,100)(44,112)(44,124)

\color[RGB]{0,0,0} 
\put(77,133){$\beta$}

\qbezier(63,150)(63,158)(63,166) 
\qbezier(82,150)(82,158)(82,166)
\qbezier(101,150)(101,158)(101,166)

\qbezier(57,124)(57,139)(57,150)
\qbezier(107,124)(107,139)(107,150)

\qbezier(57,150)(82,150)(107,150)
\qbezier(57,124)(82,124)(107,124)

\color[RGB]{0,179,51}\qbezier(63,100)(63,112)(63,124)
\color[RGB]{0,0,0} \qbezier(82,0)(82,112)(82,124)
\qbezier(101,0)(101,112)(101,124)

\color[RGB]{0,179,51} \qbezier(44,80)(44,85)(54,90)
\qbezier(54,90)(63,95)(63,100)
\color[RGB]{153,0,230} \qbezier(63,80)(63,84)(59,87)
\qbezier(49,93)(44,96)(44,100)

\color[RGB]{153,0,230} \qbezier(44,60)(44,65)(54,70)
\qbezier(54,70)(63,75)(63,80)
\color[RGB]{0,179,51} \qbezier(63,60)(63,64)(59,67)
\qbezier(49,73)(44,76)(44,80)

\color[RGB]{0,0,0} 
\put(48,45){$\vdots \, k$}

\color[RGB]{0,179,51} \qbezier(44,20)(44,25)(54,30)
\qbezier(54,30)(63,35)(63,40)
\color[RGB]{153,0,230} \qbezier(63,20)(63,24)(59,27)
\qbezier(49,33)(44,36)(44,40)

\color[RGB]{153,0,230} \qbezier(44,0)(44,5)(54,10)
\qbezier(54,10)(63,15)(63,20)
\color[RGB]{0,179,51} \qbezier(63,0)(63,4)(59,7)
\qbezier(49,13)(44,16)(44,20)
\end{picture}
}
\caption{The braid $\b \a \s_n^{2k}$, for $k>0$.}
\label{braid_clapsed}
\end{center}
\end{figure}
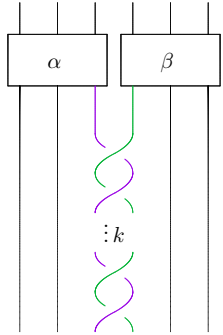
 
If we consider the braid $\b$ shifted into the group $B_{n+m}$ (so that $\b \in B_{n+m} \setminus B_{n+1}$), the trace of $\b$ does not change. So, for all $k \in \Z$, the closure of the braid $\a \s_n^{2k} \b \in B_{n+m}$ is a link with 
$(s+t)$ components. We then have
$$\Jtrs(\a \s_n^{2k} \b) = \Jtrs(\a) \, \Jtrs(\s_n^{2k}) \, \Jtrs(\b).$$ 
By Proposition \ref{unlink_lemma}, we have
$$\Jtrs(\a)=E_D^{n-s} \tau(\a) \quad\text{and}\quad \Jtrs(\b)=E_D^{m-t} \tau(\b),$$
while, due to Lemma \ref{gipower}, we have 
$$\Jtrs(\s_n^{2k}) = E_D \, \tau(\s_n^{2k})+1-E_D.$$
Thus, we obtain
$$\Jtrs(\a \s_n^{2k} \b)  = E_D^{n+m-(s+t-1)}\tau(\a \s_n^{2k} \b)  +  E_D^{n+m-(s+t)}(1-E_D)\,\tau(\a\b),$$
whence
$$\begin{array}{rcl}
\Theta_d (\widehat{\a \s_n^{2k} \b}) &=& E_D^{1-(s+t-1)}P(\widehat{\a \s_n^{2k} \b})+ \lambda_D^kE_D^{1-(s+t)}(1 - E_D) P(\widehat{\a \b})\\ & &\\
& = & E_D^{1-(s+t)} \left[ E_D \, P(\widehat{\a \s_n^{2k} \b}) + \lambda_D^k (1 - E_D) P(\widehat{\a \b}) \right].
\end{array}
$$

Hence, the invariants $\Theta_d$ appear to depend not only on the Homflypt polynomial of the same link (as expected), but also on the Homflypt polynomial of the link comprising the two components unlinked. Further, we note that  the linking number between the two components appears on the coefficient $\lambda_D^k$ of $P(\widehat{\a \b})$.

\subsection{Behaviour of $\Jtrs$ on the elements $\a e_i$}\label{aei} The behaviour of the invariants $\Theta_d$ on arbitrary links depends, as we shall see, on a property of the specialized trace $\Jtrs$ in relation to the elements $e_i$, and a skein relation that this implies. Indeed, let $\a \in B_n$.
The behaviour of the specialized trace $\Jtrs$ on the elements $\a e_i$ seems to depend on whether the strands $i$ and $i+1$ belong to the same component of the link $\widehat{\a}$ or not. For example, an immediate consequence of Proposition \ref{step 2} and Corollary \ref{CORR} is  the following:

\begin{lemma}\label{lem and conj}
If $\a \in B_n$ is a braid such that $\widehat{\a}$ is a knot, then
\begin{equation}\label{good for knots}
\Jtrs(\a e_i) = \Jtrs(\a) \quad \text{for all }i =1,\ldots,n-1.
\end{equation}
\end{lemma}

However, Lemma~\ref{lem and conj} does not  hold for links. The simplest case is the disjoint union of knots. Indeed, let $\widehat{\a}$ be the disjoint union of the knots $\widehat{\a_j}$, for $j=1,\ldots,k$. We have $\a=\a_1 \a_2 \ldots \a_k \in B_n$, where $\a_j \in B_{i_j} \setminus B_{i_{j-1}+1}$ for some $1 \leq i_1 < \cdots < i_k \leq n$ with $i_j - i_{j-1} >1$ (we take $i_0:=-1$ and $B_0:=\emptyset$).  Then we have the following:

\begin{lemma}\label{trdis} Let $\a=\a_1 \ldots \a_k \in B_n$ be a disjoint union of $k$ knots as above. Let $i \in \{1,\ldots,n-1\}$. Then
$$\Jtrs(\a e_i) = \left\{
\begin{array}{ll}
\Jtrs(\a) &\text{ if } i \notin \{i_1, \ldots, i_k\} ; \\
E_D \, \Jtrs(\a) &\text{ if } i \in \{i_1, \ldots, i_k\}.
\end{array}\right.$$
\end{lemma}

\begin{proof}

If $i \notin \{i_1, \ldots, i_k\}$, then $i_{j-1} < i < i_j$ for some $j=1,\ldots,k$. Then
$$
\begin{array}{rclll}
\Jtrs(\a e_i)& = &\Jtrs(\a_1 \ldots \a_{j-1}) \, \Jtrs(\a_j e_i) \, \Jtrs(\a_{j+1} \ldots \a_k)  & &\\
& \stackrel{(\ref{good for knots})}{=} &\Jtrs(\a_1 \ldots \a_{j-1}) \, \Jtrs(\a_j) \, \Jtrs(\a_{j+1} \ldots \a_k) & =&  \Jtrs(\a),
\end{array}$$
On the other hand, if $i=i_j$ for some $j=1,\ldots,k$, then
$$
\begin{array}{rclll}
\Jtrs(\a e_{i_j})& = &\Jtrs(\a_1 \ldots \a_{j-1}) \, \Jtrs(\a_j e_{i_j}) \, \Jtrs(\a_{j+1} \ldots \a_k) & &\\
& \stackrel{(\ref{later for classical})}{=} &\Jtrs(\a_1 \ldots \a_{j-1}) \, E_D \, \Jtrs(\a_j) \, \Jtrs(\a_{j+1} \ldots \a_k) & = & E_D\, \Jtrs(\a).
\end{array}
$$
\end{proof}

Let us now investigate how $\Jtrs(\a e_i)$ behaves in more complicated links $\a$, where some components may be linked. For example, 
using Lemma~\ref{gipower}, we obtain the following:

\begin{lemma} Let $i \in \{1,\ldots,n-1\}$ and let $k \in \Z$. Then 
\begin{equation}\label{s_i^{2k}}
\Jtrs( \s_i^{2k}e_i)= \Jtrs(\s_i^{2k}) + E_D -1 .
\end{equation}
\end{lemma}

Let us now consider a type of links studied in 
Subsection~\ref{asn2kb}.

\begin{lemma} Let $\a \in B_n$, $\b \in B_{n+m} \setminus B_{n+1}$ be such that $\widehat{\a}$, $\widehat{\b}$ are knots. 
Let $i \in \{1,\ldots,n+m-1\}$ and let $k \in \Z$. Then
$$\Jtrs(\a  \s_n^{2k} \b e_i) = \left\{
\begin{array}{ll}
\Jtrs(\a  \s_n^{2k} \b) + (E_D-1) \Jtrs(\a   \b) &\text{ if } i =n ; \\
\Jtrs(\a  \s_n^{2k} \b) &\text{ if } i\neq n.
\end{array}\right.$$
\end{lemma}

\begin{proof} 
First, we have
$$\begin{array}{rcl}
\Jtrs(\a  \s_n^{2k} \b e_n) & = & \Jtrs(\a) \Jtrs( \s_n^{2k} e_n) \Jtrs(\b) \\ 
&  \stackrel{(\ref{s_i^{2k}})}{=} &\Jtrs(\a) \Jtrs( \s_n^{2k} ) \Jtrs(\b) + (E_D-1)  \Jtrs(\a)\Jtrs(\b)\\  
&= & \Jtrs(\a  \s_n^{2k} \b) + (E_D-1) \Jtrs(\a   \b). 
\end{array}$$
Now, if $i<n$, then
$$\Jtrs(\a  \s_n^{2k} \b e_i) =  \Jtrs(\a e_i) \Jtrs( \s_n^{2k}\b) \stackrel{(\ref{good for knots})}{=} \Jtrs(\a) \Jtrs( \s_n^{2k}\b) =
 \Jtrs(\a  \s_n^{2k} \b).$$
 Similarly, if $i>n$, then
  $$\Jtrs(\a  \s_n^{2k} \b e_i) =  \Jtrs(\a \s_n^{2k})\Jtrs(\b e_i) \stackrel{(\ref{good for knots})}{=} \Jtrs(\a\s_n^{2k}) \Jtrs( \b) =
 \Jtrs(\a  \s_n^{2k} \b).$$
\end{proof}

\begin{rem}
\rm Using Lemma \ref{trdis}, we can obtain a similar result for all links of the form $\widehat{\a \s_n^{2k} \b}$, where $\widehat{\a}$ and $\widehat{\b}$ are disjoint unions of knots.
\end{rem}

\subsection{A special skein relation for links}\label{special_skein}
The behaviour of $\Jtrs$ on the elements $\a e_i$ explored 
in Subsection~\ref{aei} makes the search for a skein relation for the invariants $\Theta_d$ difficult, since $e_i$ does not have by itself a topological interpretation on classical links. However, we shall show that there exists a (non-defining) skein relation for the invariants $\Theta_d$, which can only be applied on crossings involving two different components. The result is  proved via the invariants $\Phi_{d,D}$. Specifically, we have the following.

\begin{prop}\label{skein_ei}
Let $\b \in  {\mathcal F}_n$ and $i \in \{1,\ldots,n-1\}$. Let
$$L_+ = \widehat{\b \s_i},\,\,\,\,\,L_-= \widehat{\b \s_i^{-1}} \,\,\,\,\,\text{and}\,\,\, \,\,L_0 = \widehat{\b}.$$
Suppose we apply the skein relation \eqref{framed_skein} of $\Phi_{d,D}$  on $L_+$ on the crossing $\s_i$ and that the $i$-th and $(i+1)$-st strands (at the region of the crossing) belong to different components. Then the skein relation reduces to the skein relation of the Homflypt polynomial $P=P(q,\lambda_D)$:
\begin{equation}\label{new skein}
\frac{1}{\sqrt{\lambda_D}} \Phi_{d,D}(L_+) - \sqrt{\lambda_D} \Phi_{d,D}(L_-) = (q-q^{-1}) \Phi_{d,D}(L_0),
\end{equation}
see Figure~\ref{fig_skein_ei}. Furthermore, if we take $\b \in B_n$ as a framed braid with all framings zero, then the above skein relation of $\Phi_{d,D}$ also holds for the invariants $\Theta_d$, since it involves only classical links:
\begin{equation}\label{new skein delta}
\frac{1}{\sqrt{\lambda_D}} \Theta_d(L_+) - \sqrt{\lambda_D} \Theta_d(L_-) = (q-q^{-1}) \Theta_d(L_0),
\end{equation}
\end{prop}

\begin{figure}[H]
\begin{center}
\begin{picture}(216,80)
\put(20,53){$\beta$}

\qbezier(6,70)(6,78)(6,86) 
\qbezier(25,70)(25,78)(25,86)
\qbezier(44,70)(44,78)(44,86)

\color[RGB]{0,179,51} \qbezier(5,20)(4,22)(5,24) 
\color[RGB]{153,0,230} \qbezier(25,20)(26,22)(25,24) 

\color[RGB]{0,0,0}

\qbezier(0,44)(0,59)(0,70)
\qbezier(50,44)(50,59)(50,70)

\qbezier(0,70)(25,70)(50,70)
\qbezier(0,44)(25,44)(50,44)

\color[RGB]{0,179,51} \qbezier(5,24)(5,29)(15,34)
\qbezier(15,34)(25,39)(25,44)
\color[RGB]{153,0,230} \qbezier(25,24)(25,28)(20,31)
\qbezier(10,37)(5,40)(5,44)

\color[RGB]{0,0,0} \qbezier(44,20)(44,32)(44,44)

\put(100,53){$\beta$}

\qbezier(86,70)(86,78)(86,86) 
\qbezier(105,70)(105,78)(105,86)
\qbezier(124,70)(124,78)(124,86)

\color[RGB]{0,179,51} \qbezier(85,20)(84,22)(85,24) 
\color[RGB]{153,0,230} \qbezier(105,20)(106,22)(105,24) 

\color[RGB]{0,0,0}

\qbezier(80,44)(80,59)(80,70)
\qbezier(130,44)(130,59)(130,70)

\qbezier(80,70)(105,70)(130,70)
\qbezier(80,44)(105,44)(130,44)

\color[RGB]{0,179,51} \qbezier(85,24)(85,28)(90,31)
\qbezier(100,37)(105,40)(105,44)
\color[RGB]{153,0,230} \qbezier(105,24)(105,29)(95,34)
\qbezier(95,34)(85,39)(85,44)

\color[RGB]{0,0,0} \qbezier(124,20)(124,32)(124,44)

\put(180,53){$\beta$}

\qbezier(166,70)(166,78)(166,86) 
\qbezier(185,70)(185,78)(185,86)
\qbezier(204,70)(204,78)(204,86)

\qbezier(160,44)(160,59)(160,70)
\qbezier(210,44)(210,59)(210,70)

\qbezier(160,70)(185,70)(210,70)
\qbezier(160,44)(215,44)(210,44)

\color[RGB]{153,0,230} \qbezier(166,20)(166,32)(166,44) 
\qbezier(185,20)(185,32)(185,44)
\color[RGB]{0,0,0} \qbezier(204,20)(204,32)(204,44)

\put(5,5){\tiny{$L_{+}=\widehat{\beta \sigma_1}$}}
\put(80,5){\tiny{$L_{-}=\widehat{\beta {\sigma_1}^{-1}}$}}
\put(165,5){\tiny{$L_0=\widehat{\beta}$}}
\end{picture}
\caption{The links in the special 	skein relation in open braid form.}
\label{fig_skein_ei}
\end{center}
\end{figure}
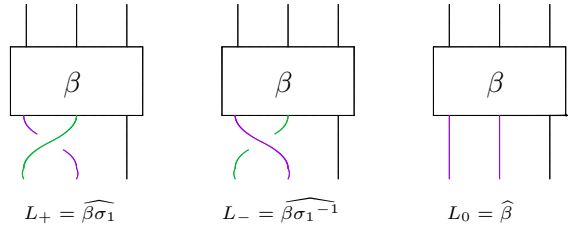

\begin{proof}
The skein relation of $\Phi_{d,D}$ is
$$
\frac{1}{\sqrt{\lambda_D}} \Phi_{d,D}(L_+) - \sqrt{\lambda_D} \Phi_{d,D}(L_-) = \frac{q-q^{-1}}{d} \sum_{s=0}^{d-1} \Phi_{d,D}(L_s),
$$
where $L_s = \widehat{\b t_i^s t_{i+1}^{d-s}}$ (recall Figure~\ref{fig_framed_skein}).  Since the $i$-th and $(i+1)$-st strands belong to different components of $L_+$, then they must belong to the same component of $L_s$ (otherwise, the crossing $\s_i$ would put them in the same component of $L_+$). Thus, the link $L_s$ can be also represented by the braid where the framing of the $(i+1)$-st strand is added to the framing of the $i$-th strand. Hence, both framed braids $\b t_i^s t_{i+1}^{d-s}$ and $\b t_i^{s + d-s}=\b$ represent $L_s$, for all $s=0,\ldots,d-1$. We conclude that
$$
\frac{1}{d} \sum_{s=0}^{d-1} \Phi_{d,D}(L_s)  = \frac{1}{d} \sum_{s=0}^{d-1} \Phi_{d,D}(\widehat{\b}) = \Phi_{d,D}(\widehat{\b}) =  \Phi_{d,D}(L_0).
$$
\end{proof}

\begin{rem} \rm \label{skein_pm2}
Another way of writing the skein relation \eqref{new skein} for $\Phi_{d,D}$ is the following:
\begin{equation}\label{new skein 2}
\Phi_{d,D}(L_\pm) = \lambda_D^{\pm 1} \Phi_{d,D}(L_\mp) \pm (q-q^{-1}) \sqrt{\lambda_D}^{\pm 1} \Phi_{d,D}(L_0).
\end{equation}
\end{rem}

\begin{rem} \rm \label{oldquadr6}
With the old quadratic relations, we would not have been able to derive the skein relation of Proposition~\ref{skein_ei} for the invariants $\Delta_d$. The reason is that the skein relation \eqref{framed_skein_old} for the invariants $\Gamma_{d,D}$ contains also diagrams $L_{s\times}$ of $L_+$ with framings, which cannot be collected together since they belong to different components.
\end{rem}

\bigbreak
Proposition~\ref{skein_ei} allows us to prove that the elements $\b e_i$ have a specific behaviour under the trace $\Jtrs$ when the $i$-th and $(i+1)$-st strands belong to the same component of the link $\widehat{\b}$. 

\begin{prop}\label{trace_ei}
Let $\b \in {\mathcal F}_n$ and $i \in \{1,\ldots,n-1\}$. If the $i$-th and $(i+1)$-st strands of $\b$ belong to the same component of the link  $\widehat{\b}$, then $\Jtrs(\b e_i) = \Jtrs(\b)$. 
\end{prop}

\begin{proof}
Consider the framed link $\widehat{\b \s_i}$. The $i$-th and $(i+1)$-st strands of $\widehat{\b \s_i}$ belong to different components. By Proposition~\ref{skein_ei}, we have:
$$
\Phi_{d,D}(L_0) = \frac{1}{d} \sum_{s=0}^{d-1} \Phi_{d,D}(L_s),
$$
where $L_s = \widehat{\b t_i^s t_{i+1}^{d-s}}$ and $L_0 = \widehat{\b}$. Since $L_s$ and $L_0$, represented as framed braids, have the same number of strands and the same exponent sum, the definition of the invariant $\Phi_{d,D}$ yields
$$
\Jtrs(\b) = \frac{1}{d} \sum_{s=0}^{d-1} \Jtrs(\b t_i^s t_{i+1}^{d-s}),
$$
Due to the linearity of $\Jtrs$, we obtain the desired result.
\end{proof}

\begin{cor} \label{trace_ei_cor}
Let $\b \in {\mathcal F}_n$ and $i,k \in \{1,\ldots,n\}$. If the $i$-th and $k$-th strands of $\b$ belong to the same component of the link  $\widehat{\b}$, then $\Jtrs(\b e_{i,k}) = \Jtrs(\b)$. 
\end{cor}

\begin{proof}We have
$$
\Jtrs(\b e_{i,k}) = \Jtrs(\b \s_{k-1} \ldots \s_i e_i (\s_{k-1} \ldots \s_i)^{-1}) = \Jtrs((\s_{k-1} \ldots \s_i)^{-1} \b \s_{k-1} \ldots \s_i e_i).
$$
The $i$-th and $(i+1)$-st strands of the braid $(\s_{k-1} \ldots \s_i)^{-1} \b \s_{k-1} \ldots \s_i$ belong to the same component
of its closure (which is equivalent to $\widehat{\b}$), so by Proposition~\ref{trace_ei}, we obtain
$$
\Jtrs((\s_{k-1} \ldots \s_i)^{-1} \b \s_{k-1} \ldots \s_i e_i)= \Jtrs((\s_{k-1} \ldots \s_i)^{-1} \b \s_{k-1} \ldots \s_i )= \Jtrs(\b),
$$
as desired.
\end{proof}

The most important consequence of the special skein relation \eqref{new skein delta} is the following: let $L = \widehat{\a}$ be an arbitrary link, in the form of a closed braid. The braid $\a$ can be written as a product of a pure braid followed by a standard permutation braid. Moreover, by conjugation, this can be done so that each component of the link is composed of consecutive strands of the braid. Now, writing the pure braid as a word of the generators of the pure braid group makes recognizable the clasps between any two components of the link. Applying the special skein relation of Proposition~\ref{skein_ei} on a clasp, it will result in a diagram with a local Reidemeister II move and the diagram with the two different components merged into one, see Figure~\ref{fig_clasps}.

\begin{figure}[H]
\begin{center}
\begin{picture}(305,80)
\put(0,40){$\dfrac{1}{\sqrt{\lambda_D}}$}

\put(50,53){$\alpha$}

\qbezier(36,70)(36,78)(36,86)
\qbezier(55,70)(55,78)(55,86)
\qbezier(74,70)(74,78)(74,86)

\color[RGB]{0,179,51} \qbezier(35,20)(34,22)(35,24) 
\color[RGB]{153,0,230} \qbezier(55,20)(56,22)(55,24) 

\color[RGB]{0,0,0}

\qbezier(30,44)(30,59)(30,70)
\qbezier(80,44)(80,59)(80,70)

\qbezier(30,70)(55,70)(80,70)
\qbezier(30,44)(55,44)(80,44)

\color[RGB]{0,179,51} \qbezier(35,24)(35,29)(45,34)
\qbezier(45,34)(55,39)(55,44)
\color[RGB]{153,0,230} \qbezier(55,24)(55,28)(50,31)
\qbezier(40,37)(35,40)(35,44)

\qbezier(35,0)(35,5)(45,10)
\qbezier(45,10)(55,15)(55,20)
\color[RGB]{0,179,51} \qbezier(55,0)(55,4)(50,7)
\qbezier(40,13)(35,16)(35,20)

\color[RGB]{0,0,0}

\qbezier(74,0)(74,32)(74,44)

\put(90,40){$= \sqrt{\lambda_D}$}

\put(150,53){$\alpha$}

\qbezier(136,70)(136,78)(136,86) 
\qbezier(155,70)(155,78)(155,86)
\qbezier(174,70)(174,78)(174,86)

\color[RGB]{0,179,51} \qbezier(135,20)(134,22)(135,24) 
\color[RGB]{153,0,230} \qbezier(155,20)(156,22)(155,24) 

\color[RGB]{0,0,0}

\qbezier(130,44)(130,59)(130,70)
\qbezier(180,44)(180,59)(180,70)

\qbezier(130,70)(155,70)(180,70)
\qbezier(130,44)(155,44)(180,44)

\color[RGB]{0,179,51}\qbezier(135,24)(135,28)(140,31)
\qbezier(150,37)(155,40)(155,44)
\color[RGB]{173,0,230} \qbezier(155,24)(155,29)(145,34)
\qbezier(145,34)(135,39)(135,44)

\color[RGB]{0,0,0}

\color[RGB]{173,0,230} \qbezier(135,0)(135,5)(145,10)
\qbezier(145,10)(155,15)(155,20)
\color[RGB]{0,179,51} \qbezier(155,0)(155,4)(150,7)
\qbezier(140,13)(135,16)(135,20)

\color[RGB]{0,0,0}

\qbezier(174,0)(174,32)(174,44)

\put(190,40){$+ (q-q^{-1})$}
\put(270,53){$\alpha$}

\qbezier(256,70)(256,78)(256,86) 
\qbezier(275,70)(275,78)(275,86)
\qbezier(294,70)(294,78)(294,86)

\qbezier(250,44)(250,59)(250,70) 
\qbezier(300,44)(300,59)(300,70)

\qbezier(250,70)(275,70)(300,70)
\qbezier(250,44)(275,44)(300,44)

\color[RGB]{153,0,230} \qbezier(255,0)(255,5)(265,10)
\qbezier(265,10)(275,15)(275,20)
\qbezier(275,0)(275,4)(270,7)
\qbezier(260,13)(255,16)(255,20)

\qbezier(255,20)(255,32)(255,44) 
\qbezier(275,20)(275,32)(275,44)
\color[RGB]{0,0,0} \qbezier(294,0)(294,32)(294,44)

\end{picture}
\caption{Removing a clasp using the special skein relation.}
\label{fig_clasps}
\end{center}
\end{figure}
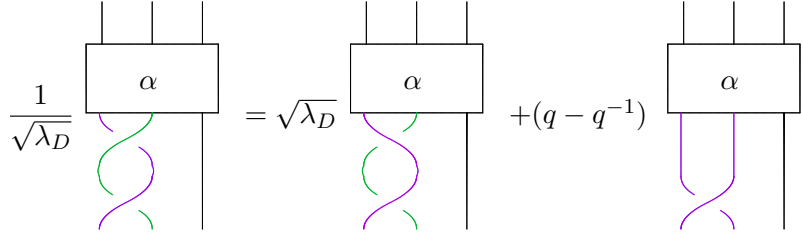

Then by induction on the number of clasps, we can decompose the initial link $L$ into sums of knots and disjoint unions of knots.

\subsection{Behaviour on $2$-component links} \label{behaviour2comp}
We now associate to every link $L$ a weighted graph $G_L$. A vertex of the graph corresponds to a component of the link, and an edge with weight $r$ corresponds to $r$ clasps  between the two components (either with positive or negative crossings). The edges of weight $0$ (which correspond to two unlinked components) are removed from the graph. For example, the graph \begin{picture}(32,12)\put(0,0){\usebox{\claspF}}\end{picture} describes a $2$-component link with $r$ clasps (see Figure~\ref{3clasps} for examples of links with $r=1$ and $r=3$). The clasps may or may not be consecutive. That is, the link cannot be recovered from the graph.

\begin{figure}[H]
\includegraphics[scale=0.15]{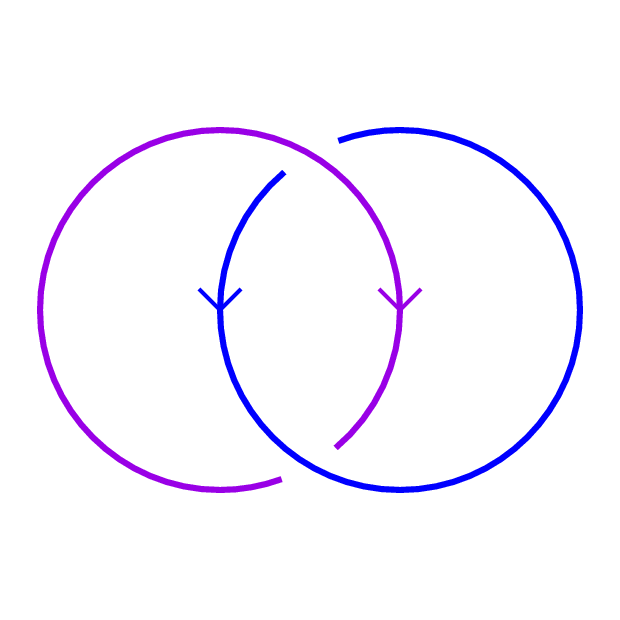} \qquad
\includegraphics[scale=0.45]{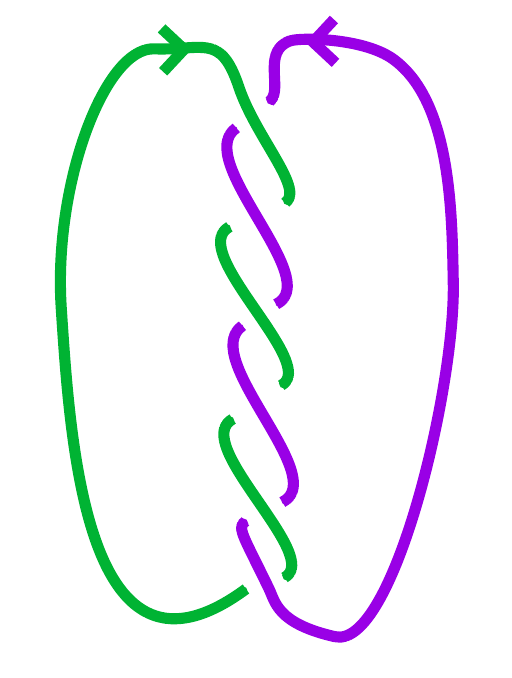}
\caption{The Hopf link \cite{chli} and a 2-component link with three negative clasps.}
\label{3clasps}
\end{figure}

Since the skein relation can be only used on crossings involving different components, it can be used to untie a component of a link. Then, by reducing to the case of $1$-component links, that is, the case of knots, we can compare the invariants $\Theta_d$ to the Homflypt polynomial $P$. 

First, we prove the following result for $2$-component links.

\begin{prop}\label{twocomp}
Let $L$ be a $2$-component link whose components are linked by $r$ clasps. 
Let $I$ be the disjoint union of these two components.
Then
$$
\Theta_d(L) = P(L) + \lambda_D^{\xi(L)} (E_D^{-1} - 1) P(I),
$$
or diagrammatically,
$$
\Theta_d(\text{\begin{picture}(32,12)\put(0,0){\usebox{\claspF}}\end{picture}}) = P(\text{\begin{picture}(32,12)\put(0,0){\usebox{\claspF}}\end{picture}}) + \lambda_D^{\xi(L)} (E_D^{-1} - 1) P(\bullet \, \bullet),
$$
where $\xi(L)$ is the sum of the signs of all clasps.
\end{prop}

\begin{proof}
We proceed by induction on $r$, the number of clasps. For $r=1$, we may have either a positive or negative clasp. If we apply the skein relation \eqref{new skein 2} on one crossing of the clasp, we get a new link with the clasp removed and a knot (since the two components will be joined):
$$
\Theta_d(L) = \lambda_D^{\pm 1} \Theta_d(I) \pm (q-q^{-1}) \sqrt{\lambda_D}^{\pm 1} \Theta_d(K),
$$
where $K$ is a knot.
Since $I$ is a disjoint union of $2$ knots and $K$ is a knot, by Theorem \ref{conjHomflypt} and Theorem \ref{disj_links}, we get
$$
\Theta_d(L) = \lambda_D^{\pm 1} E_D^{-1} P(I) \pm (q-q^{-1}) \sqrt{\lambda_D}^{\pm 1} P(K).
$$
Applying the skein relation of the Homflypt polynomial on $K$ yields
$$
\Theta_d(L) = \lambda_D^{\pm 1} E_D^{-1} P(I) + P(L) - \lambda_D^{\pm 1} P(I)=P(L) + \lambda_D^{\pm 1} (E_D^{-1} - 1) P(I).
$$

Assume now that the statement holds for $2$-component links with up to $r-1$ clasps. We have
$$
\Theta_d(L) = \lambda_D^{\pm 1} \Theta_d(\tilde{L}) \pm (q-q^{-1}) \sqrt{\lambda_D}^{\pm 1} \Theta_d(K),
$$
where $\tilde{L}$ is obtained from $L$ with the removal of one clasp and $K$ is a knot.
Following the induction hypothesis,we obtain
$$
\Theta_d(L) = \lambda_D^{\pm 1} \left[P(\tilde{L}) + \lambda_D^{\xi(\tilde{L})} (E_D^{-1} - 1) P(I) \right] \pm (q-q^{-1}) \sqrt{\lambda_D}^{\pm 1} \Theta_d(K).
$$
Since $\xi(L) = \xi(\tilde{L}) \pm 1$ and  $K$ is a knot, we have
$$\Theta_d(L)= \lambda_D^{\pm 1}P(\tilde{L}) + \lambda_D^{\xi(L)} (E_D^{-1} - 1) P(I) + P(L) - \lambda_D^{\pm 1} P(\tilde{L})=
P(L) + \lambda_D^{\xi(L)} (E_D^{-1} - 1) P(I),
$$
as desired.
\end{proof}

Since the invariants $\Theta_d$ are topologically equivalent to 
the Homflypt polynomial on knots, it would be expected that the skein relation of the Homflypt polynomial holds also for the invariants $\Theta_d$ when applied on crossings involving the same component. However, this is not true for the case of knots. Namely:

\begin{cor}\label{no_skein}
The skein relation of the Homflypt polynomial does not hold for the invariants $\Theta_d$ if applied on a knot.
\end{cor}

\begin{proof}
Let $K_+$ be a knot. Then $K_-$ is a knot and $K_0$ is a $2$-component link. Let $I$ be the disjoint union of the two components of $K_0$. Applying \eqref{skein of hom} we obtain:
\begin{align*}
&\frac{1}{\sqrt{\lambda_D}} P(K_+) - \sqrt{\lambda_D} P(K_-) = (q-q^{-1}) P(K_0)
\\
\stackrel{\text{Thm. }\ref{conjHomflypt}}{\Leftrightarrow}  &\frac{1}{\sqrt{\lambda_D}} \Theta_d(K_+) - \sqrt{\lambda_D} \Theta_d(K_-) = (q-q^{-1}) P(K_0) \\
\stackrel{\text{Prop. }\ref{twocomp}}{\Leftrightarrow} &\frac{1}{\sqrt{\lambda_D}} \Theta_d(K_+) - \sqrt{\lambda_D} \Theta_d(K_-) = (q-q^{-1}) \left[\Theta_d(K_0) - \lambda_D^{\xi(K_0)} (E_D^{-1} - 1) P(I) \right].
\end{align*}
Since $I$ is a disjoint union of two knots, by Theorem~\ref{disj_links}, we have $P(I) = E_D\, \Theta_d(I)$. Therefore, 
$$
\frac{1}{\sqrt{\lambda_D}} \Theta_d(K_+) - \sqrt{\lambda_D} \Theta_d(K_-) = (q-q^{-1}) \left[\Theta_d(K_0) + \lambda_D^{\xi(K_0)} (E_D - 1) \Theta_d(I) \right].
$$
\end{proof}

Due to Corollary~\ref{no_skein}, one cannot assume that the skein relation of the Homflypt polynomial holds for crossings of the same component on arbitrary links.

\subsection{General behaviour on links} \label{behaviourlinks}
We will now examine the behaviour of the invariants $\Theta_d$ on links with arbitrary number of components. Set
\begin{center}
$\mathcal{A}:=\mathbb{Q}[q^{\pm1},\sqrt{\lambda_D}^{\pm 1}]$. 
\end{center}

Using the skein relation \eqref{new skein delta} on a crossing involving two components, we get one link with the same number of components and another link with one component less. This result can be applied to links with an arbitrary number of components, in which case we obtain the following.

\begin{prop}\label{unlink_one}
Let $L$ be an $\ell$-component link. Let $K_1,K_2,\ldots,K_\ell$ denote the components of $L$ and let $i  \in \{1,\ldots,\ell\}$. Denote by ${L}^{(i)}$ the $\ell$-component link obtained from $L$ after unlinking $K_i$ from the rest of the link, that is,  ${L}^{(i)}$ is the disjoint union of $K_i$ and $L \setminus K_i$.
Then $\Theta_d(L)$  is an $\mathcal{A}$-linear combination of  $\Theta_d({L}^{(i)})$ and values of $\Theta_d$ on links with $\ell-1$ components. Further, this linear combination is produced by the skein relation. 
\end{prop}

\begin{proof} 
Suppose that $K_i$ is linked to $m_i$ other components  $K_{j_1}, \ldots, K_{j_{m_i}}$ with $1 \leq j_1, \ldots, j_{m_i} \leq \ell$. Let $r_1, \ldots, r_{m_i}$ be the number of clasps between $K_i$ and $K_{j_1}, \ldots, K_{j_{m_i}}$ respectively. For all $s=1,\ldots,m_i$ and $t=1,\ldots,r_s$,
denote by $\xi(j_s,t)$ the sign of the $t$-th clasp between $K_i$ and $K_{j_s}$.

If $m_i=0$, then $K_i$ is disjoint from the rest of the link. We have $L={L}^{(i)}$ and so the statement holds automatically.
Let $m_i>0$. First, we try to unlink $K_i$ from the rest of the link by resolving the clasps between $K_i$ and $K_{j_1}$ via the skein relation. As previously, applying the skein relation results in two new links: one with $r_1-1$ clasps between  $K_i$ and $K_{j_1}$, and one with the two components $K_i$ and $K_{j_1}$ joined. We apply now the skein relation to the former link.
Again, we obtain two new links: one with $r_1-2$ clasps between  $K_i$ and $K_{j_1}$, and one with the two components $K_i$ and $K_{j_1}$ joined.
 We continue the same procedure, until we have applied the skein relation $r_1$ times and we are left with:
 \begin{enumerate}[(a)]
\item  $r_1$ knots formed by the union of  $K_i$ and $K_{j_1}$, and so $r_1$ links with $\ell-1$ components, and \smallbreak
\item  one disjoint union of $K_i$ and $K_{j_1}$, and so the link obtained from $L$ after unlinking $K_i$ from $K_1$ (which is also a link with $\ell$ components).
\end{enumerate}
Diagrammatically, we have the following skein tree after we apply the skein relation $r_1$ times:
\begin{figure}[H]
\begin{picture}(300,220)
\put(182.5,208){\usebox{\claspA}}

\put(225,200){\vector(3,-4){20}}
\put(195,200){\vector(-3,-4){20}}
\put(160,185){\tiny $\lambda_D^{\xi(j_1, 1)}$}
\put(240,185){\tiny $\xi(j_1,1) (q-q^{-1}) \sqrt{\lambda_D}^{\xi(j_1,1)}$}

\put(140,155){\usebox{\claspB}}
\put(235,155){$\usebox{\claspE}_1^{(i,j_1)}$}

\put(180,150){\vector(3,-4){20}}
\put(155,150){\vector(-3,-4){20}}
\put(120,135){\tiny $\lambda_D^{\xi(j_1,2)}$}
\put(195,135){\tiny $\xi(j_1,2) (q-q^{-1}) \sqrt{\lambda_D}^{\xi(j_1,2)}$}

\put(130,105){$\vdots$}
\put(190,105){$\usebox{\claspE}_2^{(i,j_1)}$}

\put(145,100){\vector(3,-4){20}}
\put(120,100){\vector(-3,-4){20}}
\put(68,85){\tiny $\lambda_D^{\xi(j_1,r_1-1)}$}
\put(160,85){\tiny $\xi(j_1,r_1-1) (q-q^{-1}) \sqrt{\lambda_D}^{\xi(j_1,r_1-1)}$}

\put(65,55){\usebox{\claspC}}
\put(155,55){$\usebox{\claspE}_{r_1-1}^{(i,j_1)}$}

\put(105,50){\vector(3,-4){20}}
\put(80,50){\vector(-3,-4){20}}
\put(35,35){\tiny $\lambda_D^{\xi(j_1,r_1)}$}
\put(120,35){\tiny $\xi(j_1,r_1) (q-q^{-1}) \sqrt{\lambda_D}^{\xi(j_1,r_1)}$}

\put(35,5){\usebox{\claspD}}
\put(115,5){$\usebox{\claspE}_{r_1}^{(i,j_1)}$}
\end{picture}
\end{figure}
$ $\\
where 
\begin{picture}(48,12)\put(0,0){$\usebox{\claspE}_t^{(i, j_1)}$}\end{picture}\,, for $t=1,\ldots, r_1$, denotes the link  with $\ell-1$ components
which is produced by joining the components $K_i$ and $K_{j_1}$ when applying 
the skein relation on the clasp $t$ (the edges adjacent to the vertex are due to the fact that this component may be linked to more components). Hence, we obtain
$$\Theta_d(L)  = \lambda_D^{\eta(j_1,r_1)} \Theta_d(\text{\begin{picture}(40,12)\put(0,0){\usebox{\claspD}}\end{picture}}) + (q-q^{-1}) \sum_{t=1}^{r_1}\xi(j_1,t) \lambda_D^{\eta(j_1,t-1)+\xi(j_1,t)/2}\Theta_d \left( \text{\begin{picture}(50,12)\put(0,0){$\usebox{\claspE}_t^{(i, j_1)}$}\end{picture}}\right),$$
where $\eta(j_1,0):=0$ and $\eta(j_1,t):=\sum_{j=1}^{t}\xi(j_1,j)$ for all $t=1,\ldots,r_1$. Setting
$$\xi(j_1):=\eta(j_1,r_1) \quad \text{and}\quad \Xi_D(j_1,t):=\xi(j_1,t)\lambda_D^{\eta(j_1,t-1)+\xi(j_1,t)/2} \quad \text{ for all } t=1,\ldots,r_1,$$ we obtain
$$
\Theta_d(L) = \lambda_D^{\xi(j_1)} \Theta_d(\text{\begin{picture}(40,12)\put(0,0){\usebox{\claspD}}\end{picture}}) + (q-q^{-1}) \sum_{t=1}^{r_1} \Xi_D(j_1,t) \Theta_d \left( \text{\begin{picture}(50,12)\put(0,0){$\usebox{\claspE}_t^{(i, j_1)}$}\end{picture}} \right).
$$

Next, we apply the same procedure on the clasps between the components $K_i$ and $K_{j_2}, \ldots, K_{j_{m_i}}$.  We deduce that
\begin{equation}\label{THELINEAR}
\Theta_d(L) = \mu_D^{(m_i)}  \Theta_d(L^{(i)})
+ (q-q^{-1}) \sum_{s=1}^{m_i} \left( \mu_D^{(s-1)}  \sum_{t=1}^{r_s} \Xi_D(j_s,t) \Theta_d \left( \text{\begin{picture}(50,12)\put(0,0){$\usebox{\claspE}_t^{(i, j_s)}$}\end{picture}} \right) \right),
\end{equation}
where $\mu_D^{(0)}:=1$ and $\mu_D^{(s)} := \lambda_D^{\xi(j_1) + \cdots + \xi(j_s)}$ for all $s=1,\ldots,m_i$. 
\end{proof}

We can now apply Proposition~\ref{unlink_one} to the $\ell$-component link $L^{(i)}$ and all $(\ell-1)$-component links appearing in \eqref{THELINEAR}. We can repeat the procedure until $\Theta_d(L)$ is written as an  $\mathcal{A}$-linear combination of  values of $\Theta_d$ on disjoint unions of knots with up to $\ell$ components. 
For $k=1,\ldots,\ell$, let $\mathcal{N}(L)_k$ denote the set of all disjoint unions of $k$ knots appearing in this linear combination.
We thus have
\begin{equation}\label{c_D}
\Theta_d(L) = \sum_{k=1}^{\ell}\,\,\sum_{\widehat{\a} \in \mathcal{N}(L)_k} c(\widehat{\a}) \,\Theta_d(\widehat{\a})
\end{equation}
for some $c(\widehat{\a}) \in \mathcal{A}$. 
Since the above linear combination is produced only with the use of the skein relation, and the skein relation is the same as the one of the Homflypt polynomial $P=P(q,\lambda_D)$, we also have
\begin{equation}\label{c_P}
P(L) = \sum_{k=1}^{\ell}\,\,\sum_{\widehat{\a} \in \mathcal{N}(L)_k} c(\widehat{\a}) \,P(\widehat{\a}).
\end{equation}
The above analysis culminates to the following, which is the main result of this section.

\begin{thm}\label{links}
For any $\ell$-component link $L$, the value $\Theta_d(L)$ is an $\mathcal{A}$-linear combination of $P(L)$ and the values of $P$ on disjoint unions of knots obtained by the skein relation:
$$
\Theta_d(L) =  \sum_{k=1}^{\ell} E_D^{1-k} \sum_{\widehat{\a} \in \mathcal{N}(L)_k} c(\widehat{\a}) \,P(\widehat{\a})    =P(L) + \sum_{k=2}^{\ell} \,\,(E_D^{1-k}-1) \sum_{\widehat{\a} \in \mathcal{N}(L)_k} c(\widehat{\a}) \,P(\widehat{\a}). 
$$
Conversely, the value $P(L)$ is an  $\mathcal{A}$-linear 
combination of $\Theta_d(L)$ and the values of $\Theta_d$   
on disjoint unions of knots obtained by the skein relation:
$$
P(L) =  \sum_{k=2}^{\ell}\,\,E_D^{k-1}\sum_{\widehat{\a} \in \mathcal{N}(L)_k} c(\widehat{\a}) \,\Theta_d(\widehat{\a})=\Theta_d(L) + \sum_{k=2}^{\ell}\,\,(E_D^{k-1}-1)\sum_{\widehat{\a} \in \mathcal{N}(L)_k} c(\widehat{\a}) \,\Theta_d(\widehat{\a}).
$$
\end{thm}

\begin{proof}
Both statements follow immediately by \eqref{c_D}, \eqref{c_P} and Theorem~\ref{disj_links}. Note that for $k=1$ the corresponding term 
is zero.
\end{proof}

\begin{rem}\label{remark_disjoint} \rm
Note that the only element in $\mathcal{N}(L)_{\ell}$ is the disjoint union of all components of $L$, that is, $K_1 \sqcup \ldots \sqcup K_{\ell}$. 
We have $c(K_1 \sqcup \ldots \sqcup K_{\ell}) = \mu_D^{(m_1)}\ldots \mu_D^{(m_\ell)}$.
\end{rem}

In fact, Theorem~\ref{links} can be generalized to obtaining any of the invariants $\Theta_{d'}$ by any other invariant $\Theta_d$.
Following the established notation, the set $D$ still denotes $\Z/d\Z$ and the set $D'$ will now be $\Z/d'\Z$, for $d,d' \in \N$.
More precisely, we have the following:

\begin{thm}\label{delta_invs}
Let $d,d' \in \N$.
For any $\ell$-component link $L$,  the value $\Theta_{d'}(L)$  is an $\mathcal{A}$-linear 
combination of $\Theta_d(L)$ and the values of $\Theta_d$ 
on disjoint unions of knots obtained by the skein relation:
$$
\Theta_{d'}(L) = \Theta_d(L) + \sum_{k=2}^{\ell}\, \left( \left( \frac{E_D}{E_{D'}} \right)^{k-1} - 1 \right)\,\sum_{\widehat{\a} \in \mathcal{N}(L)_k} c(\widehat{\a}) \,\Theta_d(\widehat{\a}).
$$
\end{thm}

\begin{proof} 
Following Theorem~\ref{links}, we have
$$\Theta_{d'}(L) - \Theta_d(L) = \sum_{k=2}^{\ell} \,\,(E_{D'}^{1-k}-E_D^{1-k}) \sum_{\widehat{\a} \in \mathcal{N}(L)_k} c(\widehat{\a}) \,P(\widehat{\a}) .$$
Now, note that, for all $k=2,\ldots,\ell$,
$$E_{D'}^{1-k}-E_D^{1-k} =\left(\left( \frac{E_{D'}}{E_D}\right)^{1-k}-1\right) E_D^{1-k} = \left(\left( \frac{E_{D}}{E_{D'}}\right)^{k-1}-1\right) E_D^{1-k}.$$
By Theorem~\ref{disj_links}, we have
$$E_D^{1-k} \sum_{\widehat{\a} \in \mathcal{N}(L)_k} c(\widehat{\a}) \,P(\widehat{\a}) =\sum_{\widehat{\a} \in \mathcal{N}(L)_k} c(\widehat{\a}) \,E_D^{1-k}  P(\widehat{\a})=
 \sum_{\widehat{\a} \in \mathcal{N}(L)_k} c(\widehat{\a}) \,\Theta_d(\widehat{\a}).$$
We conclude that
$$
\Theta_{d'}(L) = \Theta_d(L) + \sum_{k=2}^{\ell}\, \left( \left( \frac{E_D}{E_{D'}} \right)^{k-1} - 1 \right)\,\sum_{\widehat{\a} \in \mathcal{N}(L)_k} c(\widehat{\a}) \,\Theta_d(\widehat{\a}).
$$
\end{proof}

\subsection{Defining $\Theta_d$ diagrammatically via skein relations}\label{subs_skein}
Let $L$ be an oriented link. By the special skein relation of Proposition~\ref{skein_ei} the value of $\Theta_d$ on $L$ at variables $(q,\lambda_D)$ can be computed  diagrammatically by applying the following procedure:
\begin{enumerate}[\bf Step 1.]
\item Apply the skein relation of Proposition~\ref{skein_ei} on crossings linking different components until the link $L$ is decomposed into disjoint unions of knots.
An algorithmic process for achieving this, different from the one described in Subsection~\ref{special_skein}, is the following: we order the components of $L$ and we select a starting point on each component. Starting from the chosen point of the first component and following its orientation we apply the skein relation on all mixed crossings we encounter, so that the arcs of this component are always overarcs. We proceed similarly with the second component changing all mixed crossing except for crossings involving the first component, and so on. In the end we obtain the split version of the original link.
 \smallbreak
\item Following Theorem~\ref{links} and its notation, we obtain
$$
\Theta_d(L) =  \sum_{k=1}^{\ell} E_D^{1-k} \sum_{\widehat{\a} \in \mathcal{N}(L)_k} c(\widehat{\a}) \,P(\widehat{\a}).
$$   \smallbreak
\item Apply the skein relation \eqref{skein of hom} of the Homflypt polynomial to obtain the value of $P$ on $\widehat{\a}$ at variables $(q,\lambda_D)$, for all disjoint unions of knots $\widehat{\a} \in \mathcal{N}(L)_k$, $k=1,\ldots,\ell$.
\end{enumerate}

Note that the special skein relation of Proposition~\ref{skein_ei} for $\Theta_d$ at variables $(q,\lambda_D)$ is the same as the skein relation of the Homflypt polynomial \eqref{skein of hom} at variables $(q,\lambda_D)$, but it can only be applied to  crossings between different components. So, since the invariants $\Theta_d$ are well-defined via braid methods, we obtain the following:

\begin{thm}\label{skein_defined}
The invariants $\Theta_d$ can be completely defined via the Homflypt skein relation (with the use of Theorem~\ref{links}).
\end{thm}

\begin{rem}\label{theta LM} \rm
In this paper the fact that the invariants $\Theta_d$ are well-defined is established by their algebraic construction. However, one could also prove it via the diagrammatic method of Lickorish--Millett for proving that the Homflypt polynomial is well-defined  \cite{LM}. Namely, we start as in the algorithm above for evaluating the invariants $\Theta_d$, that is, we order the components of the link and we select a point on each component. Then we have to show that the evaluation does not depend on the ordering of the components, on the choice of starting points, on the Reidemeister moves or on the crossings on which the skein relation is applied, and this is done by employing inductive arguments. 
\end{rem}

Theorem~\ref{skein_defined} renders the invariants $\Theta_d$ much more easily computable than via the braid approach. Further, as we shall see next, it enables us to make a direct comparison of the invariants $\Theta_d$ with the Homflypt polynomial on links.

\begin{rem} \rm \label{oldquadr7}
Theorem~\ref{skein_defined}, as well as the preceding results of Subsections~\ref{behaviour2comp} and \ref{behaviourlinks}, are not vaild for the invariants $\Delta_d$, due to the lack of a special skein relation for classical links.
\end{rem}

\section{Comparison of the invariants $\Theta_d$ with $P$ on classical links}\label{distinguish}
In this section we prove that the invariants $\Theta_d$ are not topologically equivalent to the Homflypt polynomial.

Let $L_1$ and $L_2$ be two links. 
We will say that $L_1$ and $L_2$ are \textit{$\Theta_d$-equivalent} (respectively \textit{$P$-equivalent}) if $\Theta_d(L_1) = \Theta_d(L_2)$ (respectively $P(L_1) = P(L_2)$). 
Theorem~\ref{links} tells us that the values of the invariants $\Theta_d$ depend on values of $P$. This could be taken a strong lead that the invariants $\Theta_d$ are topologically equivalent to the Homflypt polynomial. On the other hand, again by Theorem~\ref{links}, two $P$-equivalent links may not be necessarily $\Theta_d$-equivalent. Similarly, by Theorem~\ref{delta_invs}, a $\Theta_d$-equivalent pair may not be $\Theta_{d'}$-equivalent if
$d' \neq d$.
Yet, for the case of 2-component links we prove the following: Let $L_1$, $L_2$ be two $2$-component $P$-equivalent links. Then 
$L_1$ and $L_2$ are distinguished by some $\Theta_d$ if and only if they are distinguished by any $\Theta_{d'}$ for
$d,d' \geq 2$.
However, this result cannot be generalized to links with more components.

 Afterwards, we investigate computationally whether any $P$-equivalent pairs of links can be distinguished by the invariants $\Theta_d$. Our computations (using the program mentioned in Remark~\ref{new_program}) showed that six pairs of $P$-equivalent $3$-component links are distinguished by the invariants $\Theta_d$. We present these pairs of links and we give the values of $\Theta_d$ on them in Appendix~\ref{computations}. 
Further, using Theorem~\ref{skein_defined}, for one of the six pairs we give a complete diagrammatic proof that they are not $\Theta_d$-equivalent for every
$d \geq 2$.
This implies that the invariants $\Theta_d$ are not topologically equivalent to the Homflypt polynomial $P$ on links, which is the main result of this paper. 

\subsection{Comparison of the invariants $\Theta_d$ on $P$-equivalent $2$-component links}
We shall now consider the set of links with two components. 
\begin{thm}\label{2comp}
Let $d,d' \geq 2$
and let $L_1$ and $L_2$ be a pair of  $2$-component $P$-equivalent links.
Then $L_1$ and $L_2$ are $\Theta_d$-equivalent if and only if they are $\Theta_{d'}$-equivalent.
\end{thm}

\begin{proof}
By Theorem~\ref{links}, we have 
$$
\Theta_d(L_1) = P(L_1) + \sum_{\widehat{\a_1} \in \mathcal{N}(L_1)_2} (E_D^{-1}-1) \, c(\widehat{\a_1}) P(\widehat{\a_1}).
$$
and
$$
\Theta_d(L_2) = P(L_2) + \sum_{\widehat{\a_2} \in \mathcal{N}(L_2)_2} (E_D^{-1}-1) \, c(\widehat{\a_2}) P(\widehat{\a_2}).
$$
Hence,
$$
\Theta_d(L_1) - \Theta_d(L_2) = (E_D^{-1}-1) \underbrace{\left[ \sum_{\widehat{\a_1} \in \mathcal{N}(L_1)_2} c(\widehat{\a_1}) P(\widehat{\a_1}) - \sum_{\widehat{\a_2} \in \mathcal{N}(L_2)_2} c(\widehat{\a_2}) P(\widehat{\a_2}) \right]}_{=: T_2(L_1,L_2)}.
$$
If $d \neq 1$, we have $E_D \neq 1$, and so $\Theta_d(L_1) =\Theta_d(L_2)$ if and only if $T_2(L_1,L_2)=0$. Since  $T_2(L_1,L_2)$ is independent of the choice of $d$,
we obtain the desired result.
\end{proof}

\begin{rem} \rm
For links with more than two components, it is not easy to derive directly a statement similar  to that of Theorem~\ref{2comp}. We will demonstrate this by considering $3$-component links: let $L_1$ and $L_2$ be two 3-component $P$-equivalent links. Then, by Theorem~\ref{links}, we obtain
{\small
$$
\Theta_d(L_1) - \Theta_d(L_2) = (E_D^{-1}-1) \Bigg[ T_2(L_1,L_2) + (E_D^{-1} + 1) \underbrace{\left(\sum_{\widehat{\a_1} \in \mathcal{N}(L_1)_3} c(\widehat{\a_1}) P(\widehat{\a_1}) - \sum_{\widehat{\a_2} \in \mathcal{N}(L_2)_3} c(\widehat{\a_2}) P(\widehat{\a_2})\right)}_{:=T_3(L_1,L_2)} \Bigg].
$$}
Now the quantity $T_2(L_1,L_2) + (E_D^{-1} + 1) T_3(L_1,L_2)$ involves $E_D$, so it depends on the choice of $d$.
Hence, we cannot use the same argument as in the proof of Theorem~\ref{2comp}.
\end{rem}

\subsection{Computations on specific $P$-equivalent links}\label{sub-comp}
As mentioned in Remark~\ref{new_program}, a computer program has been developed on Mathematica, which calculates the trace $\Jtrs$
by using only the five rules of Theorem~\ref{specialtrace} and the new quadratic relation. This program has been used to compute the invariants $\Theta_d$ on all 2.978 knots up to 12 crossings and their mirror images, in order to confirm computationally Theorem~\ref{conjHomflypt} and the mirroring property of $\Theta_d$ \cite[Proposition 3.6]{chjakala}. The data were obtained by {\it KnotInfo} \cite{chli}.

Now, out of 4.188 links (with up to 11 crossings),
there are 89 pairs of $P$-equivalent links which do not differ only by orientation, that is, they are different links if considered as unoriented links. Using the data from {\it LinkInfo} \cite{chli}, we computed the invariants $\Theta_d$ on all of them. Out of these 89 $P$-equivalent pairs of links, 83 are still $\Theta_d$-equivalent for generic $d$,
 yet we found that the following \textit{six pairs of $3$-component $P$-equivalent links are not $\Theta_d$-equivalent for every $d \geq 2$}:
\begin{table}[H]
\centering
{\renewcommand{\arraystretch}{1.2}
\begin{tabular}{|c|c|}
\hline
$L11n358\{0,1\}$ & $L11n418\{0,0\}$\\ \hline
$L11a467\{0,1\}$ & $L11a527\{0,0\}$\\ \hline
$L11n325\{1,1\}$ & $L11n424\{0,0\}$\\ \hline
$L10n79\{1,1\}$ & $L10n95\{1,0\}$\\ \hline
$L11a404\{1,1\}$ & $L11a428\{0,1\}$\\ \hline
$L10n76\{1,1\}$ & $L11n425\{1,0\}$\\ \hline
\end{tabular}
}
\caption{Six $P$-equivalent pairs of 3-component links which are not $\Theta_d$-equivalent.}
\label{pairs}
\end{table}

Specifically, for these pairs we computed the differences of the polynomials:

{\small \begin{align*}
&\Theta_d(L11n358\{0,1\})-\Theta_d(L11n418\{0,0\}) = \frac{(E_D-1) (\lambda_D -1) (q-1)^2 (q+1)^2 \left(q^2-\lambda_D \right)
   \left(\lambda_D  q^2-1\right)}{E_D \lambda_D^4 q^4},
\\
&\Theta_d(L11a467\{0,1\})-\Theta_d(L11a527\{0,0\}) = \frac{(E_D-1) (\lambda_D -1) (q-1)^2 (q+1)^2 \left(q^2-\lambda_D \right)
   \left(\lambda_D q^2-1\right)}{E_D \lambda_D^4 q^4},
\\
&\Theta_d(L11n325\{1,1\})-\Theta_d(L11n424\{0,0\}) = -\frac{(E_D-1) (\lambda_D -1) (q-1)^2 (q+1)^2 \left(q^2-\lambda_D \right)
   \left(\lambda_D q^2-1\right)}{E_D \lambda_D ^3 q^4},
\\
&\Theta_d(L10n79\{1,1\})-\Theta_d(L10n95\{1,0\}) = \frac{(E_D-1) (\lambda_D -1) (q-1)^2 (q+1)^2 \left(\lambda_D +\lambda_D
   q^4+\lambda_D  q^2-q^2\right)}{E_D \lambda_D^4 q^4},
\\
&\Theta_d(L11a404\{1,1\})-\Theta_d(L11a428\{0,1\}) = \frac{(E_D-1) (\lambda_D -1) (\lambda_D+1) (q-1)^2 (q+1)^2
   \left(q^4-\lambda_D  q^2+1\right)}{E_D q^4},
\\  
&\Theta_d(L10n76\{1,1\})-\Theta_d(L11n425\{1,0\}) =  \frac{(E_D-1) (\lambda_D -1) (\lambda_D+1) (q-1)^2 (q+1)^2}{E_D \lambda_D^3 q^2}.
\end{align*}}

Note that the factor $(E_D - 1)$ is common to all six pairs. This confirms that the pairs have the same Homflypt polynomial, since for $E_D = 1$ the difference collapses to zero. For a comprehensive list of the links of Table~\ref{pairs} with their braid words and the values of $\Theta_d$ on them, see Appendix~\ref{computations}. Further, all the computations can be found on \url{http://www.math.ntua.gr/~sofia/yokonuma}.

\subsection{Distinguishing diagrammatically the links $L11n358\{0,1\}$ and $L11n418\{0,0\}$}
 We have seen that the value of $\Theta_d$ at variables $(q,\lambda_D)$ can be computed also diagrammatically by applying the procedure
described in Subsection \ref{subs_skein} (Theorem~\ref{skein_defined}).
Notice that the difference between calculating $\Theta_d$ and $P$ on a link diagram using this procedure lies exclusively in Step 2, since Steps 1 and 3 involve only the skein relation of the Homflypt polynomial.
In Step 2 a coefficient which depends on $E_D$ appears when calculating $\Theta_d$; this does not exist for $P$. Further, the difference of the invariants $\Theta_d$ on a $P$-equivalent pair of links depends only on the values of the disjoint unions of knots that appear after we apply the special skein relation on both links, due to Theorem~\ref{links}. Hence, if two $P$-equivalent links decompose into the same disjoint unions of knots with the same coefficients under the special skein relation, then they are $\Theta_d$-equivalent.
If, however, they have different decompositions and they are not $\Theta_d$-equivalent, then this may only be caught in Step 2. 

\bigbreak
We now state the main result of this paper:

\begin{thm}\label{ThetaNotP}
The invariants $\Theta_d$ are not topologically equivalent to the Homflypt polynomial for any $d \geq 2$.
\end{thm}

\begin{proof}
Given our computational data, it suffices to prove that the invariants $\Theta_d$ distinguish one pair of links from Table~\ref{pairs}. We now prove diagrammatically, using Theorem~\ref{skein_defined}, that the links $L11n358\{0,1\}$ and $L11n418\{0,0\}$ are not $\Theta_d$-equivalent for every $d \geq 2$.
The diagrams in Figure~\ref{fig_link_diag} (obtained from \cite{chli}) illustrate the above links:
\begin{figure}[H]
\includegraphics[scale=0.15]{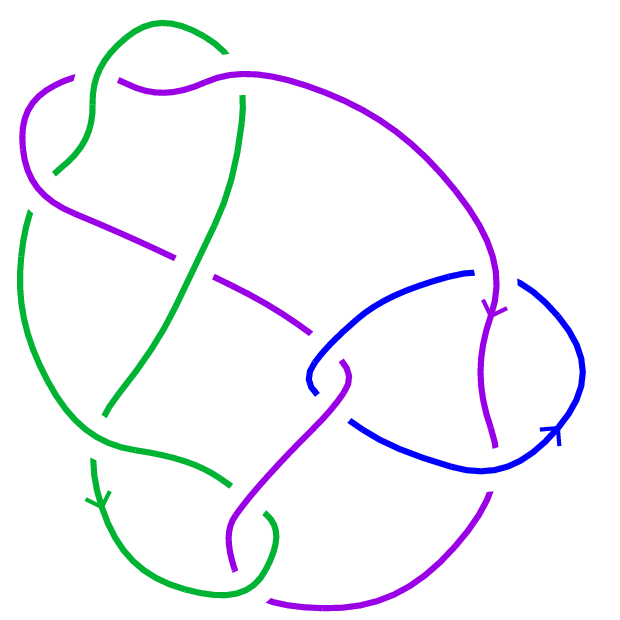} \qquad \qquad
\includegraphics[scale=0.15]{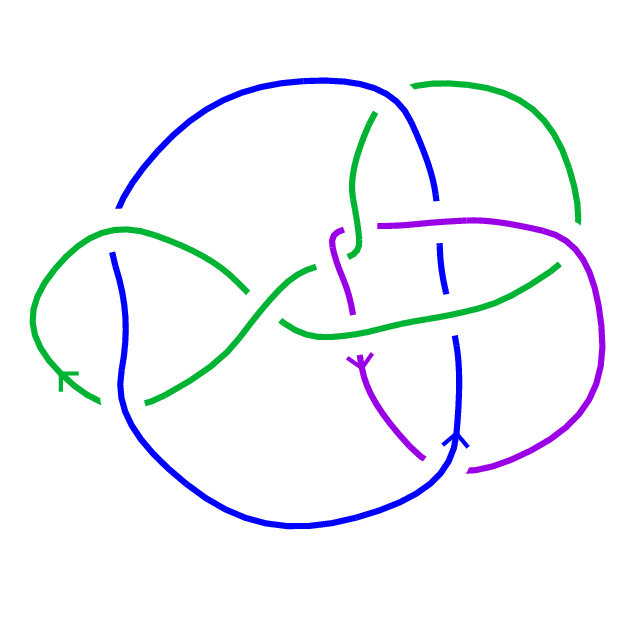}
\caption{The links $L11n358\{0,1\}$ and $L11n418\{0,0\}$.}
\label{fig_link_diag}
\end{figure}

Denote by $U$ the unknot and by $L^*$ the mirror image of a link $L$. Applying the skein relation \eqref{new skein delta} (see also Remark~\ref{skein_pm2}) on each diagram resolves the diagram into disjoint unions of knots. These knots need to be identified. 
The isotopy class of a knot diagram can be easily identified either by comparison with known diagrams or by calculating its Homflypt polynomial; the latter can be done either directly from the diagram or by using a braid presentation obtained by the knot diagram.

Beginning with the diagram of the link $L11n358\{0,1\}$ we obtain the skein tree illustrated in Figure~\ref{skein_tree_358}:

\begin{figure}[H]
\begin{picture}(380,330)

\put(95,230){\includegraphics[scale=0.25]{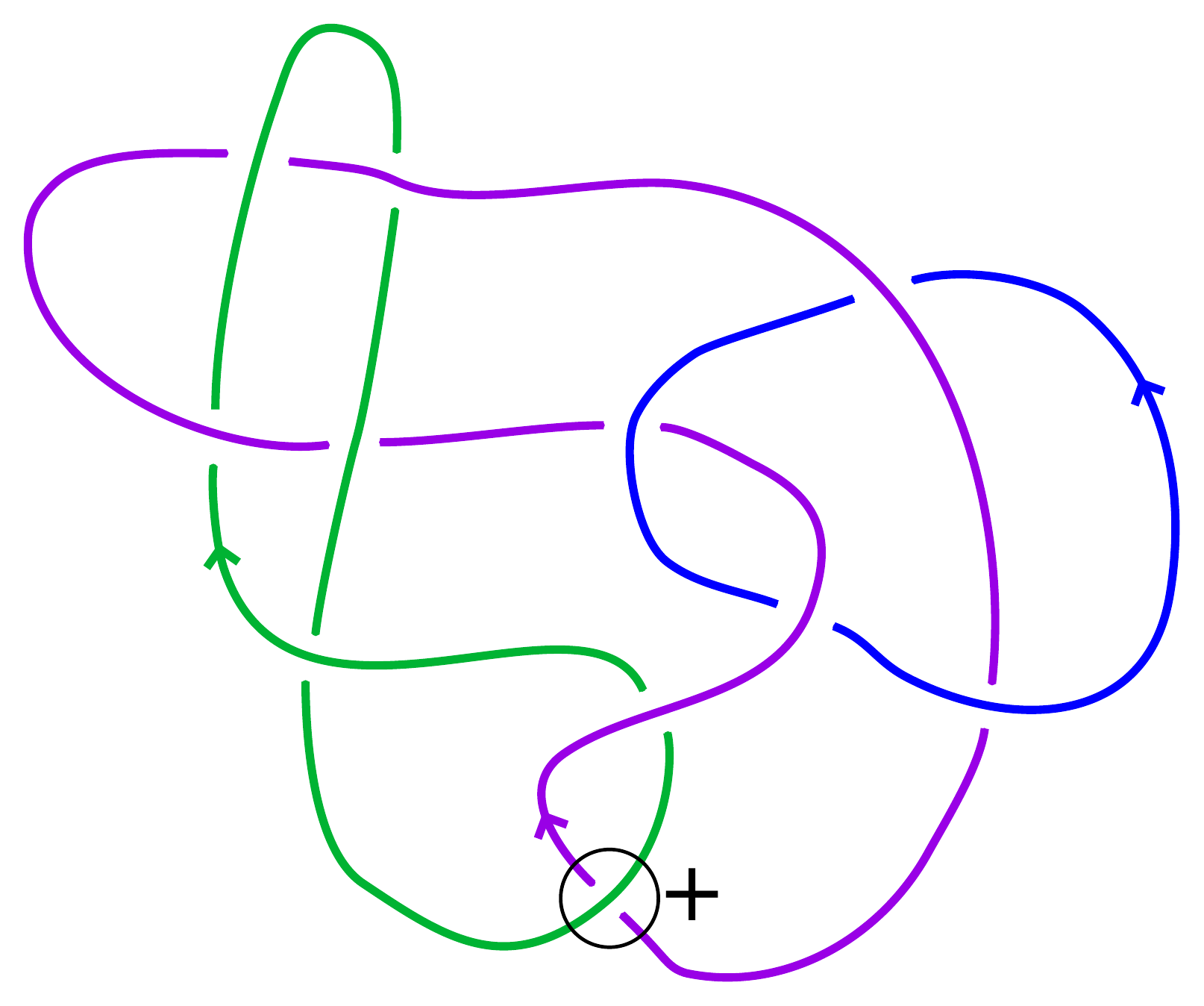}}
\put(215,270){\vector(3,-4){20}}
\put(230,260){$(q-q^{-1}) \sqrt{\lambda_D}$}
\put(110,270){\vector(-3,-4){20}}
\put(85,260){$\lambda_D$}

\put(0,180){\includegraphics[scale=0.25]{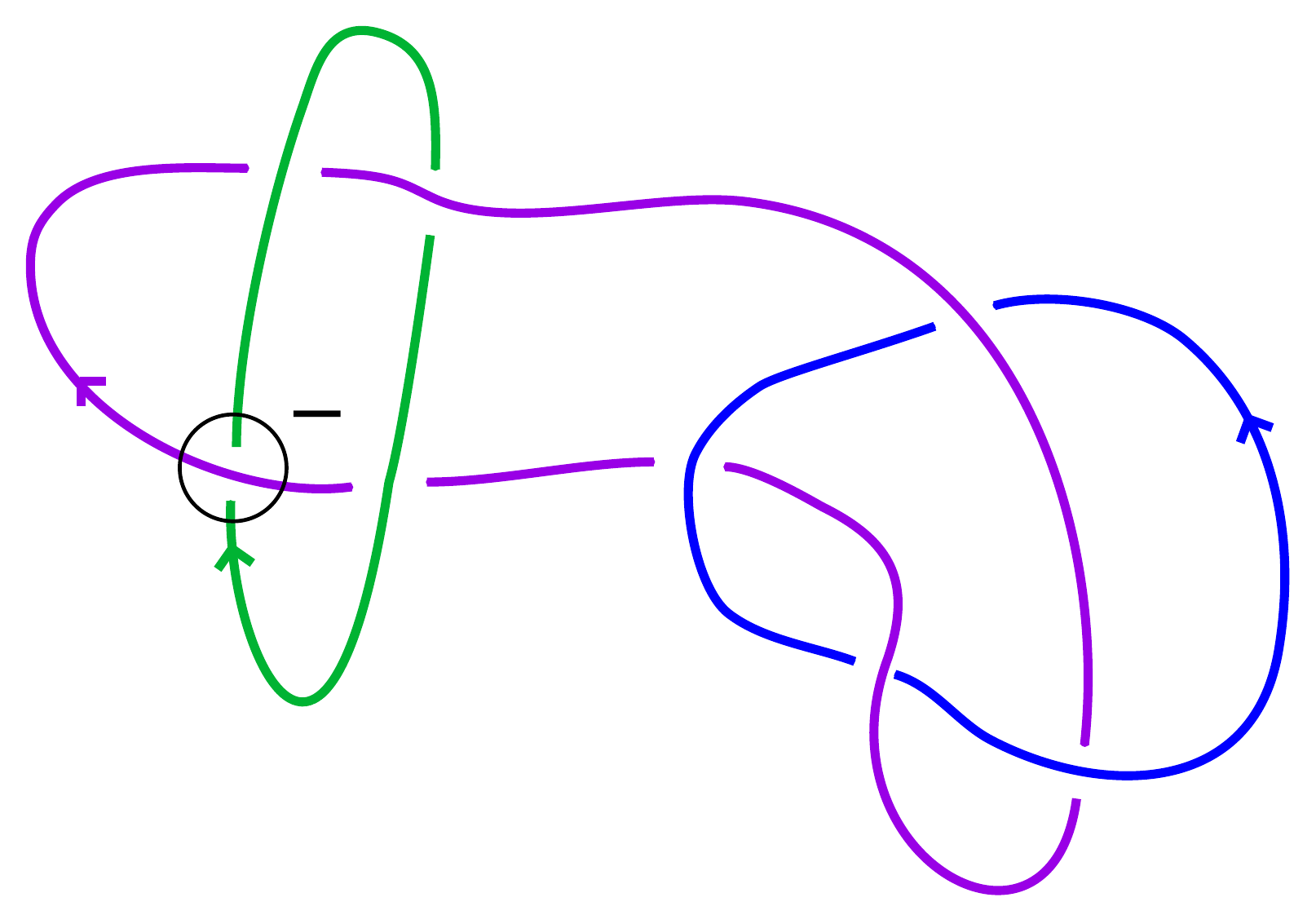}}
\put(190,165){\includegraphics[scale=0.25]{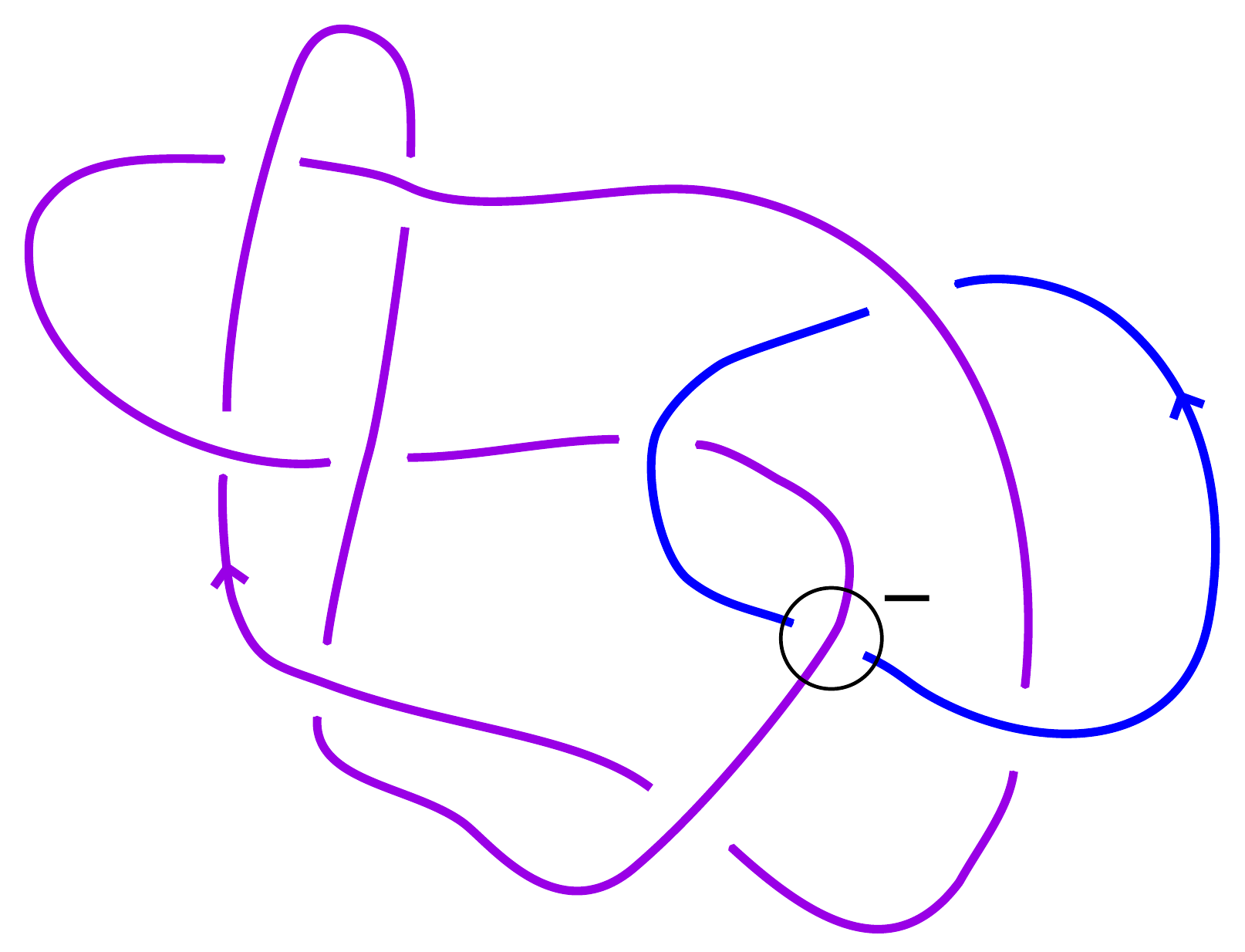}}
\put(40,200){=: $L_1$}
\put(290,180){\vector(0,-4){28}}
\put(295,165){$-\frac{(q-q^{-1})}{\sqrt{\lambda_D}}$}
\put(215,185){\vector(-3,-2){40}}
\put(180,175){$\lambda_D^{-1}$}

\put(70,90){\includegraphics[scale=0.25]{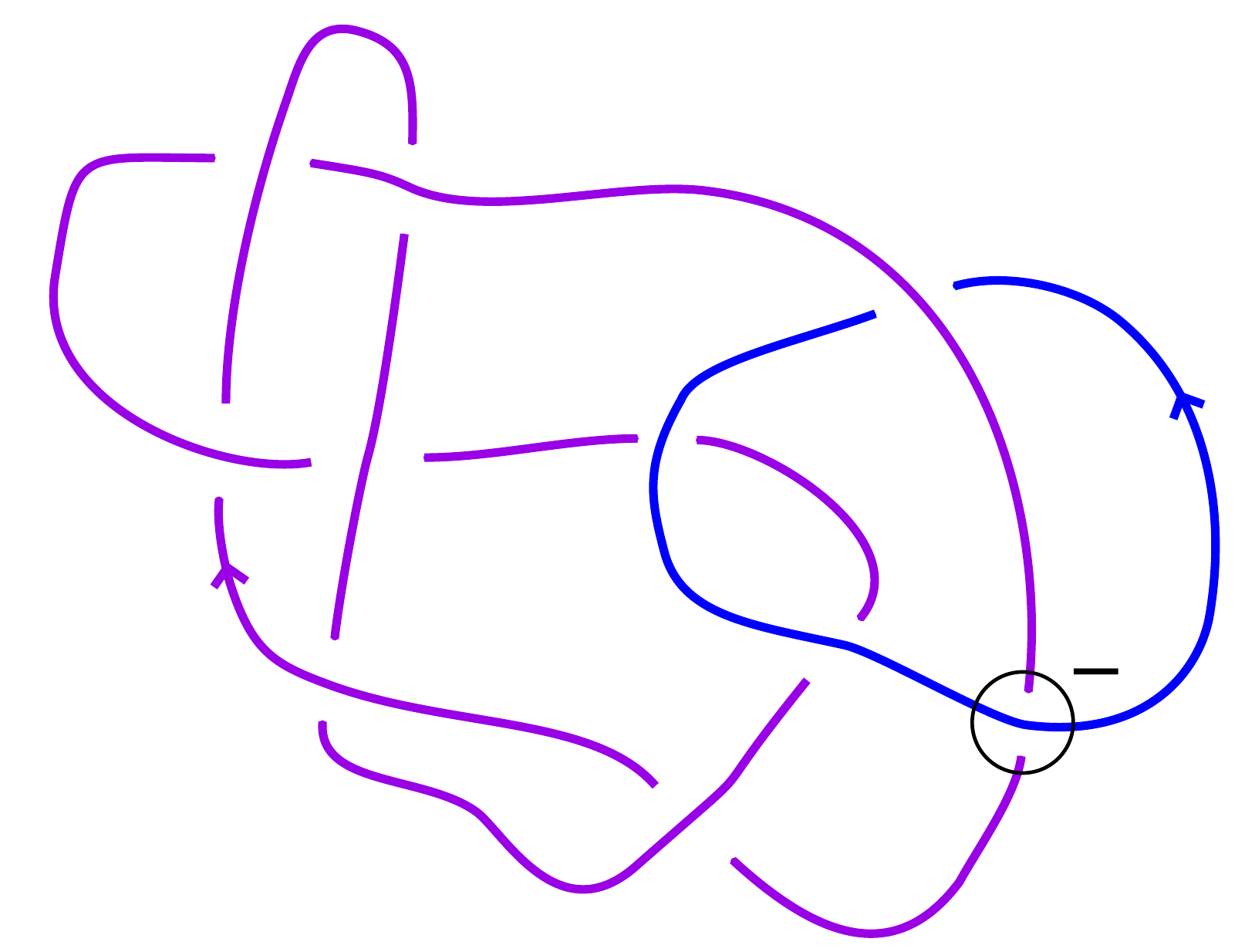}}
\put(240,70){\includegraphics[scale=0.25]{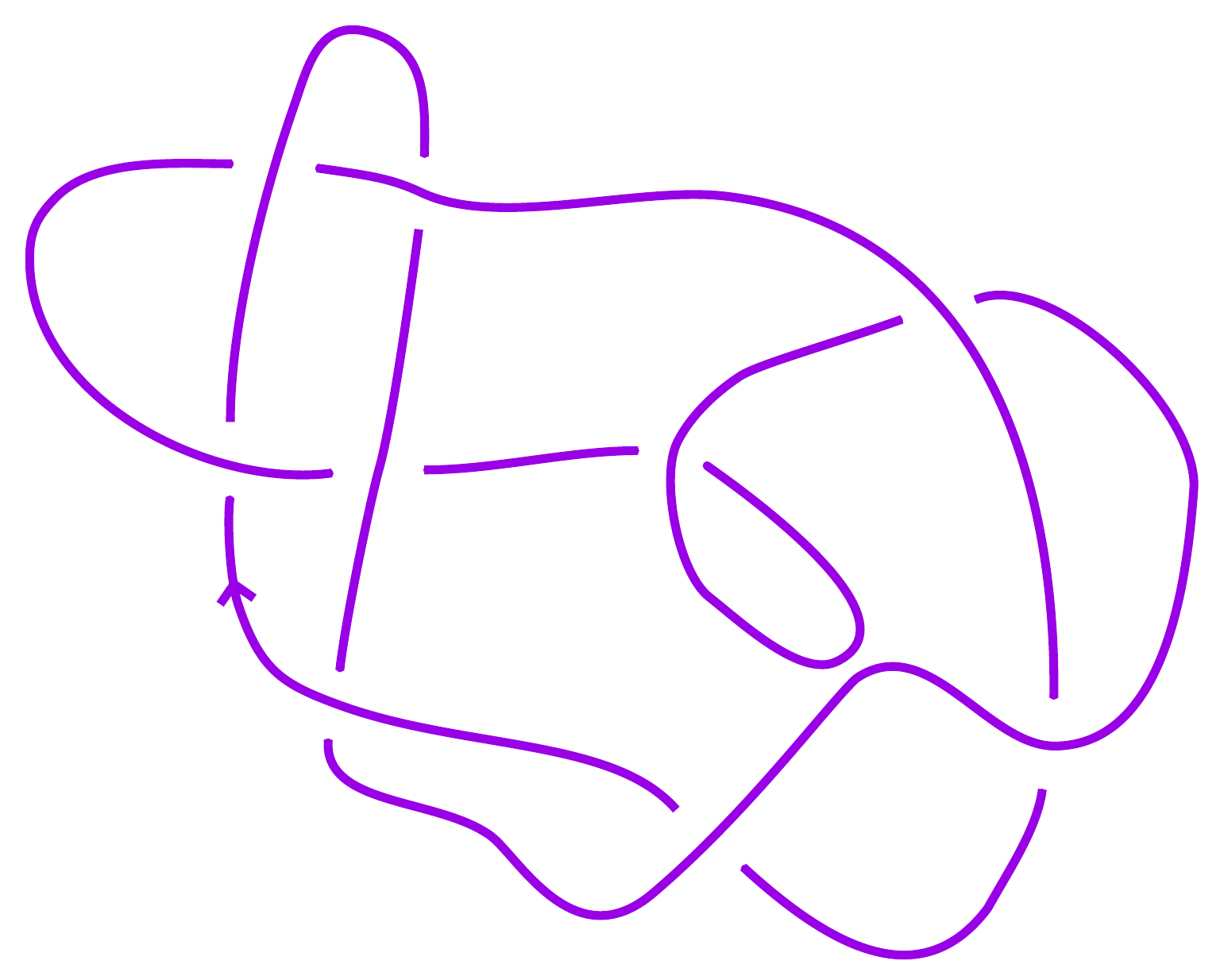}}
\put(335,60){$3_1^*$}
\put(100,100){\vector(-3,-2){40}}
\put(65,90){$\lambda_D^{-1}$}
\put(170,100){\vector(3,-2){40}}
\put(190,90){$-\frac{(q-q^{-1})}{\sqrt{\lambda_D}}$}

\put(0,0){\includegraphics[scale=0.25]{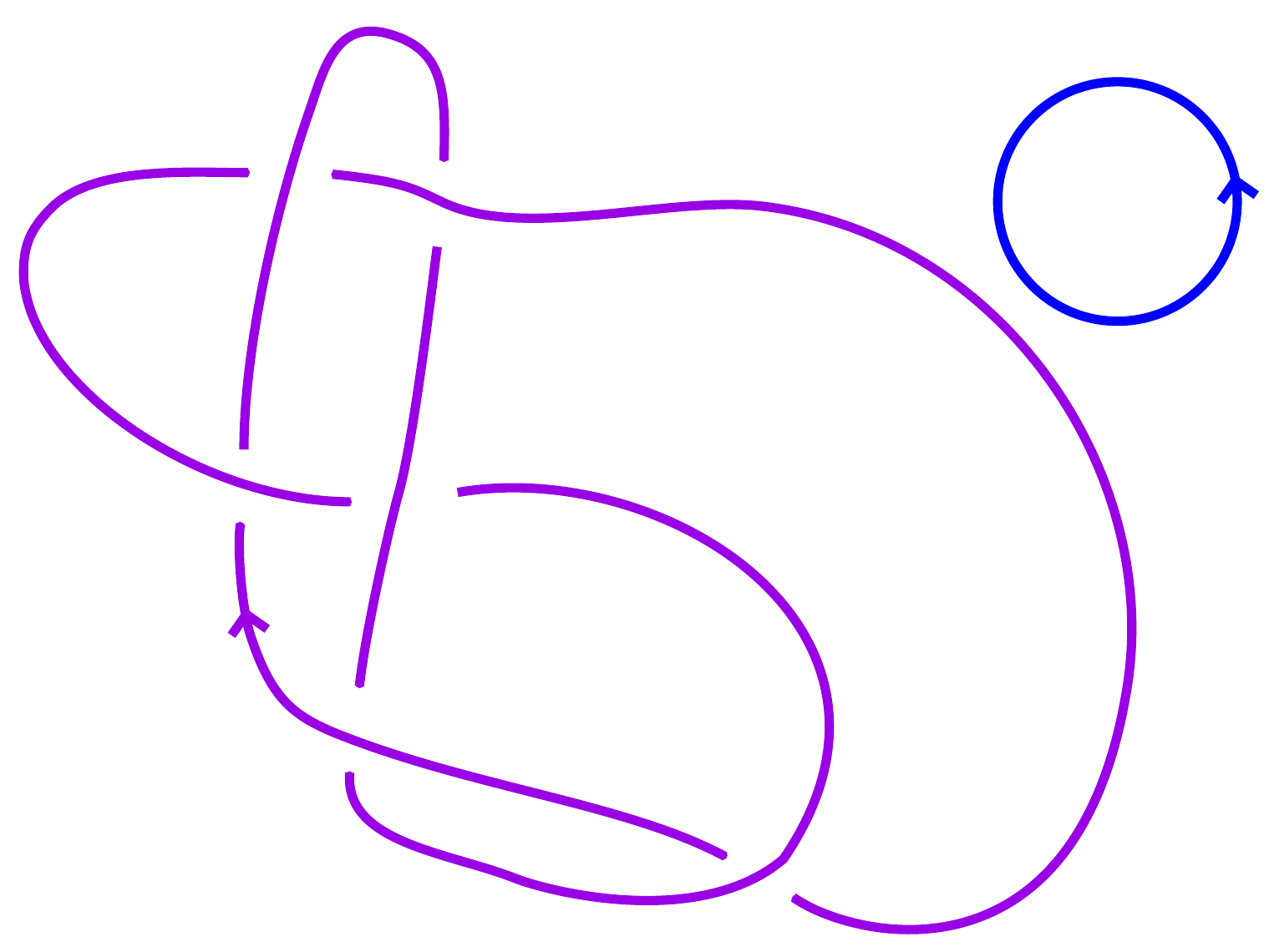}}
\put(190,0){\includegraphics[scale=0.25]{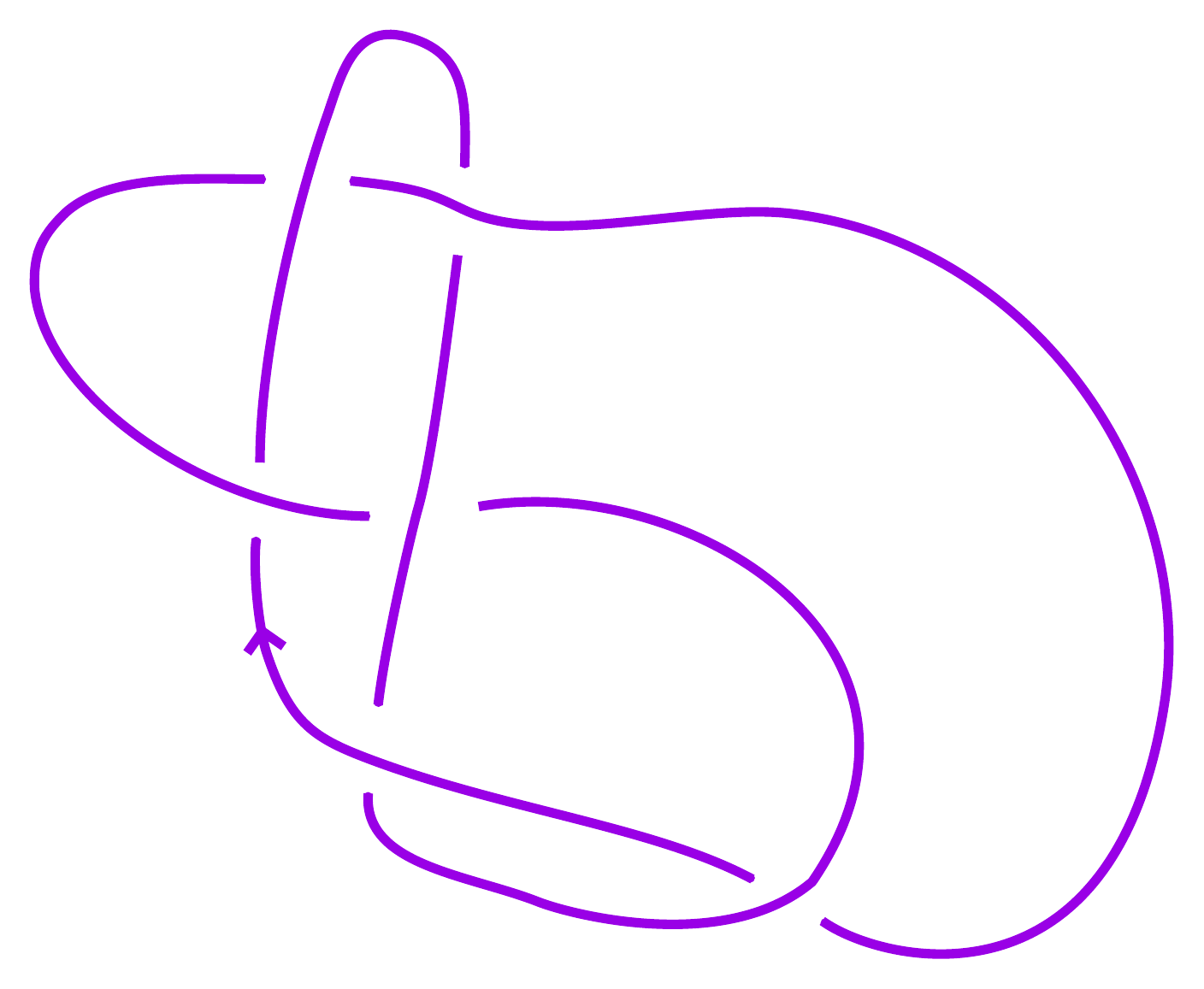}}
\put(0,0){$6_2^* \sqcup U$}
\put(210,0){$6_2^*$}
\end{picture}
\caption{The skein tree of the link $L11n358\{0,1\}$.}
\label{skein_tree_358}
\end{figure}

In the skein tree of Figure~\ref{skein_tree_358} there are resulting knots which need to be identified. For example, the knot $6_2^*$ can be easily identified by its diagram (available at  \url{http://www.indiana.edu/~knotinfo/diagram_display/diagram_display_6_2.html} \cite{chli}). Hence we have:
\begin{align*}
\Theta_d(L11n358\{0,1\}) = &\lambda_D \, \Theta_d(L_1) + (q-q^{-1}) \lambda_D^{-\frac{3}{2}} \, \Theta_d(6_2^* \sqcup U)\\
&-(q-q^{-1})^2 \lambda_D^{-1} \, \Theta_d(6_2^*) - (q-q^{-1})^2 \, \Theta_d(3_1^*).
\end{align*}

Now, the skein tree of the link $L_1$ is illustrated in Figure~\ref{skein_tree_L1}:

\begin{figure}[H]
\begin{picture}(460,380)
\put(150,350){$L_1=$}

\put(180,300){\includegraphics[scale=0.25]{L11n358_L.pdf}}
\put(200,310){\vector(-3,-4){30}}
\put(160,290){$\lambda_D^{-1}$}
\put(280,310){\vector(3,-4){30}}
\put(300,290){$\frac{-(q-q^{-1})}{\sqrt{\lambda_D}}$}

\put(135,200){\includegraphics[scale=0.25]{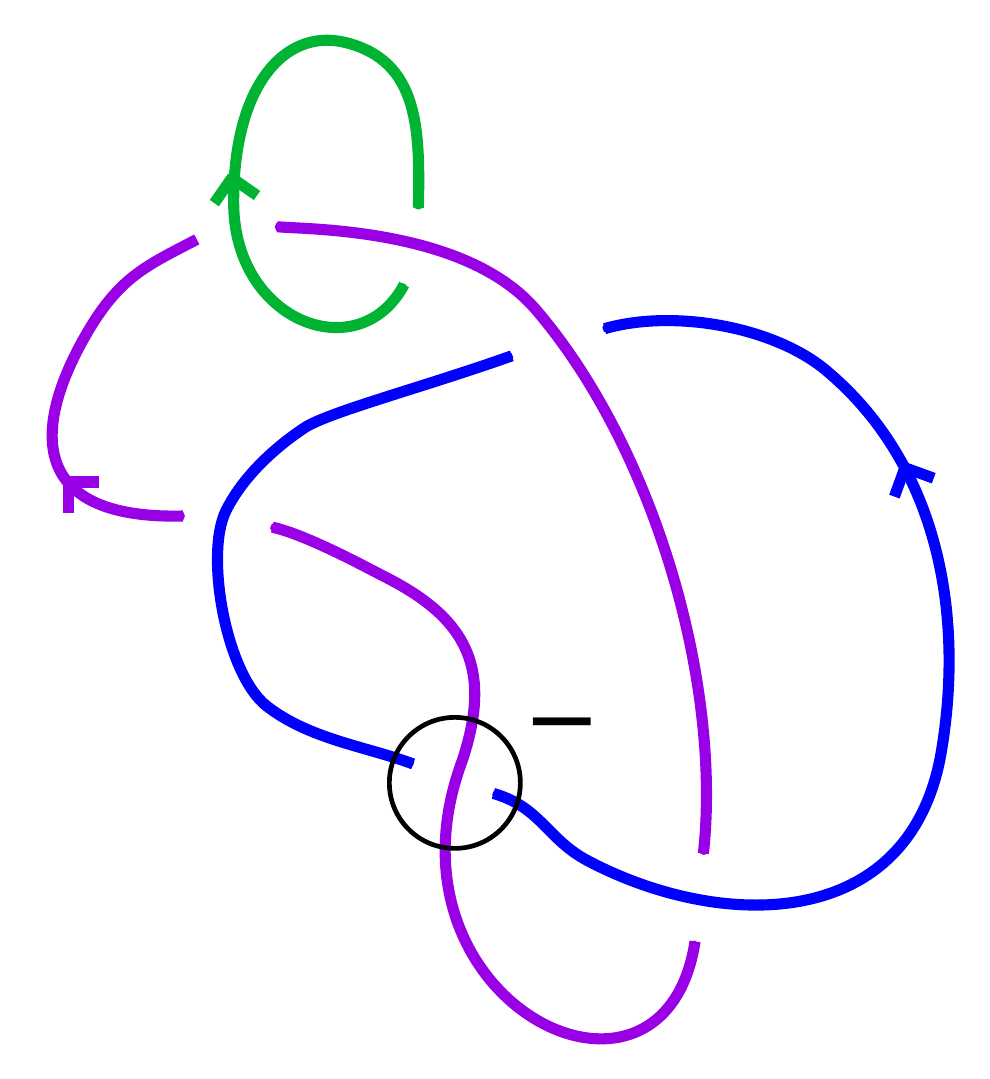}}
\put(300,200){\includegraphics[scale=0.25]{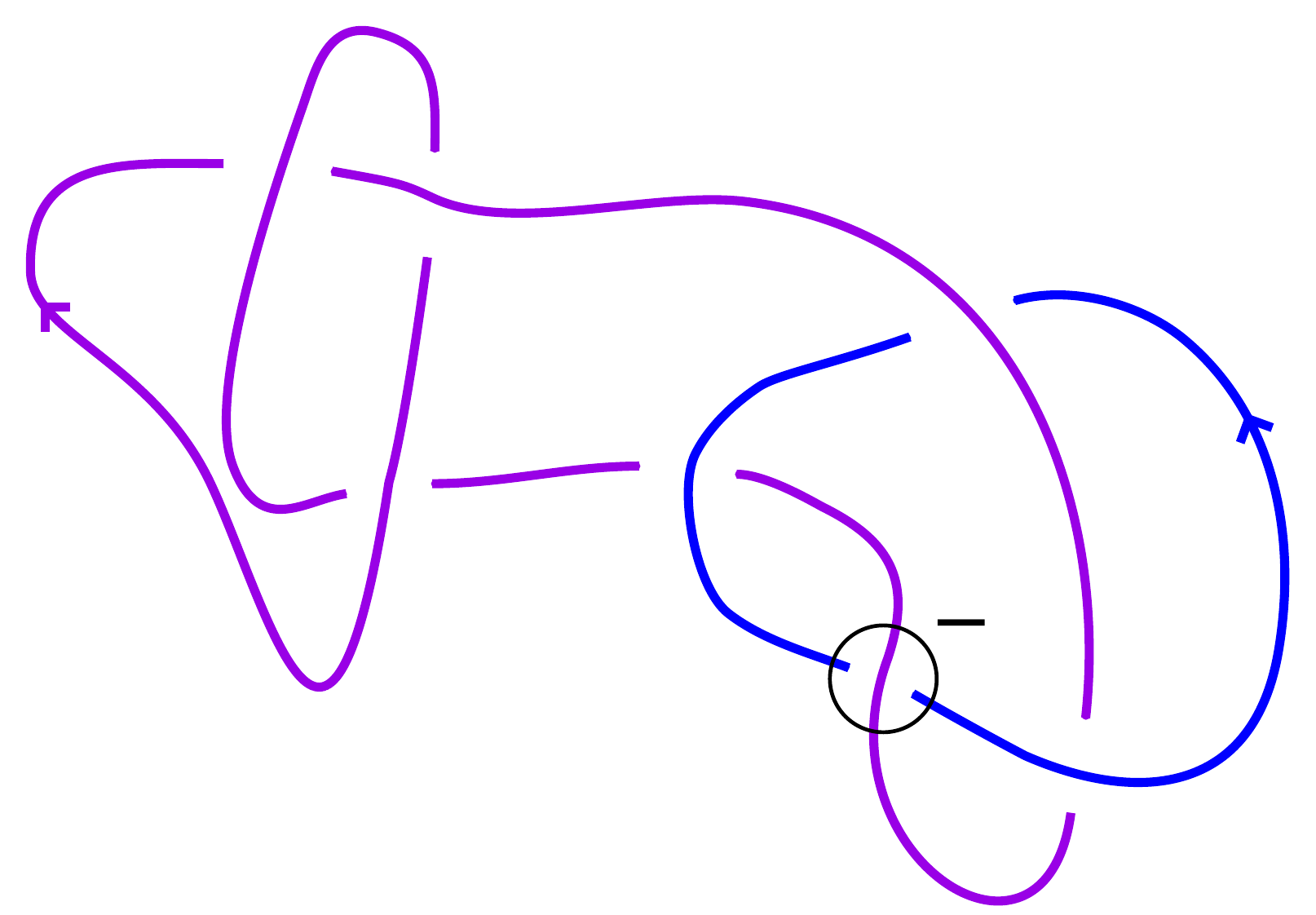}}

\put(150,220){\vector(-3,-4){30}}
\put(110,200){$\lambda_D^{-1}$}
\put(210,220){\vector(3,-4){30}}
\put(225,200){$\frac{-(q-q^{-1})}{\sqrt{\lambda_D}}$}

\put(340,210){\vector(0,-1){30}}
\put(320,190){$\lambda_D^{-1}$}
\put(410,210){\vector(1,-2){15}}
\put(420,190){$\frac{-(q-q^{-1})}{\sqrt{\lambda_D}}$}

\put(80,130){\includegraphics[scale=0.25]{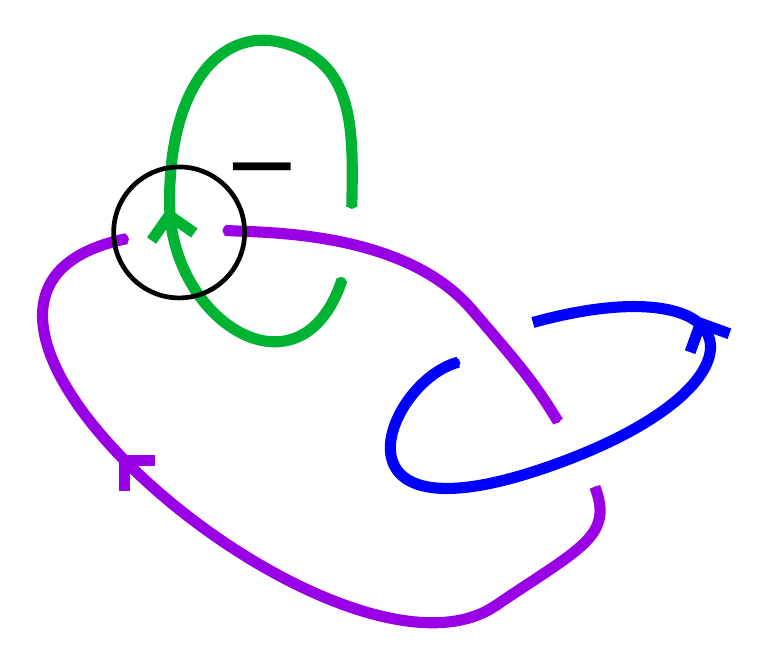}}
\put(190,120){\includegraphics[scale=0.25]{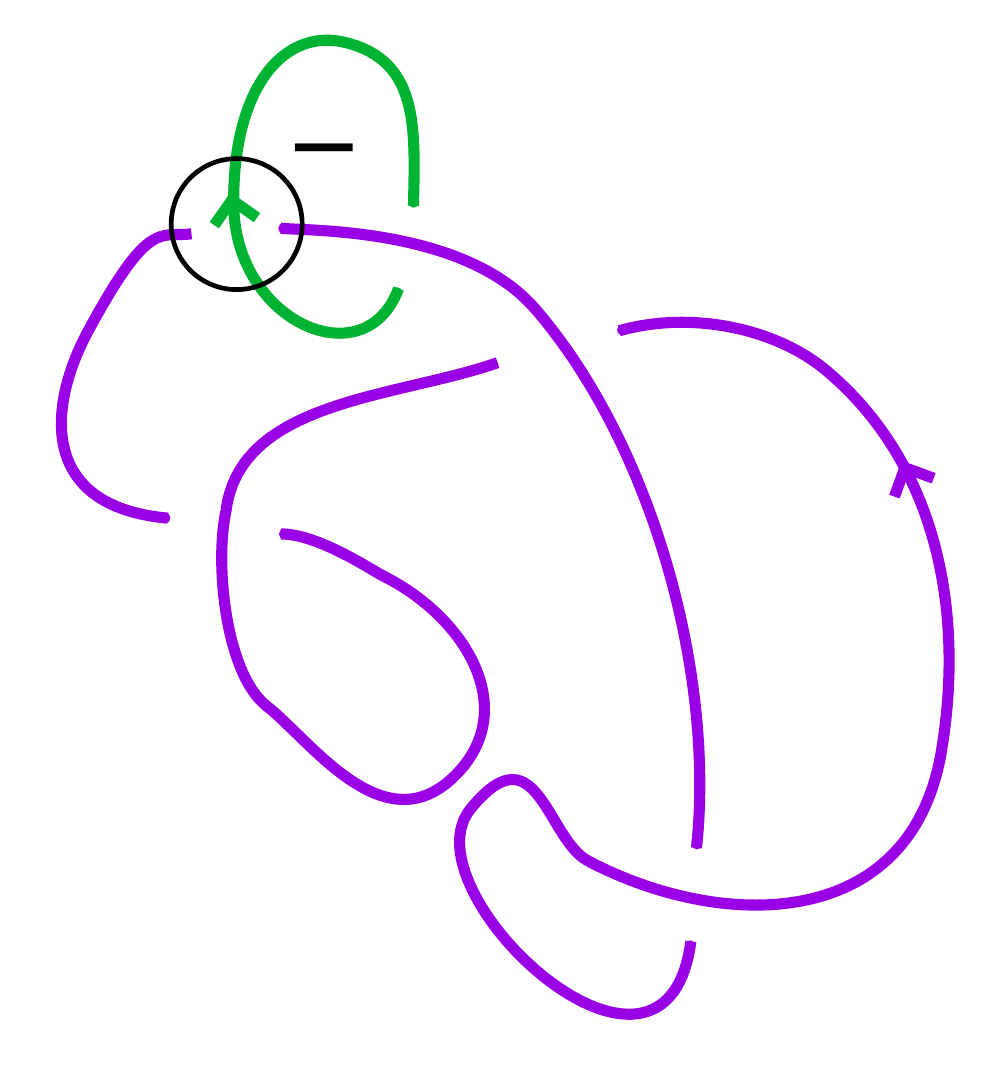}}
\put(325,120){\includegraphics[scale=0.25]{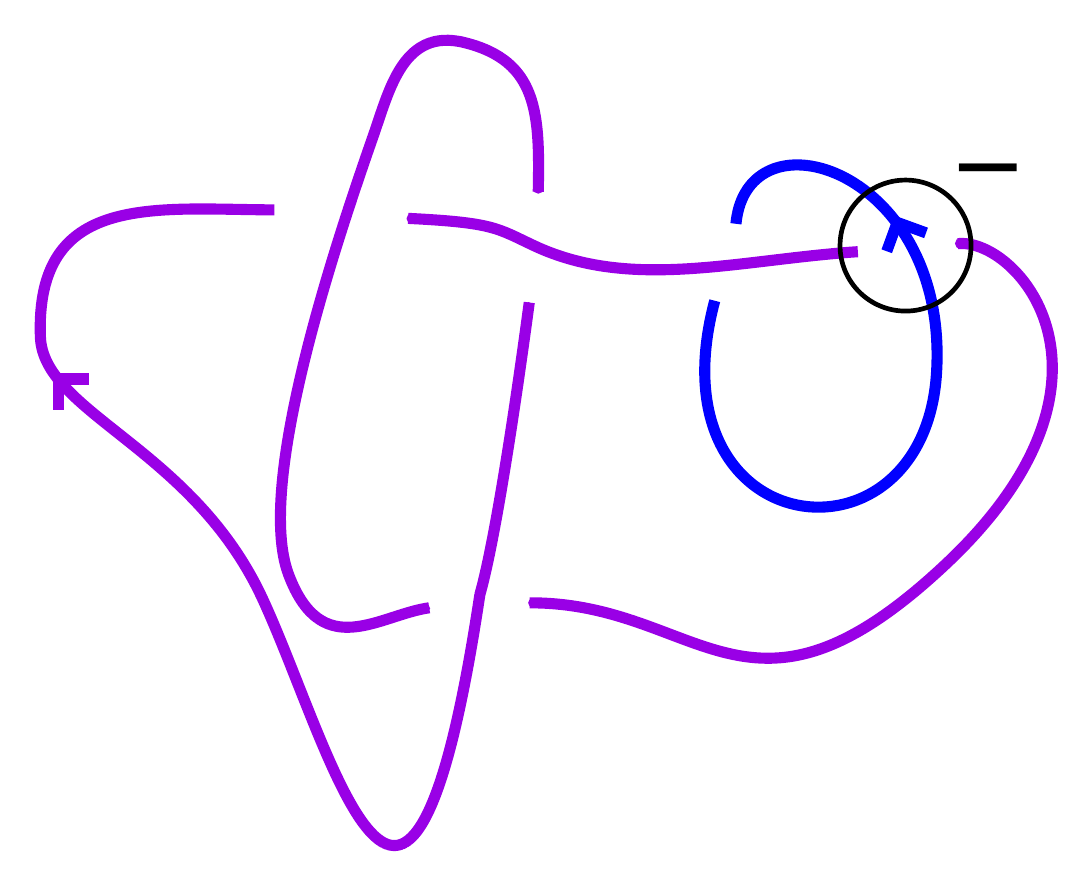}}
\put(430,160){$3_1^*$}

\put(90,120){\vector(-3,-4){30}}
\put(55,100){$\lambda_D^{-1}$}
\put(120,120){\vector(3,-4){30}}
\put(135,100){$\frac{-(q-q^{-1})}{\sqrt{\lambda_D}}$}

\put(220,120){\vector(-1,-2){15}}
\put(190,100){$\lambda_D^{-1}$}
\put(250,120){\vector(1,-2){15}}
\put(260,100){$\frac{-(q-q^{-1})}{\sqrt{\lambda_D}}$}

\put(340,120){\vector(-1,-2){15}}
\put(310,100){$\lambda_D^{-1}$}
\put(370,120){\vector(1,-2){15}}
\put(380,100){$\frac{-(q-q^{-1})}{\sqrt{\lambda_D}}$}

\put(27,30){\includegraphics[scale=0.25]{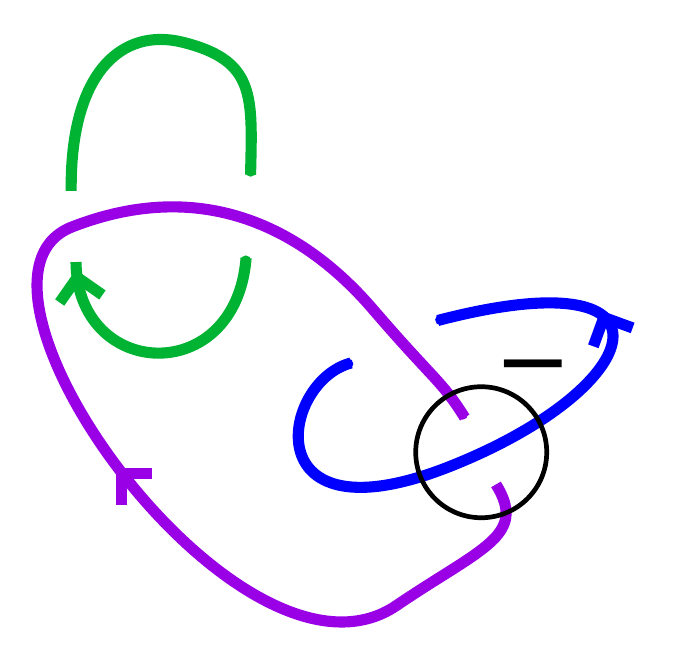}}
\put(155,30){\includegraphics[scale=0.25]{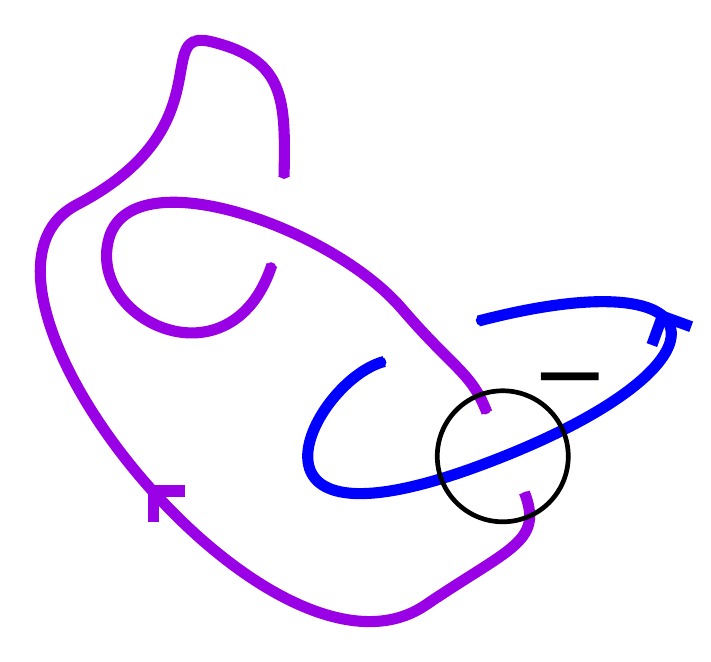}}
\put(190,80){$U \sqcup U$}
\put(260,80){$U$}
\put(310,80){$3_1^* \sqcup U$}
\put(380,80){$3_1^*$}

\put(30,40){\vector(-1,-4){7}}
\put(10,25){$\lambda_D^{-1}$}
\put(70,40){\vector(1,-4){7}}
\put(75,25){$\frac{-(q-q^{-1})}{\sqrt{\lambda_D}}$}

\put(160,40){\vector(-1,-4){7}}
\put(140,25){$\lambda_D^{-1}$}
\put(200,40){\vector(1,-4){7}}
\put(205,25){$\frac{-(q-q^{-1})}{\sqrt{\lambda_D}}$}

\put(00,0){$U \sqcup U \sqcup U$}
\put(60,0){$U \sqcup U$}
\put(140,0){$U \sqcup U$}
\put(205,0){$U$}
\end{picture}
\caption{The skein tree of the link $L_1$ (skein tree of $L11n358\{0,1\}$ continued).}
\label{skein_tree_L1}
\end{figure}

Hence,
\begin{align*}
\Theta_d(L_1) = &\lambda_D^{-4} \, \Theta_d(U \sqcup U \sqcup U) - 2 (q-q^{-1}) \lambda_D^{-\frac{7}{2}} \, \Theta_d(U \sqcup U) + \lambda_D^{-3} (q-q^{-1})^2 \, \Theta_d(U)\\
&- \lambda_D^{-\frac{5}{2}}(q-q^{-1}) \, \Theta_d(U \sqcup U) + \lambda_D^{-2}(q-q^{-1})^2 \, \Theta_d(U) - \lambda_D^{-\frac{5}{2}} (q-q^{-1}) \, \Theta_d(3_1^* \sqcup U)\\
&+ \lambda_D^{-1} (q-q^{-1})^2 \, \Theta_d(3_1^*) + \lambda_D^{-2} (q-q^{-1})^2 \, \Theta_d(3_1^*).
\end{align*}

Now, the value of $\Theta_d$ on knots is known due to Theorem~\ref{conjHomflypt} and the mirroring property of the Homflypt polynomial, that is, 
$P(q,\lambda_D)(L^*) = P(q^{-1}, \lambda_D^{-1})(L)$ 
for a link $L$. Following this, and using Theorem~\ref{disj_links} we obtain:

\begin{align*}
\Theta_d(L11n358\{0,1\}) = &\left(\frac{\left(q-q^{-1}\right)^2}{\lambda_D}-\frac{1-\lambda _D}{E_D \lambda _D^2}\right) \, P(3_1^*) - \left(\frac{\left(q-q^{-1}\right)^2}{\lambda_D} - \frac{1-\lambda _D}{E_D \lambda_D^2}\right) P(6_2^*)\\
&-\frac{\left(\lambda _D^{-2}+2\lambda_D^{-3}\right) \left(1-\lambda_D\right)}{E_D}+\left(q-q^{-1}\right)^2 \left(\lambda _D^{-1}+{\lambda_D^{-2}}\right) +\frac{\left(1-\lambda _D\right)^2}{E_D^2\left(q-q^{-1}\right)^2 \lambda _D^4}.
\end{align*}

The Homflypt polynomial takes the following values on the knots $3_1^*$ and $6_2^*$:

\begin{align*}
P(3_1^*) = &\frac{-q^2 + \lambda_D + q^4 \lambda _D}{q^2 \lambda_D^2}\\
P(6_2^*) = &-\frac{\lambda_D + q^8 \lambda_D + q^4 (1 + 2 \lambda_D) - q^2 (1 + \lambda_D + \lambda_D^2) - q^6(1 + \lambda_D + \lambda_D^2)}{q^4 \lambda_D^2}.
\end{align*}

Putting everything together, we finally obtain that:

\begin{align*}
\Theta_d(L11n358\{0,1\}) = 
&\left(E_D^2 \lambda_D^4 (q-1)^2 q^6 (q+1)^2 \right)^{-1} \left(E_D \lambda_D +E_D \lambda_D q^4-2 E_D \lambda_D q^2+\lambda_D 
   q^2-q^2\right)\\
   &(E_D \lambda_D + E_D \lambda_D q^{12}-E_D \lambda_D^2
   q^{10}-2 E_D \lambda_D q^{10}-E_D q^{10}+3 E_D \lambda_D ^2 q^8\\
   &+4 E_D \lambda_D q^8 +2 E_D q^8-4 E_D \lambda_D ^2 q^6-6 E_D \lambda_D q^6-2 E_D
   q^6+3 E_D \lambda_D^2 q^4 \\
   &+4 E_D \lambda_D q^4 + 2 E_D q^4-E_D \lambda_D ^2
   q^2-2 E_D \lambda_D q^2-E_D q^2+\lambda_D q^6-q^6),
\end{align*}
which can be easily confirmed by any computer algebra system.
\smallbreak
Now we apply the same procedure to the link $L11n418\{0,0\}$. Applying the special skein relation once, we get two new links $L_2$ and $L_3$ as illustrated in Figure~\ref{skein_tree_418}:

\begin{figure}[H]
\begin{picture}(345,120)

\put(120,40){\includegraphics[scale=0.25]{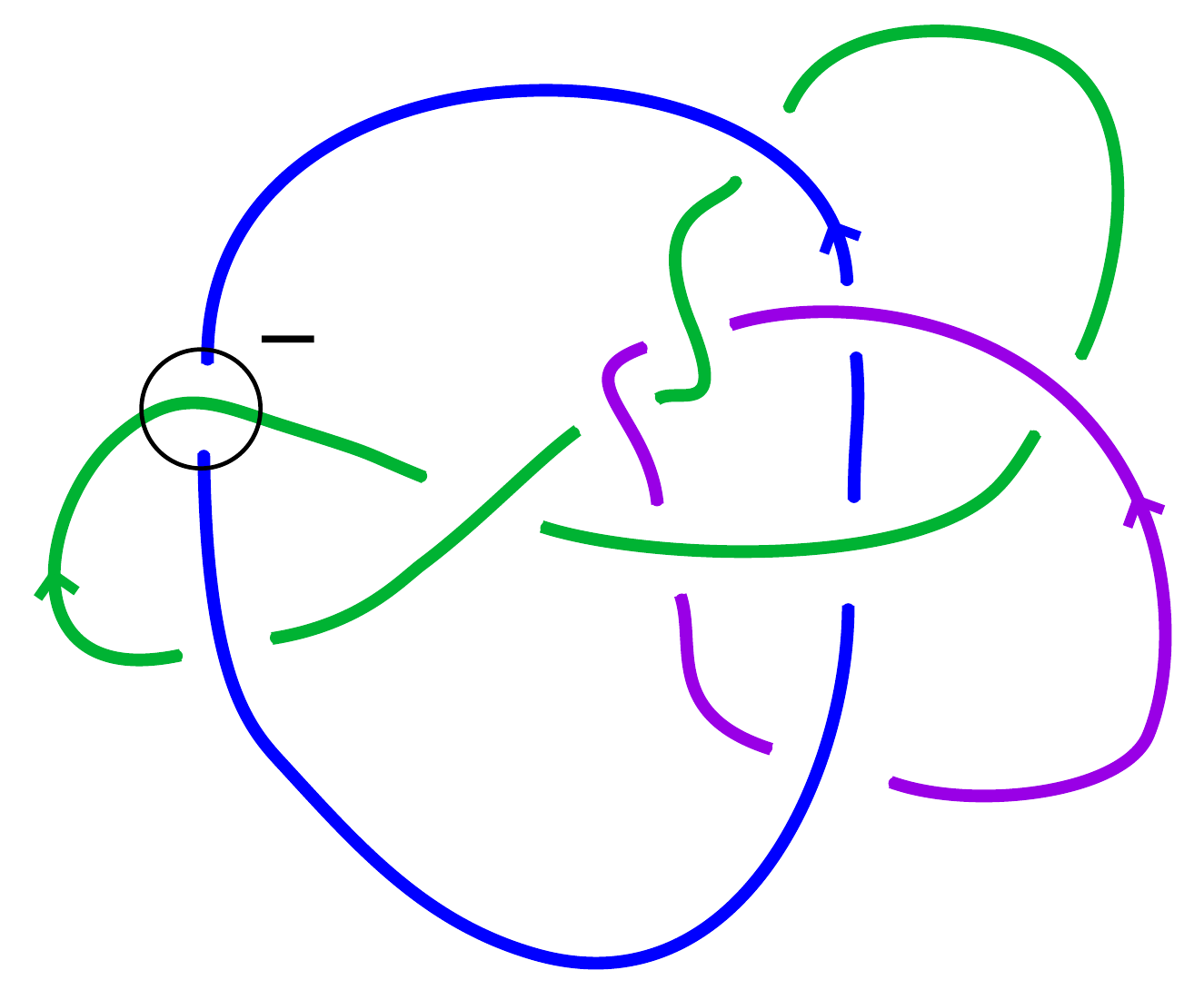}}
\put(220,60){\vector(3,-1){30}}
\put(220,65){$\frac{-(q-q^{-1})}{\sqrt{\lambda_D}}$}
\put(120,60){\vector(-3,-1){30}}
\put(90,60){$\lambda_D^{-1}$}

\put(0,0){\includegraphics[scale=0.25]{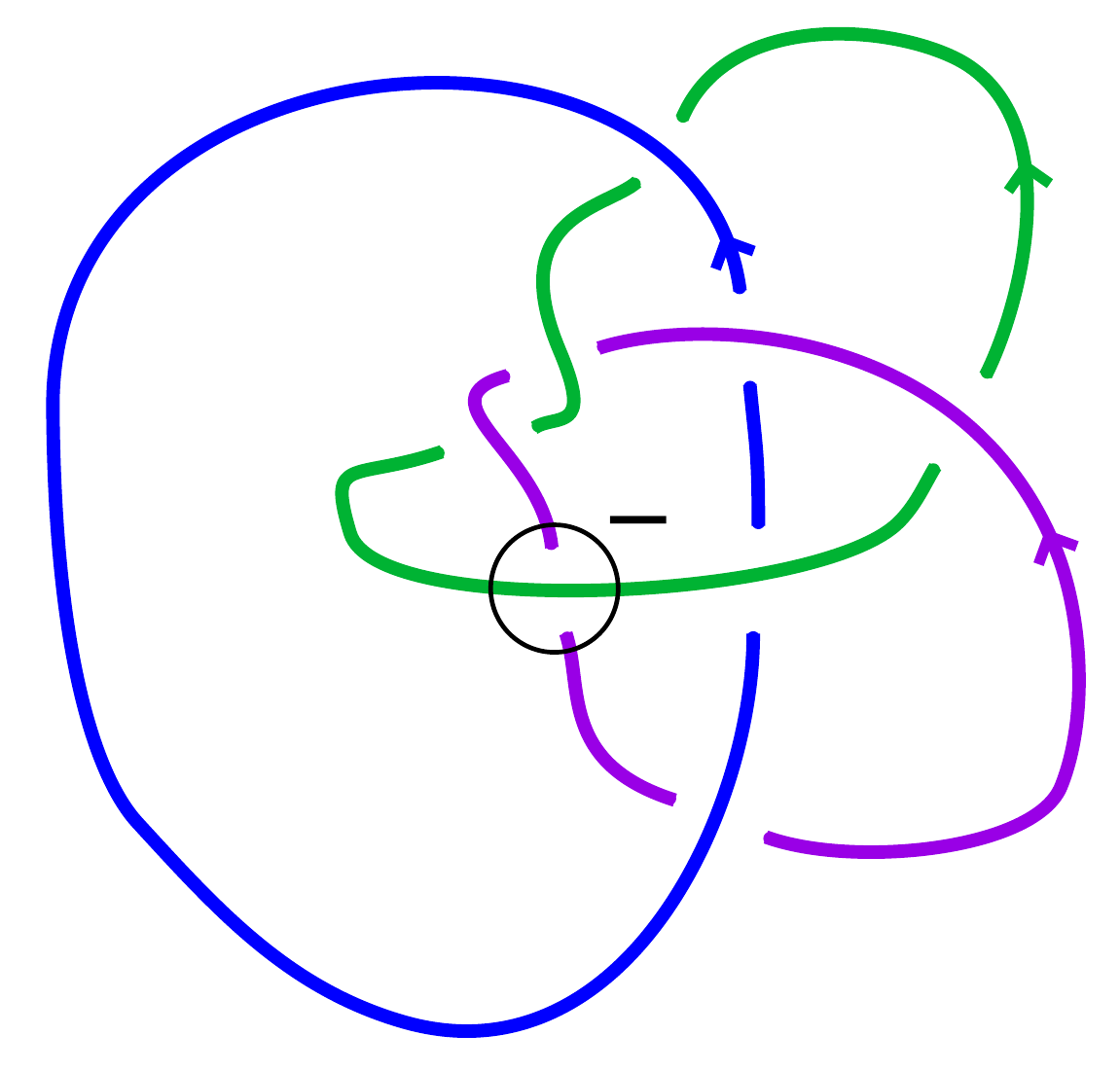}}
\put(250,0){\includegraphics[scale=0.25]{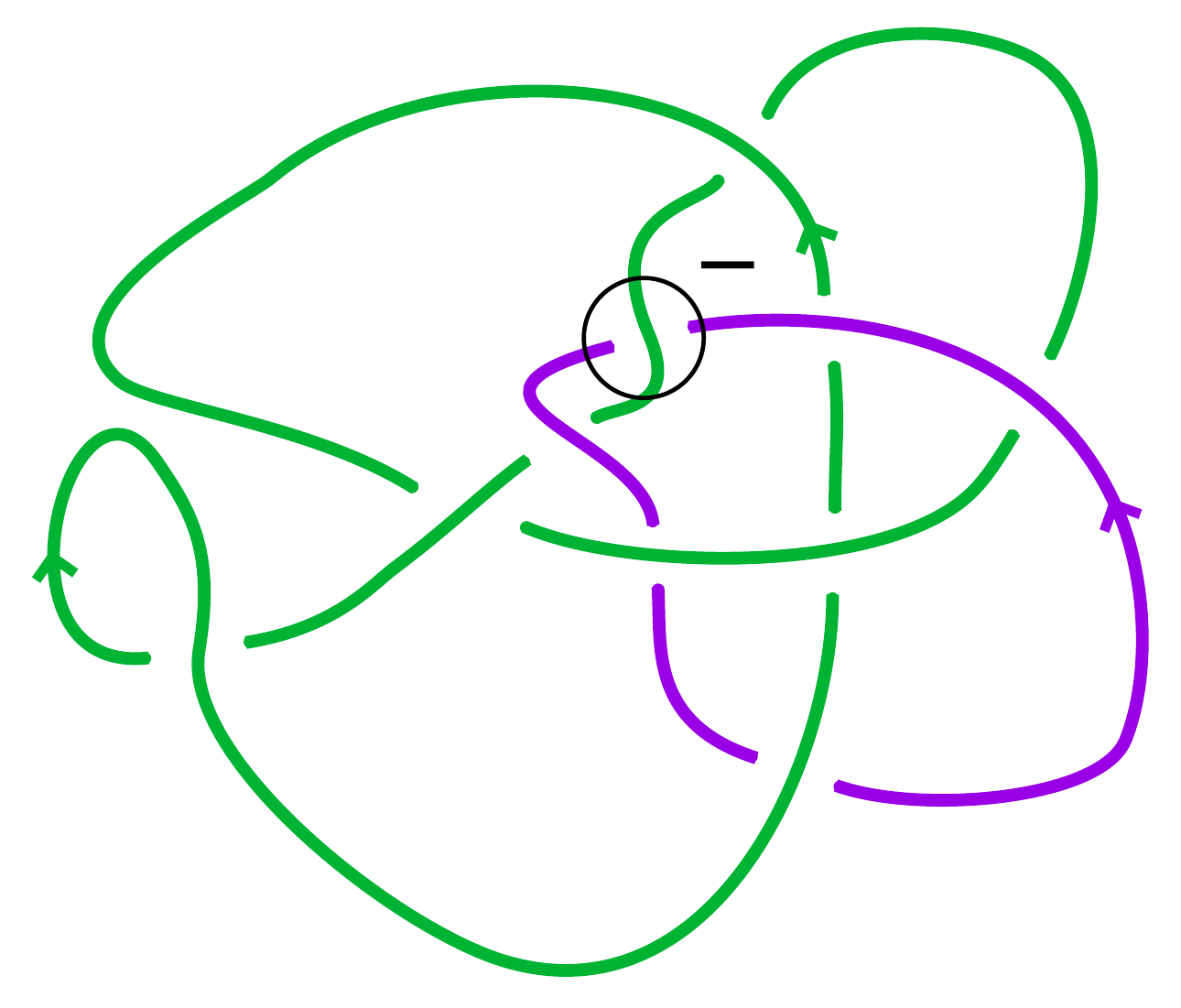}}
\put(50,0){=: $L_2$}
\put(310,0){=: $L_3$}
\end{picture}
\caption{The skein tree of the link $L11n418\{0,0\}$.}
\label{skein_tree_418}
\end{figure}
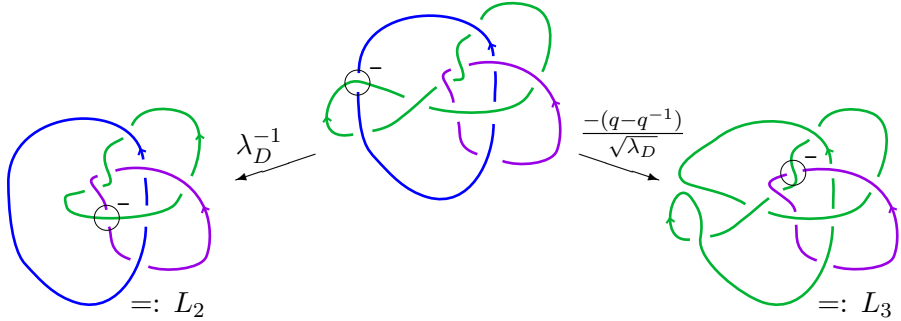

Hence, we have that:
$$
\Theta_d(L11n418\{0,0\}) = \lambda_D^{-1} \, \Theta_d(L_2) -(q-q^{-1}) \sqrt{\lambda_D}^{-1} \, \Theta_d(L_3).
$$

We analyze first the link $L_2$ into simpler links via the special skein relation, as in Figure~\ref{skein_tree_L2}:

\begin{figure}[H]
\begin{picture}(467,330)
\put(150,300){$L_2 =$}

\put(180,250){\includegraphics[scale=0.25]{L11n418_L.pdf}}

\put(270,270){\vector(3,-1){30}}
\put(270,270){$-(q-q^{-1}) \sqrt{\lambda_D}^{-1}$}
\put(180,270){\vector(-3,-1){30}}
\put(155,270){$\lambda_D^{-1}$}

\put(60,190){\includegraphics[scale=0.25]{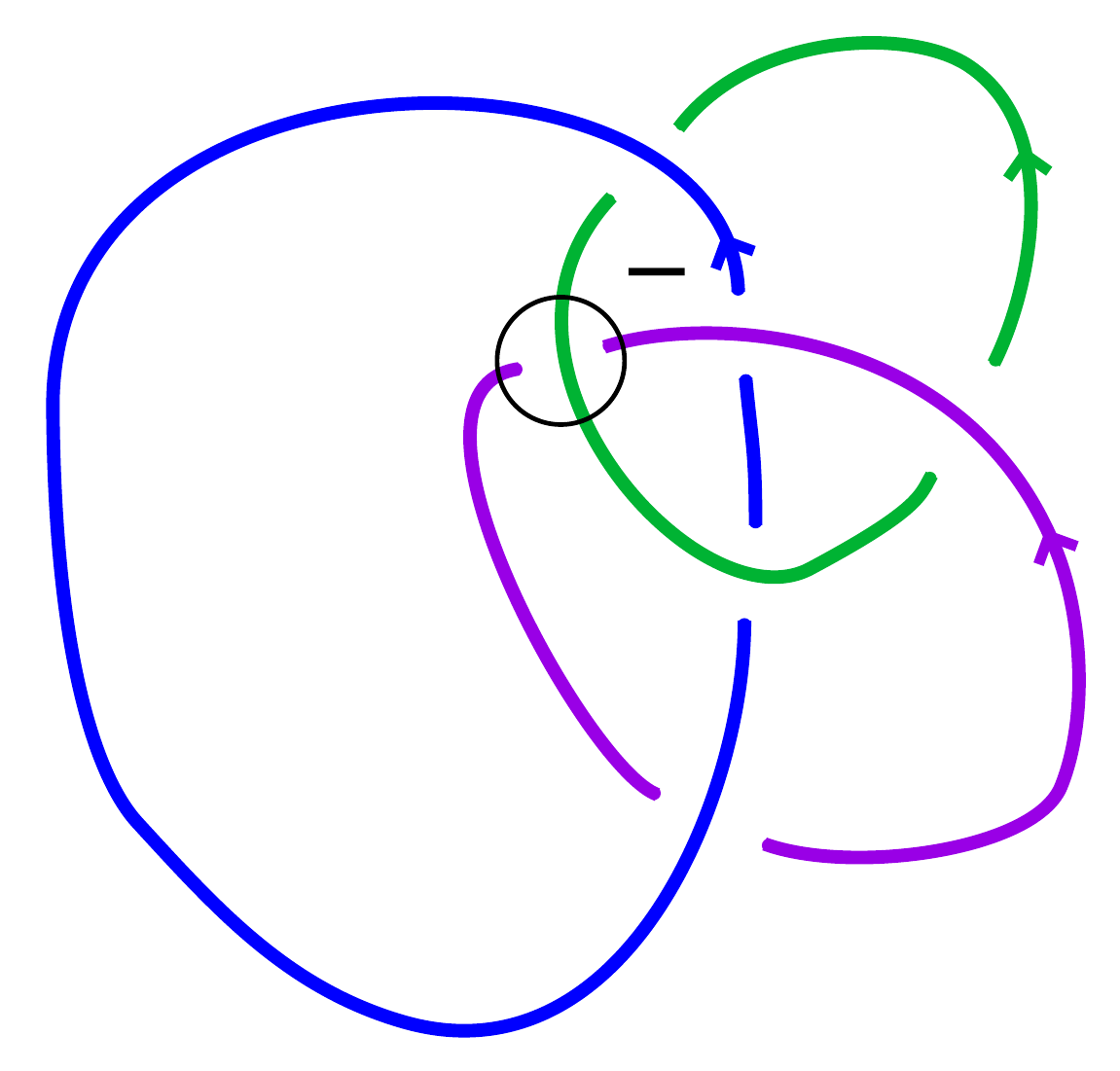}}
\put(305,190){\includegraphics[scale=0.25]{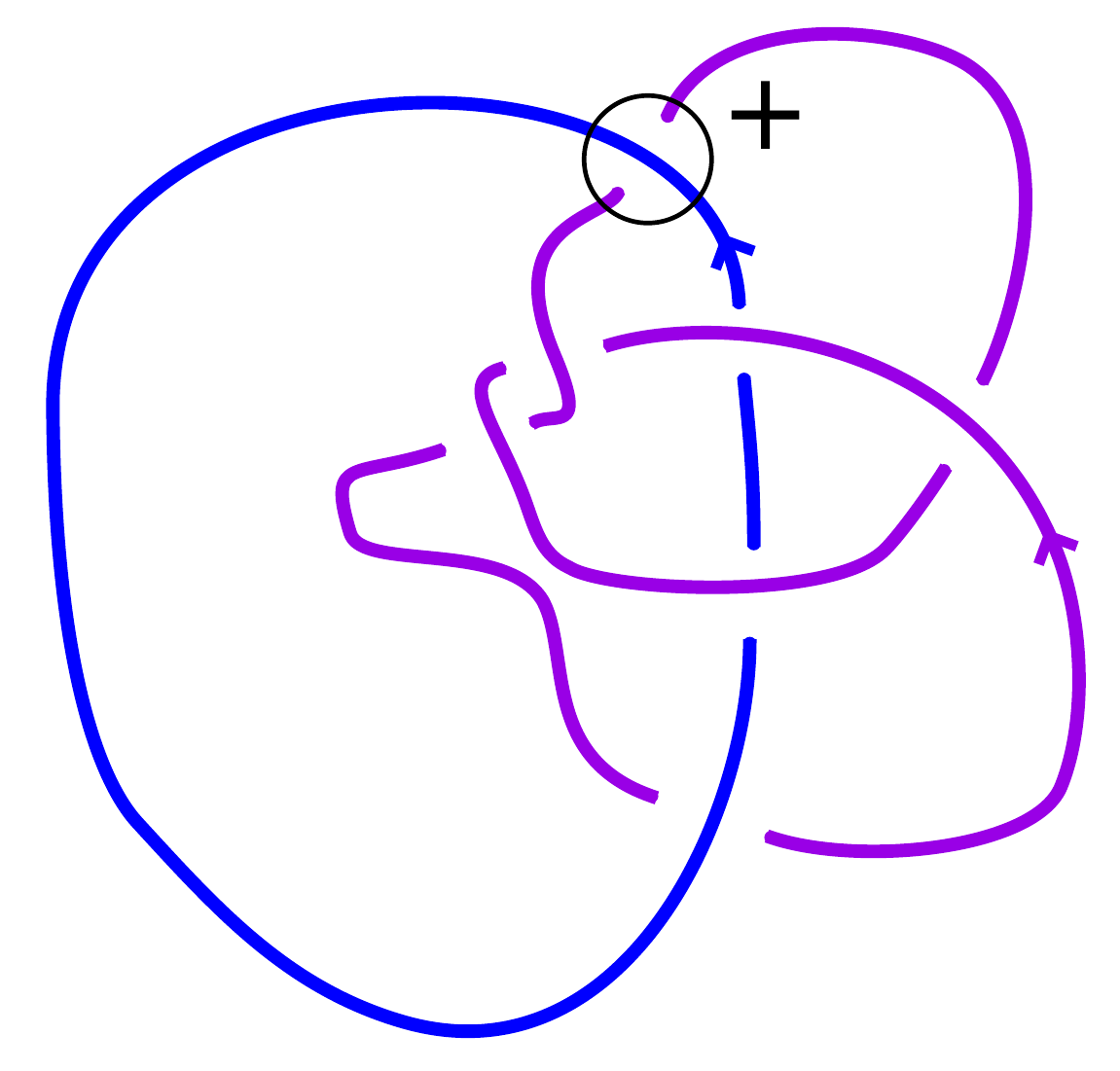}}

\put(390,205){\vector(3,-4){20}}
\put(400,195){$(q-q^{-1}) \sqrt{\lambda_D}$}
\put(310,205){\vector(-3,-4){20}}
\put(285,195){$\lambda_D$}

\put(130,200){\vector(3,-4){20}}
\put(140,190){$-(q-q^{-1}) \sqrt{\lambda_D}^{-1}$}
\put(70,200){\vector(-3,-4){20}}
\put(40,190){$\lambda_D^{-1}$}

\put(0,100){\includegraphics[scale=0.25]{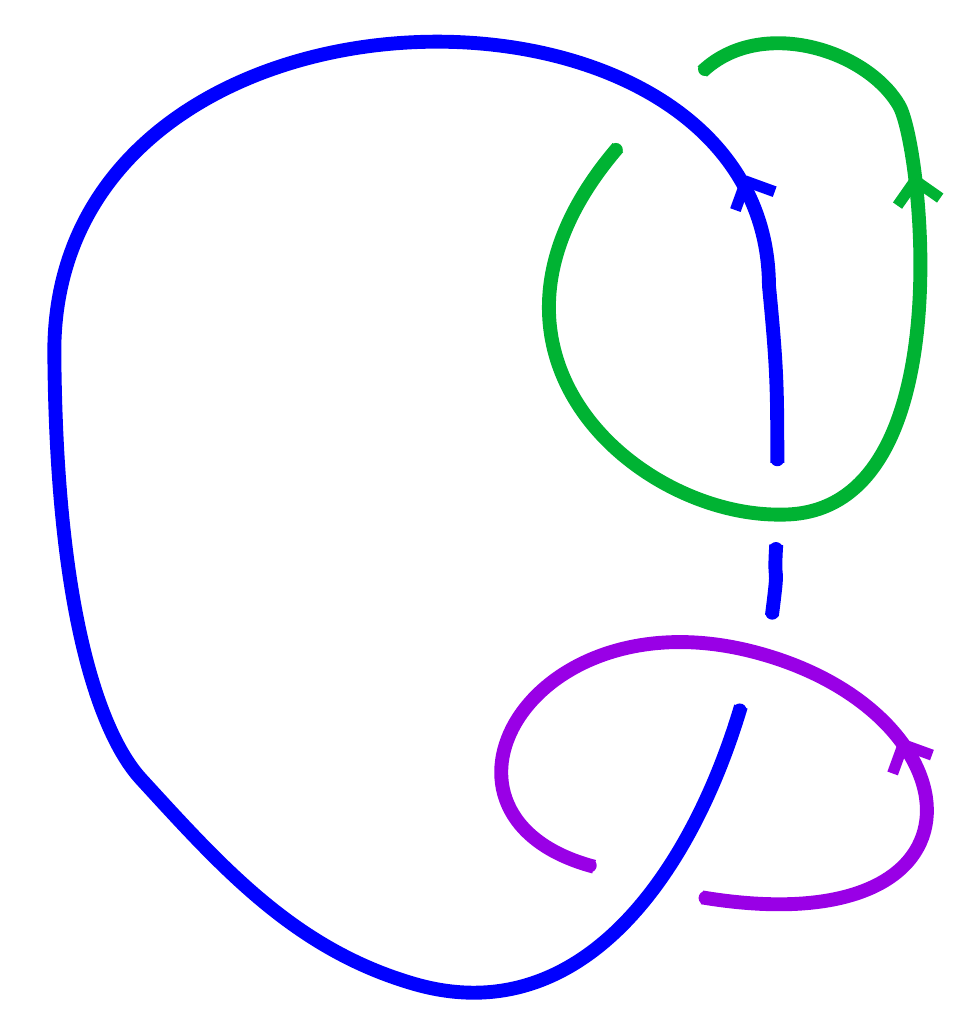}}
\put(110,100){\includegraphics[scale=0.25]{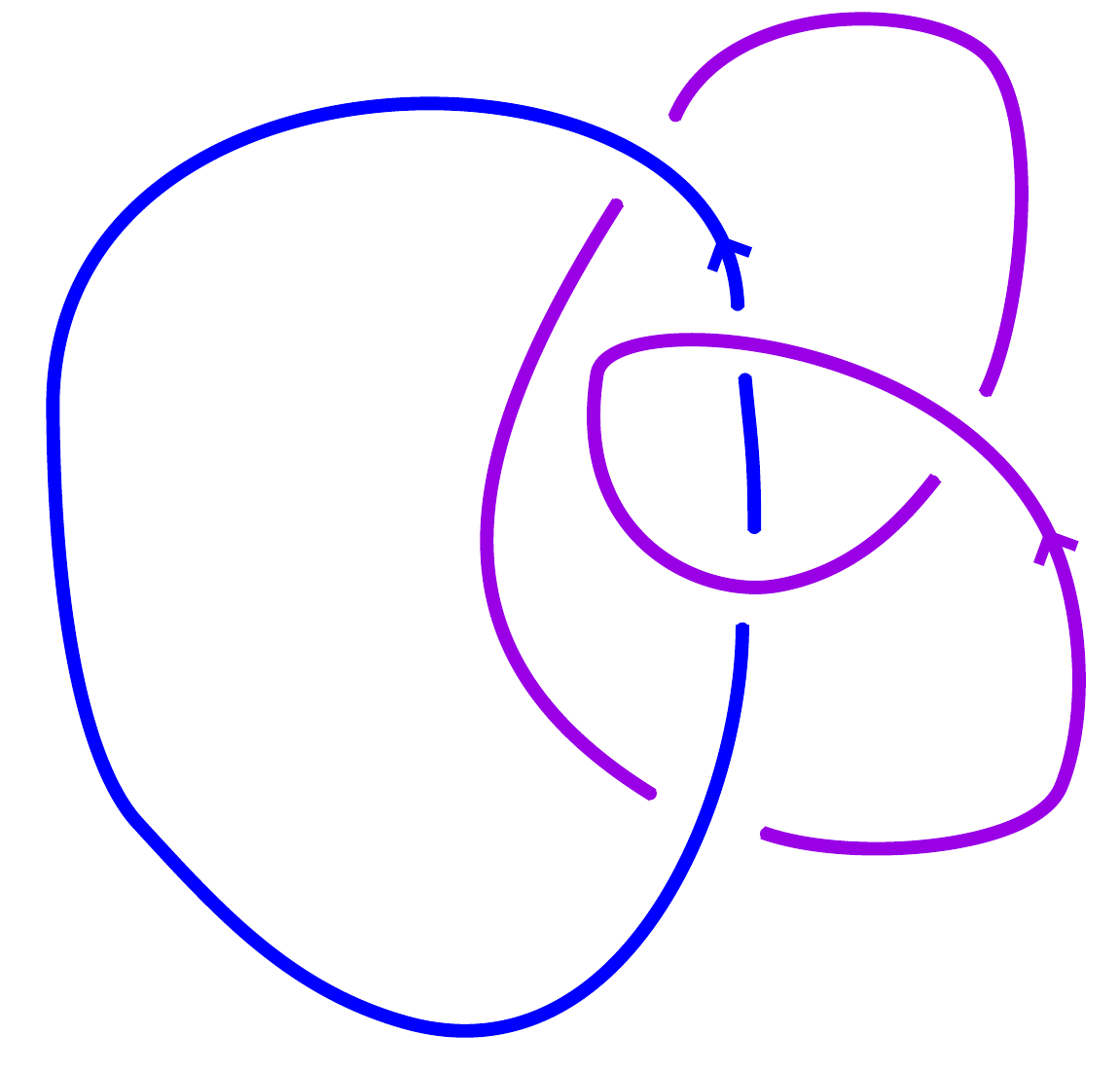}}
\put(50,100){$=: L_4$}
\put(160,100){$U \sqcup U$}

\put(220,100){\includegraphics[scale=0.25]{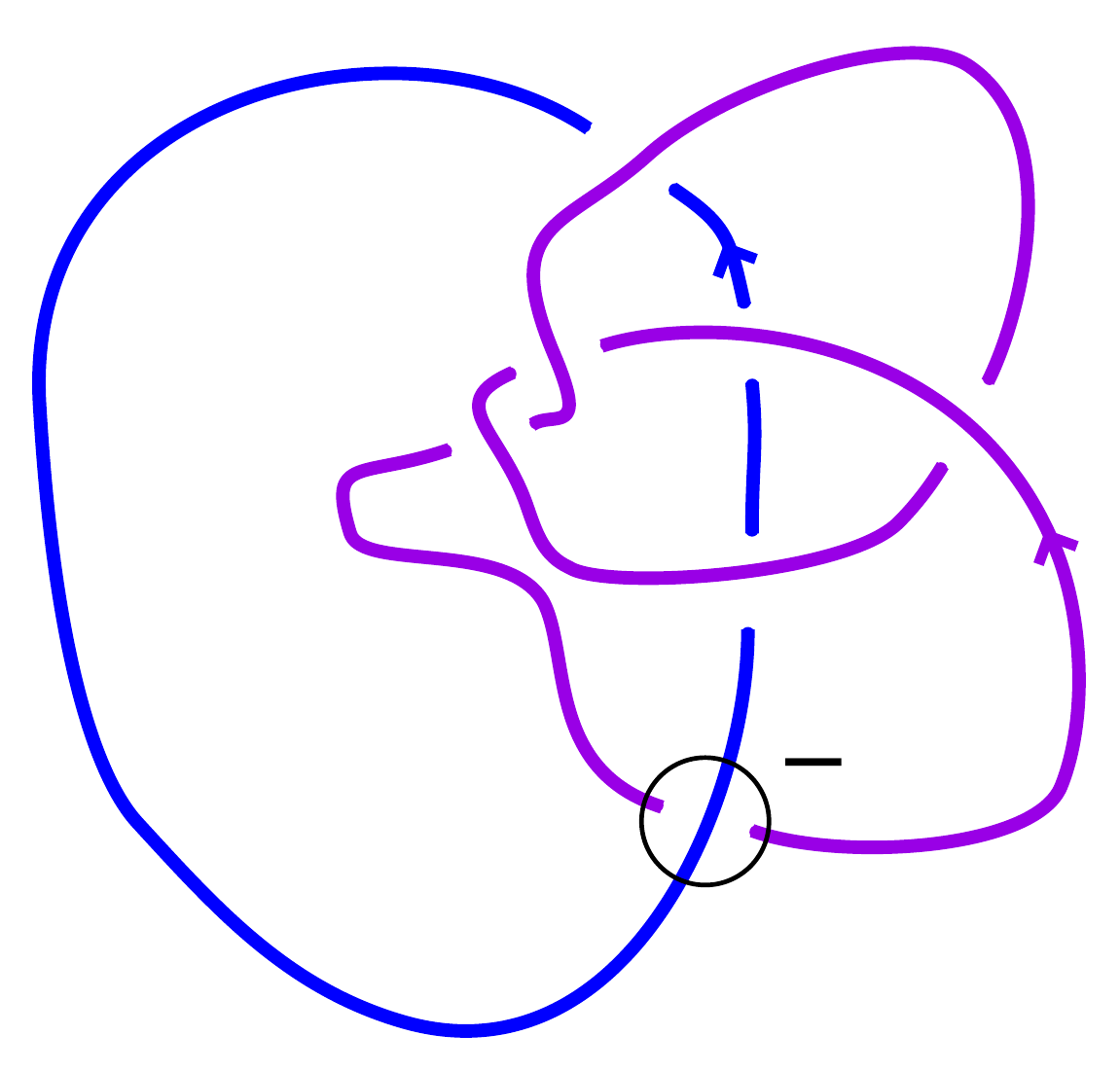}}
\put(360,100){\includegraphics[scale=0.25]{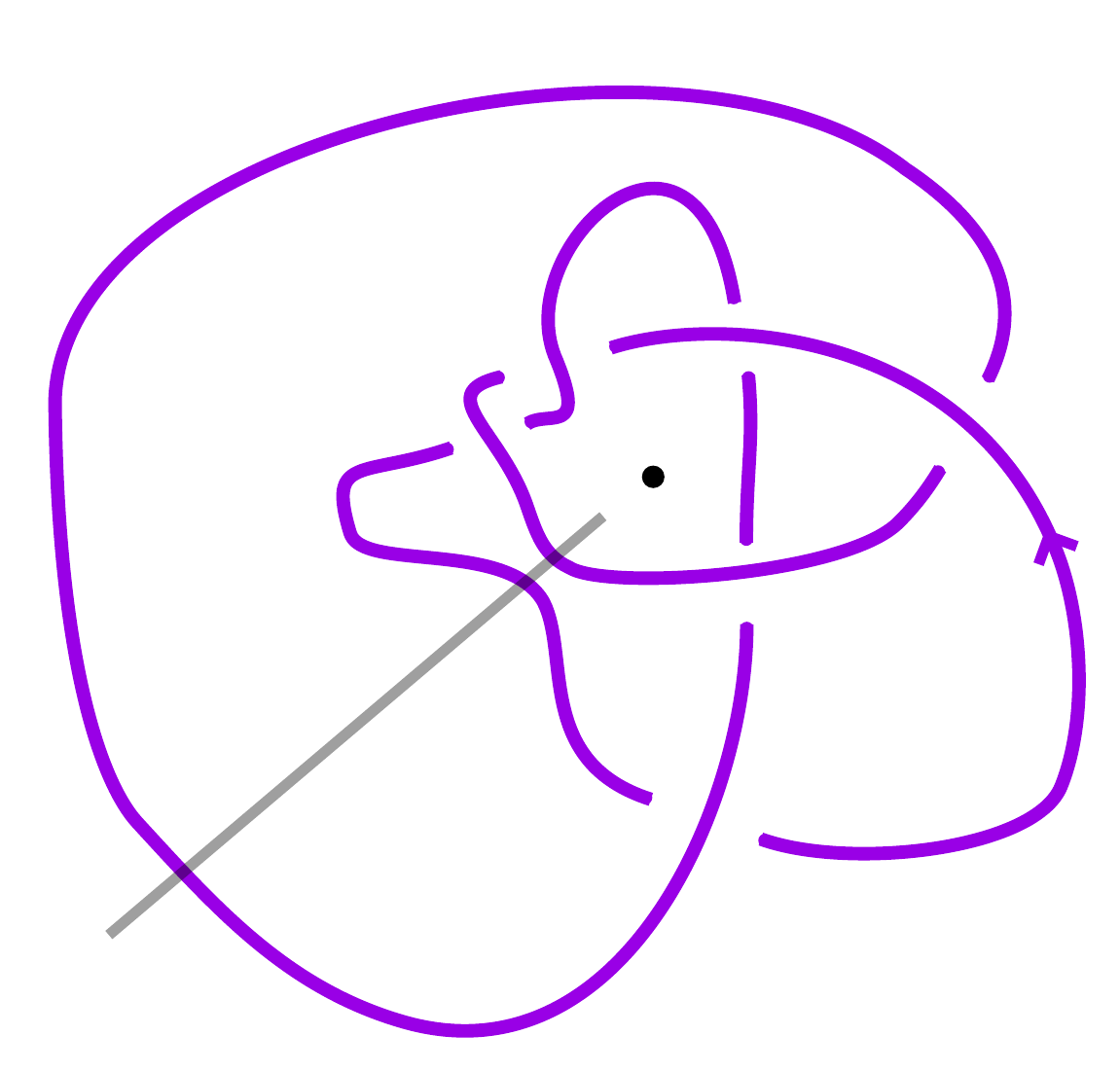}}
\put(420,100){$5_2^*$}

\put(260,100){\vector(3,-4){20}}
\put(270,90){$-(q-q^{-1}) \sqrt{\lambda_D}^{-1}$}
\put(250,100){\vector(-3,-4){20}}
\put(220,90){$\lambda_D^{-1}$}

\put(150,0){\includegraphics[scale=0.25]{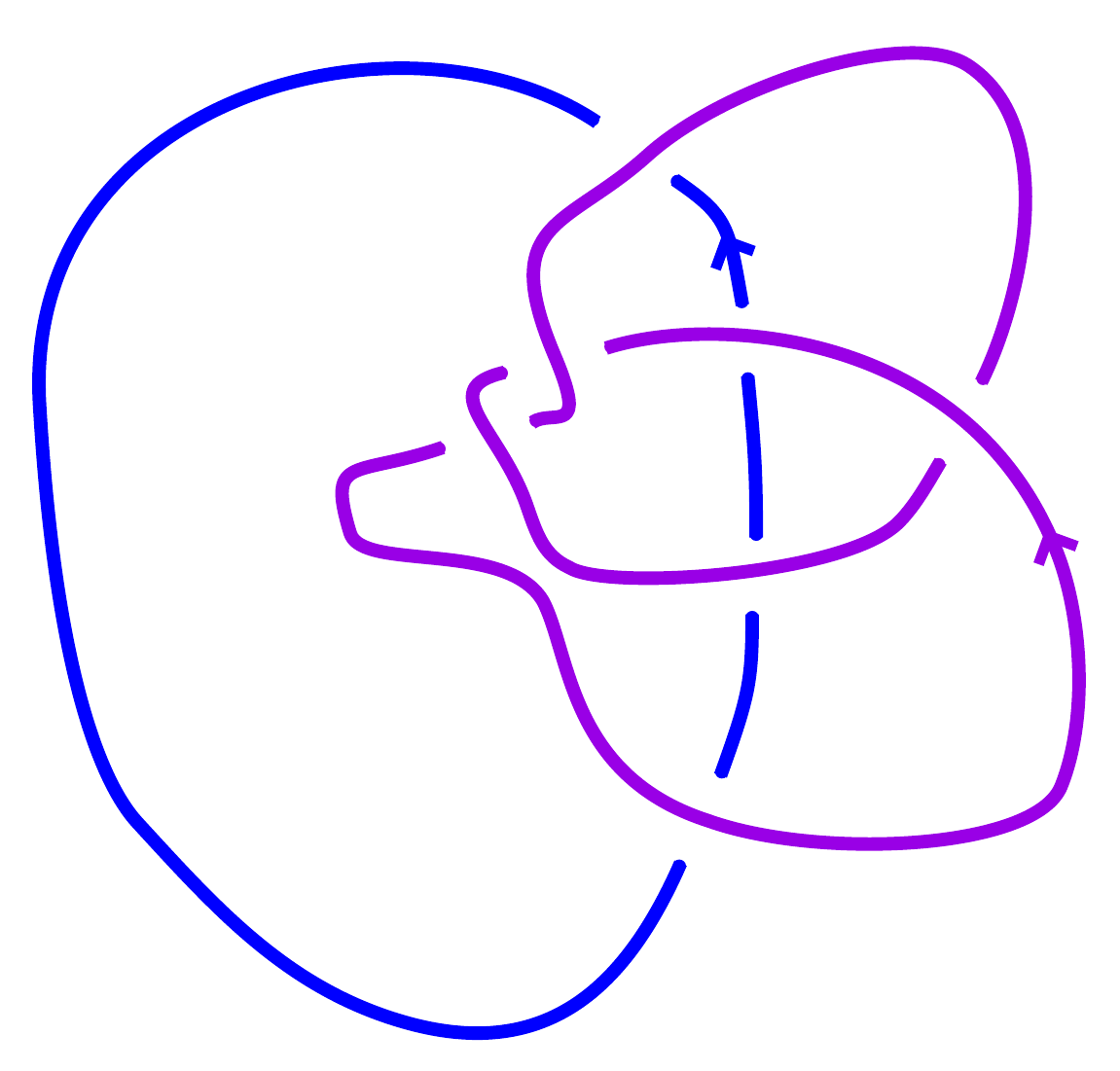}}
\put(270,0){\includegraphics[scale=0.25]{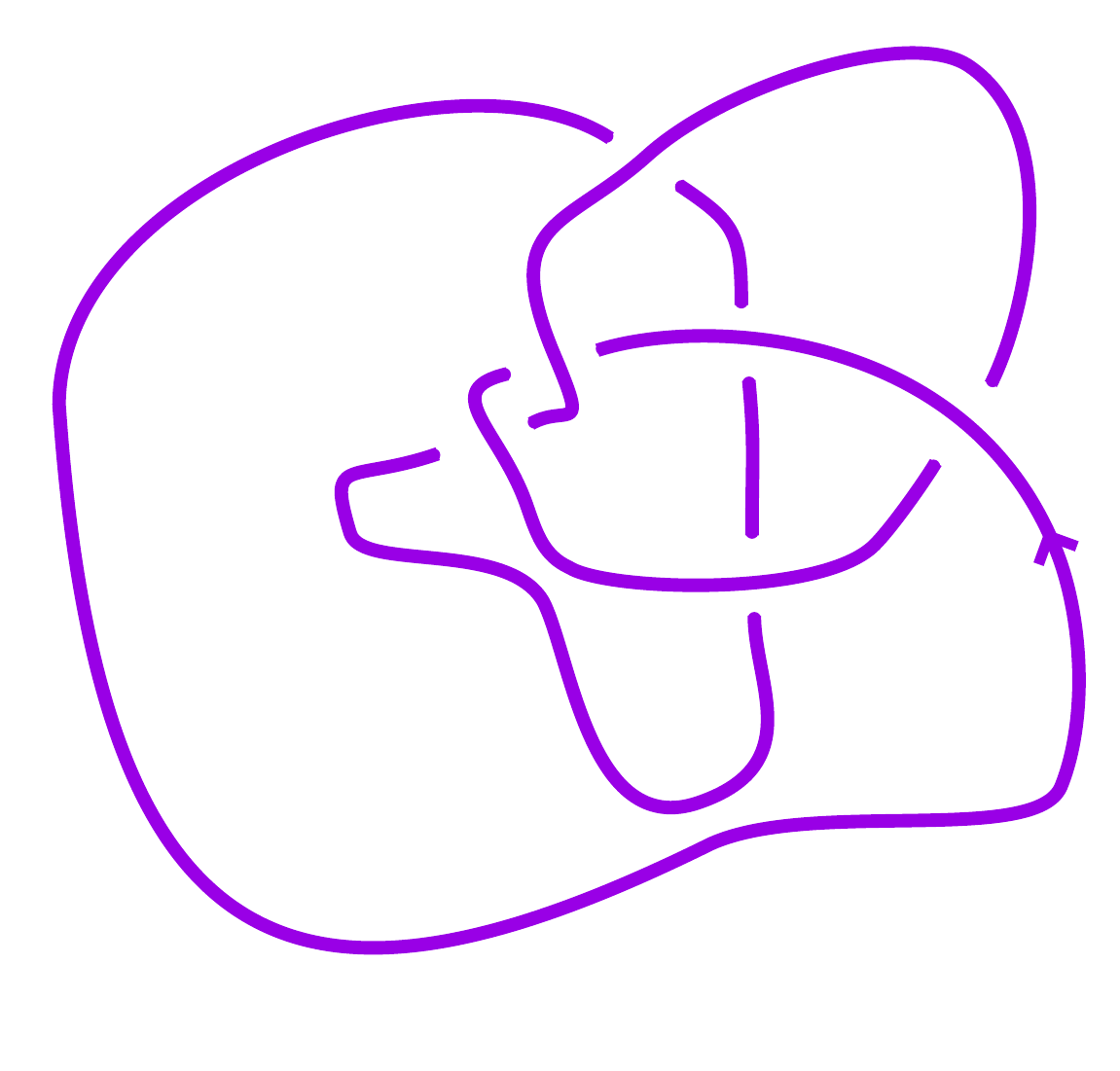}}
\put(200,0){$3_1^* \sqcup U$}
\put(320,0){$3_1^*$}
\end{picture}
\caption{The skein tree of the link $L_2$ (skein tree of $L11n418\{0,0\}$ continued).}
\label{skein_tree_L2}
\end{figure}

By the above decomposition we obtain:

\begin{align*}
\Theta_d(L_2) = \lambda_D^{-2} \, \Theta_d(L_4) &- (q-q^{-1}) \lambda_D^{-\frac{3}{2}} \, \Theta_d(U \sqcup U) - (q-q^{-1}) \sqrt{\lambda_D}^{-1} \, \Theta_d(3_1^* \sqcup U)\\
& +(q-q^{-1})^2 \, \Theta_d(3_1^*) - (q-q^{-1})^2 \, \Theta_d(5_2^*).
\end{align*}

In the skein tree of Figure~\ref{skein_tree_L2}, the knot $5_2^*$ was identified by its value of $P$. Namely, the diagram is braided and it corresponds to the braid word $\s_1^{-1} \s_2 \s_1^{-1} \s_2^{-3}$ (cutting along the grey line on the diagram).  We then computed the Homflypt polynomial on this word through the trace $\tau$, and we found that it corresponds to the knot $5_2^*$.

Further, the link $L_4$ can be easily decomposed, so we have:

$$
\Theta_d(L_4) = \Theta_d(U \sqcup U \sqcup U) + (q-q^{-1}) (\lambda_D^{-\frac{1}{2}} - \lambda_D^{\frac{1}{2}}) \, \Theta_d(U \sqcup U) - (q-q^{-1})^2 \, \Theta_d(U).
$$

Similarly, we decompose in Figure~\ref{skein_tree_L3} the link $L_3$:

\begin{figure}[H]
\begin{picture}(400,390)
\put(170,350){$L_3=$}

\put(200,310){\includegraphics[scale=0.25]{L11n418_R.pdf}}
\put(280,320){\vector(3,-4){30}}
\put(300,300){$\frac{-(q-q^{-1})}{\sqrt{\lambda_D}}$}
\put(220,320){\vector(-3,-4){30}}
\put(180,300){$\lambda_D^{-1}$}

\put(130,200){\includegraphics[scale=0.25]{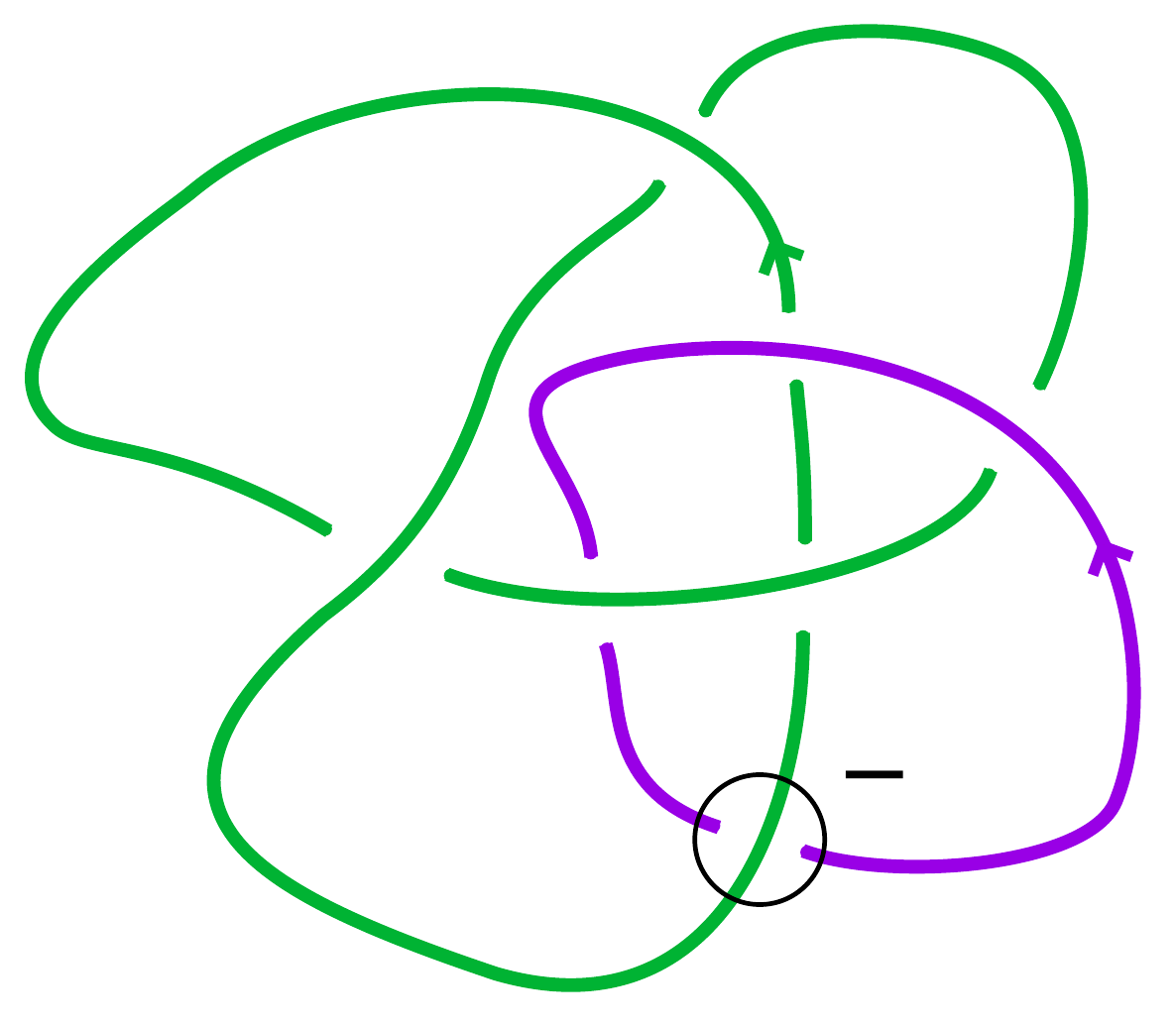}}
\put(280,200){\includegraphics[scale=0.25]{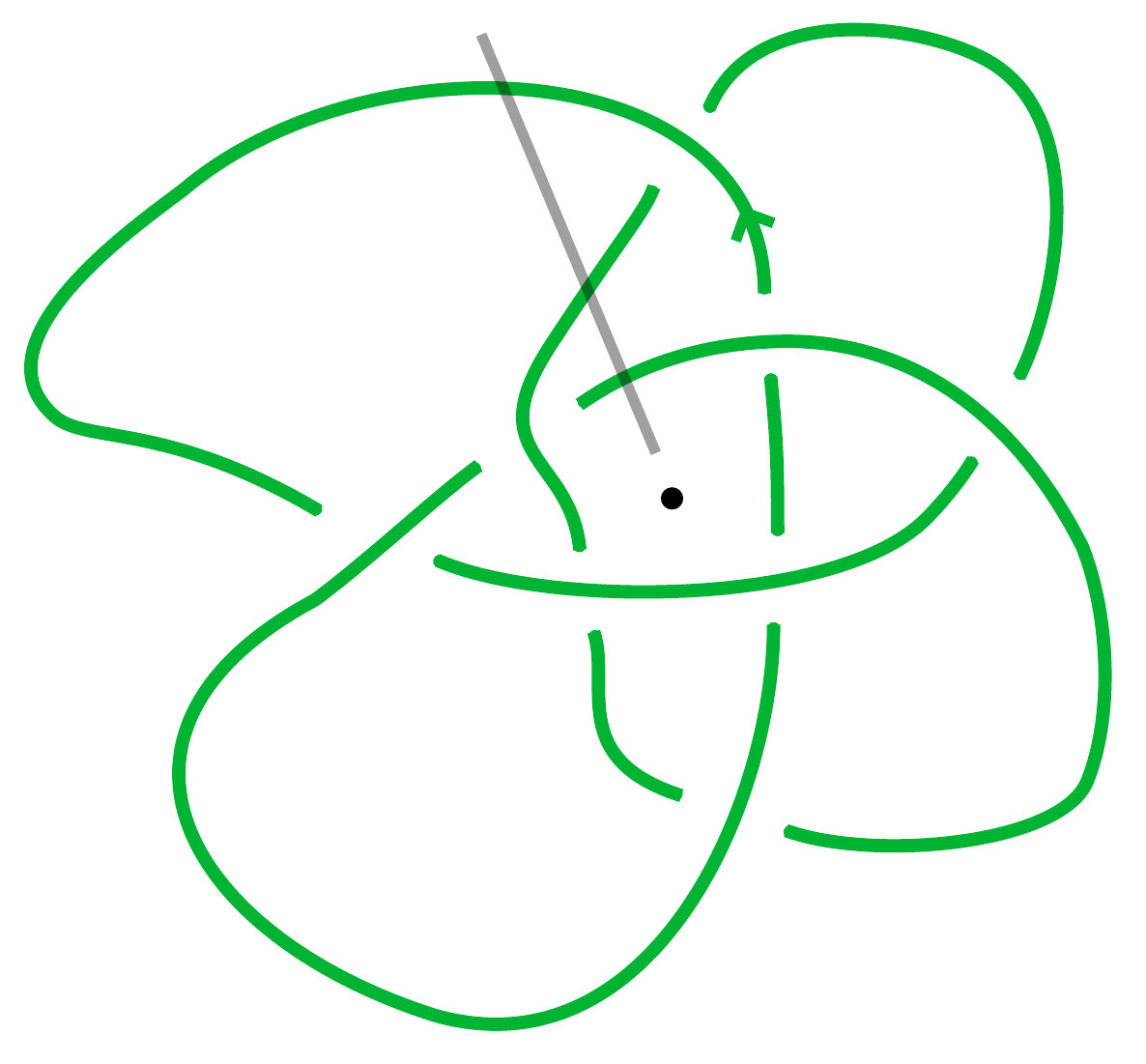}}
\put(340,200){$8_{20}$}

\put(220,210){\vector(3,-4){30}}
\put(235,190){$\frac{-(q-q^{-1})}{\sqrt{\lambda_D}}$}
\put(140,210){\vector(-3,-4){30}}
\put(100,190){$\lambda_D^{-1}$}

\put(65,100){\includegraphics[scale=0.25]{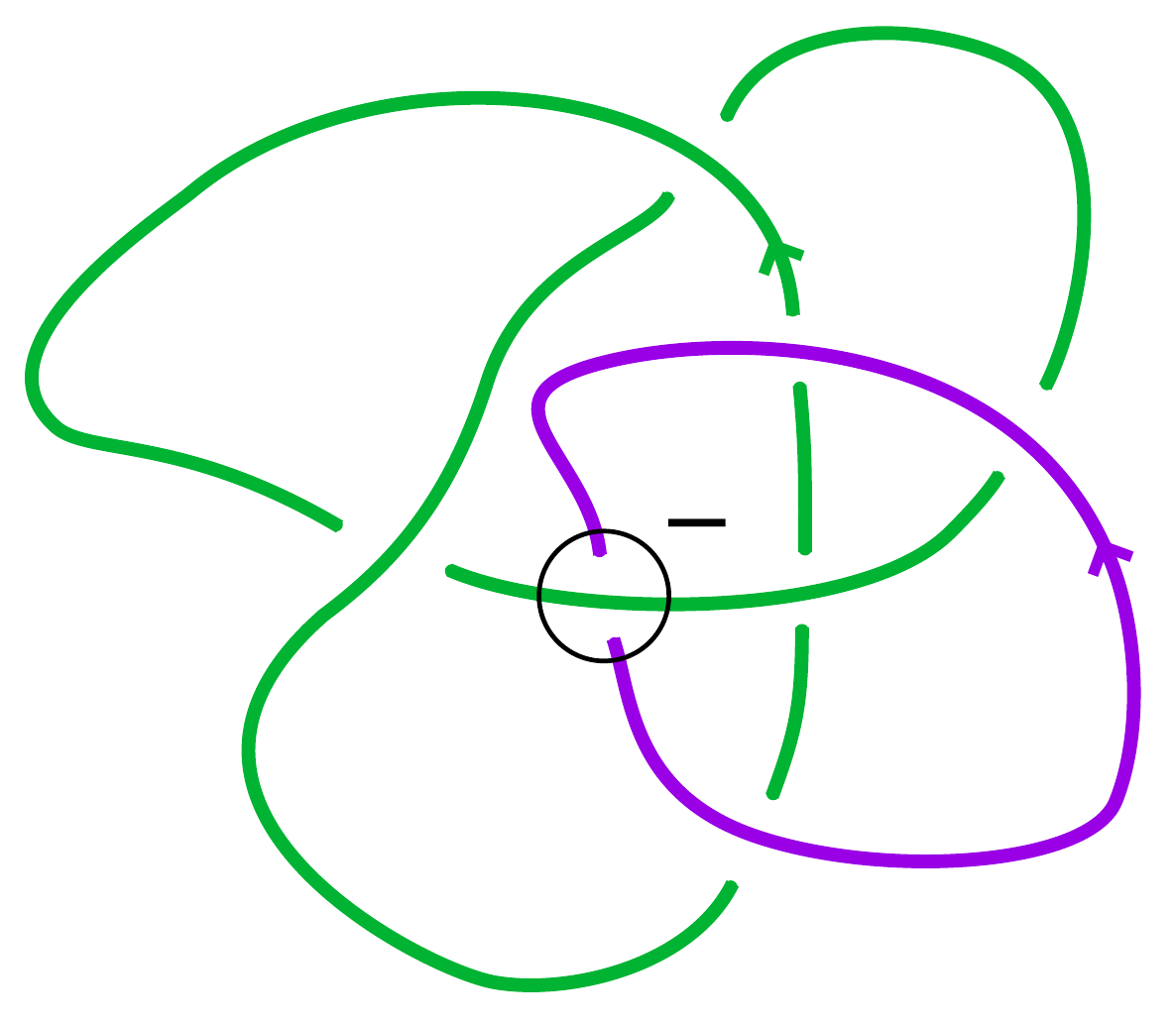}}
\put(220,100){\includegraphics[scale=0.25]{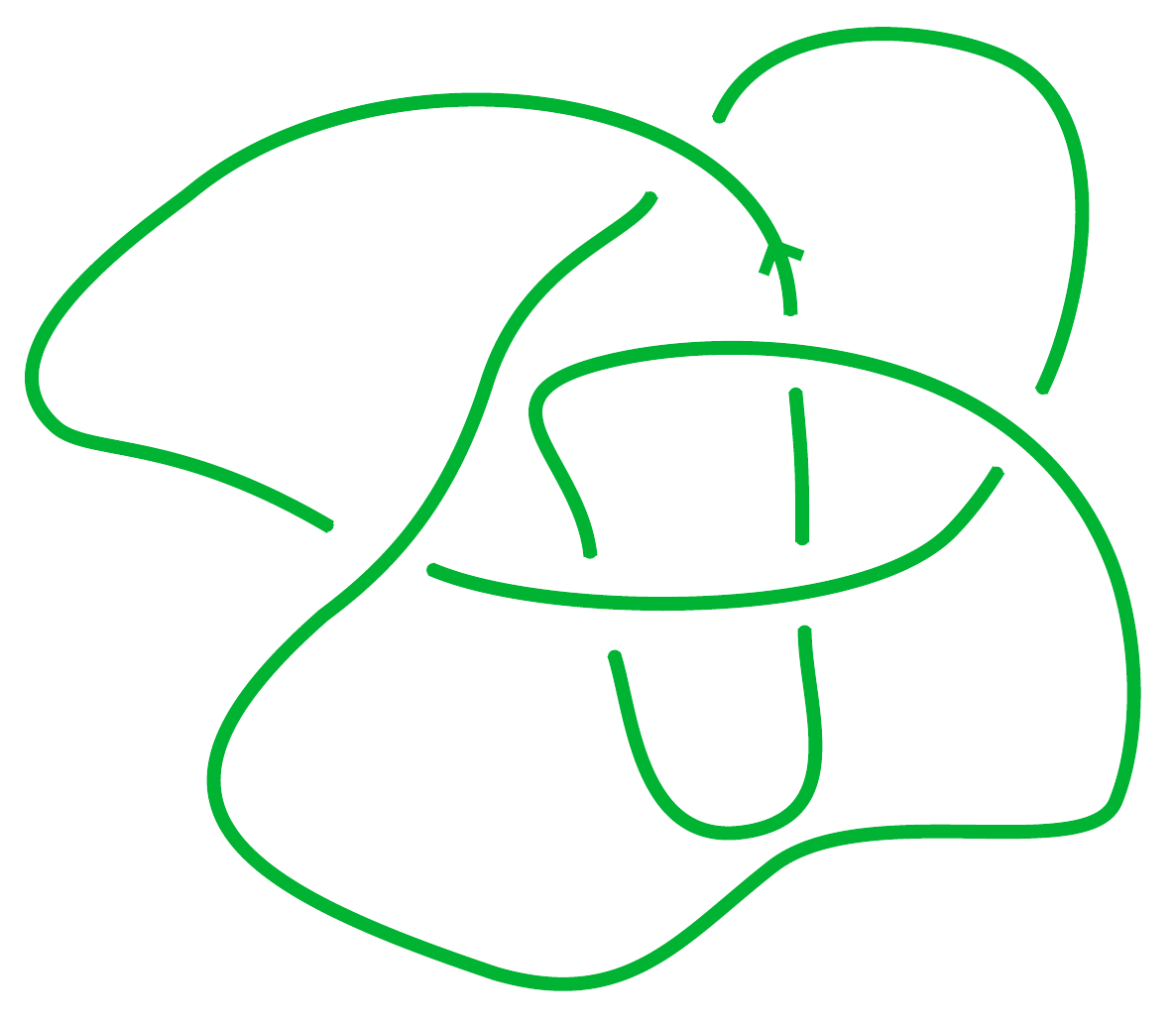}}
\put(300,105){$U$}

\put(160,120){\vector(3,-4){30}}
\put(175,100){$\frac{-(q-q^{-1})}{\sqrt{\lambda_D}}$}
\put(70,120){\vector(-3,-4){30}}
\put(30,100){$\lambda_D^{-1}$}

\put(0,10){\includegraphics[scale=0.25]{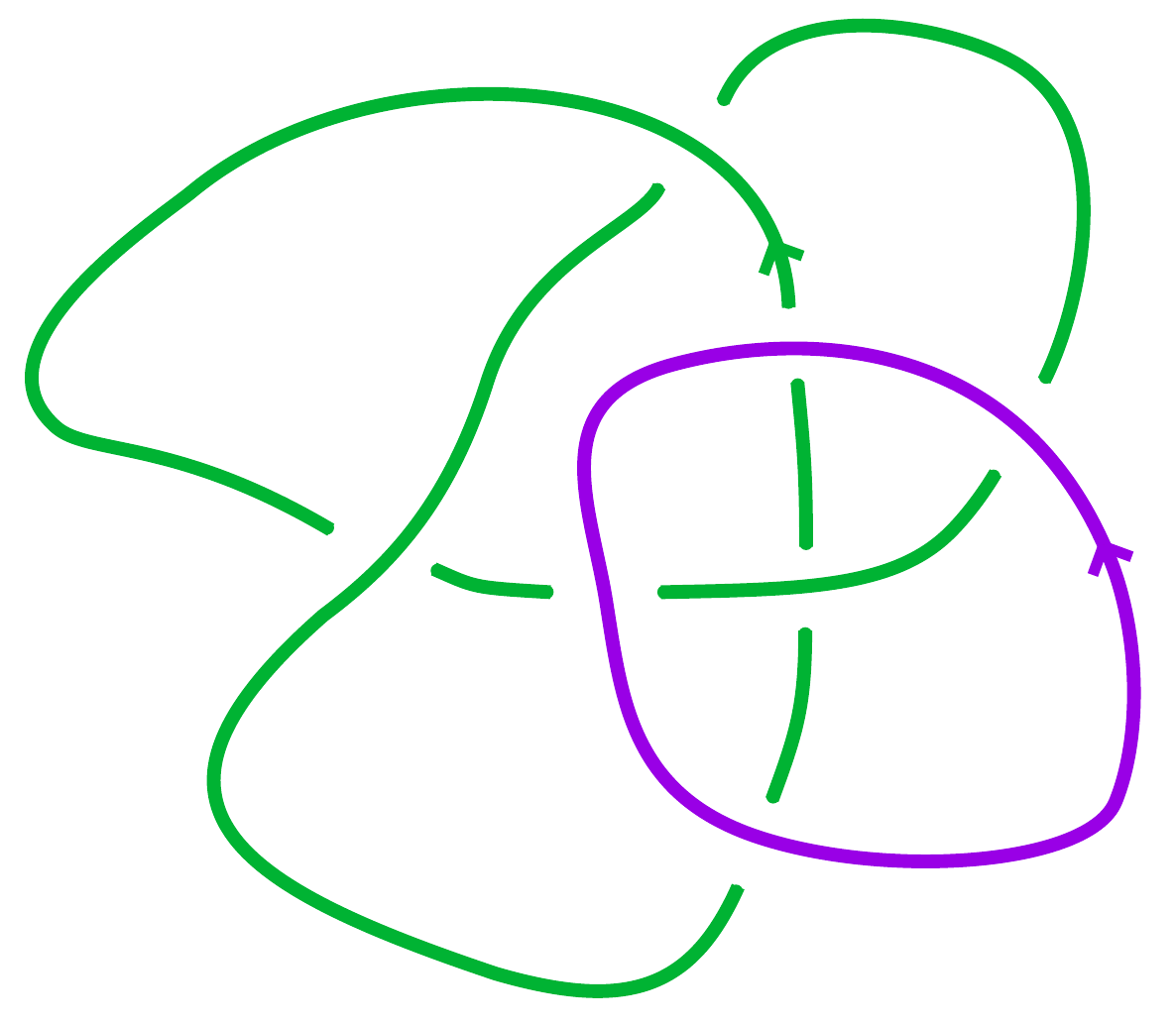}}
\put(150,10){\includegraphics[scale=0.25]{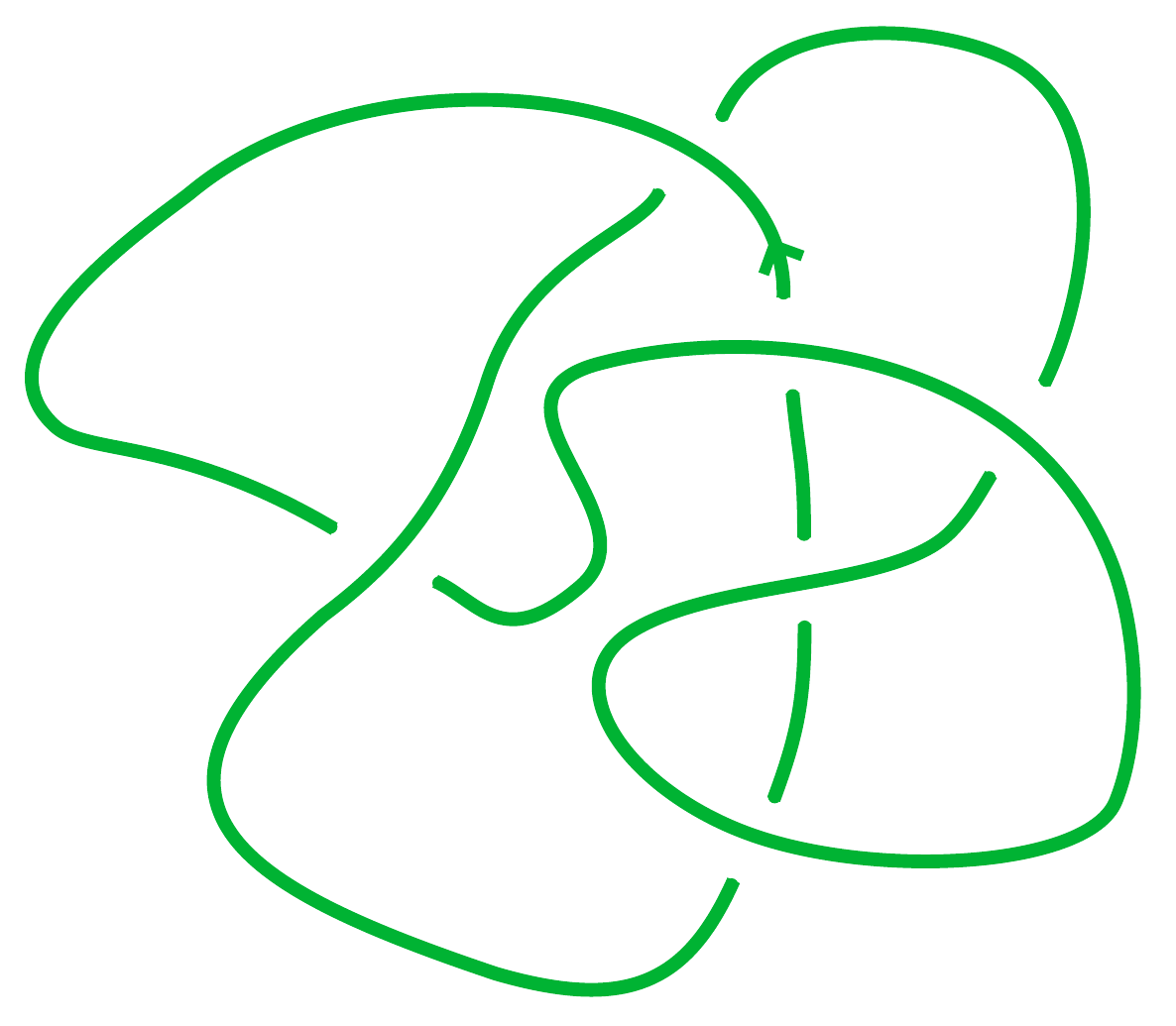}}
\put(50,0){$3_1 \sqcup U$}
\put(200,0){$3_1$}
\end{picture}
\caption{The skein tree of the link $L_3$ (skein tree of $L11n418\{0,0\}$ continued).}
\label{skein_tree_L3}
\end{figure}

The knot $8_{20}$ in Figure~\ref{skein_tree_L3} is identified by its braid word $\s_2^{-1} \s_1 \s_2^{-1} \s_1^{-1} \s_2 \s_1^{-1} \s_2^{-1} \s_1$ (cutting along the grey line). Hence, for the link $L_3$ we have:
\begin{align*}
\Theta_d(L_3) = &\lambda_D^{-3} \, \Theta_d(3_1 \sqcup U) - (q-q^{-1})\lambda_D^{-\frac{5}{2}} \, \Theta_d(3_1)\\
 &-(q-q^{-1})\lambda_D^{-\frac{3}{2}} \, \Theta_d(U) - (q-q^{-1}) \sqrt{\lambda_D}^{-1} \, \Theta_d(8_{20}).
\end{align*}

By Theorem~\ref{conjHomflypt} and Theorem~\ref{disj_links} we now obtain:

\begin{align*}
\Theta_d(L11n418\{0,0\}) = &\left(\frac{\lambda _D-1}{E_D \lambda_D^4}+\frac{\left(q-q^{-1}\right)^2}{\lambda_D^3}\right) P(3_1) + \left(\frac{\lambda _D-1}{E_D \lambda_D^2}+\frac{\left(q-q^{-1}\right)^2}{\lambda_D}\right) P(3_1^*)\\
&-\frac{\left(q-q^{-1}\right)^2}{\lambda _D} P(5_2^*) +\frac{\left(q-q^{-1}\right)^2}{\lambda_D} P(8_{20})+\frac{\left(\lambda _D^{-4}-2\lambda_D^{-3}\right) \left(1-\lambda _D\right)}{E_D}\\
&+\frac{\left(q-q^{-1}\right)^2}{\lambda_D^2}-\frac{\left(q-q^{-1}\right)^2}{\lambda_D^3}+\frac{\left(1-\lambda _D\right)^2}{E_D^2 \left(q-q^{-1}\right)^2 \lambda_D^4}.
\end{align*}

For the knots $3_1$, $5_2^*$ and $8_{20}$ we have that:

\begin{align*}
P(3_1) = &(q^{-2} + q^2 - \lambda_D) \lambda_D\\
P(5_2^*) = &\frac{\lambda_D + \lambda_D^2 + q^4 \lambda_D (1 + \lambda_D) - q^2 (1 + \lambda_D + \lambda_D^2)}{q^2 λ^3}\\
P(8_{20}) = &\frac{\lambda_D + q^8 \lambda_D + q^4 \lambda_D (2 + \lambda_D) - q^2 (1 + \lambda_D^2) - q^6 (1 + \lambda_D^2)}{q^4 \lambda_D^2}.
\end{align*}

Using Theorem~\ref{disj_links} and the value $P(3_1^*)$ from the analysis of the link $L11n358\{0,1\}$, we get the final result:

\begin{align*}
\Theta_d(L11n418\{0,0\}) = &\left( E_D^2 \lambda_D^4 q^6
   \left(q^2-1\right)^2 \right)^{-1} \Big[ E_D^2 \left(q^2-1\right)^4 \Big(\lambda_D^2+\lambda_D^2
   q^8+\lambda_D \left(-\lambda_D^2+\lambda_D -2\right) q^6\\
   &+\left(4
   \lambda_D^2-\lambda_D +1\right) q^4 +\lambda_D \left(-\lambda_D
   ^2+\lambda_D -2\right) q^2\Big)\\
   &+E_D \left(q^2-1\right)^2 q^4
   \big[2 (\lambda_D -1) \lambda_D +2 (\lambda_D -1) \lambda_D
   q^4\\
   &-\left(\lambda_D^3-3 \lambda_D^2+4 \lambda_D-2\right)
   q^2\big]+(\lambda_D -1)^2 q^8 \Big],
\end{align*}
which can be easily confirmed computationally.
\smallbreak
By using a computer algebra system, it is easy to see that the pair of the two links is not $\Theta_d$-equivalent for every 
$d \geq 2$.
The proof of Theorem~\ref{ThetaNotP} is now concluded.
\end{proof}

\begin{rem}\rm
The intrinsic difference in computing the invariants $\Theta_d$ and $P$ lies on the different values 
of these invariants on disjoint unions of knots. In particular, if $K$ is a knot and $U$ is the unknot, for the invariants $\Theta_d$ it holds that:
$$
\Theta_d(K \sqcup U) = \frac{1-\lambda_D}{(q-q^{-1}) \sqrt{\lambda_D} E_D} \Theta_d(K),
$$
while for $P$ we have that:
$$
P(K \sqcup U) = \frac{1-\lambda_D}{(q-q^{-1}) \sqrt{\lambda_D}} P(K).
$$
\end{rem}

Following Theorem~\ref{ThetaNotP}, Theorem~\ref{skein_defined} and Remark~\ref{kauffman polynomial}, we have obtained a new family of skein invariants for classical links. The fact that the invariants $\Theta_d$ can be defined diagrammatically through skein relations is very important for two reasons: 
\begin{enumerate}[(a)]
\item the computation of the invariants $\Theta_d$ becomes much easier, and \smallbreak
\item there are very few other skein invariants in the literature (such as the Alexander--Conway polynomial, the Jones polynomial, the Homflypt polynomial and the Kauffman polynomial).
\end{enumerate}

\subsection{Comparing the properties of the invariants $\Theta_d$ and $P$}\label{properties}
The invariants $\Theta_d$ satisfy some of the properties of the Homflypt polynomial. More specifically, for $L, L'$ arbitrary oriented links, these are:

\begin{itemize}
\item \textit{Reversing orientation:} $$\Theta_d(L) = \Theta_d(\overleftarrow{L}),$$ where $\overleftarrow{L}$ is the link $L$ with reversed orientation on all components.
\item \textit{Split links:} $$\Theta_d(L \sqcup L') = \frac{1-\lambda_D}{(q-q^{-1})\sqrt{\lambda_D} E_D} \Theta_d(L) \Theta_d(L').$$
\item \textit{Connected sums:} $$\Theta_d(L \# L') = \Theta_d(L) \Theta_d(L').$$
\item \textit{Mirror images:} $$\Theta_d(q,\lambda_D)(L^*) = \Theta_d(q^{-1},\lambda_D^{-1})(L),$$ where $L^*$ is the mirror image of $L$.
\end{itemize}

All the above properties have been proved algebraically via the Markov traces $\Jtr$ and $\Jtrs$ in \cite[Section 3]{chjakala}. However, the behaviour of the invariants $\Theta_d$ under \textit{mutation} could not be examined algebraically. The operation of mutation on a link diagram is defined by choosing a disk which intersects the diagram at exactly four points and then rotating 180$^\circ$ the $2$-tangle encircled by the disk. It is known that the Homflypt polynomial, being a skein invariant, does not distinguish mutant links. A natural question is whether the invariants $\Theta_d$ exhibit the same behaviour, since they satisfy a skein relation. Indeed, the invariants $\Theta_d$ do not distinguish links differing by mutation, as stated in the following:

\begin{prop}\label{mutants}
Let $L$ and $L'$ be two mutant links. Then $\Theta_d(L)=\Theta_d(L')$.
\end{prop}

\begin{proof}
Consider as previously a diagram of the link $L$ where we apply the mutation operation on a $2$-tangle $T$ to obtain the link diagram of $L'$. Now, we apply on $L$ the $L'$ the splitting algorithm described in Subsection~\ref{subs_skein}, but only on crossings which lie outside the $2$-tangle $T$, simultaneously on both diagrams and after assigning them the same ordering of components. Hence, we obtain the following decompositions:
$$
P(L) = \sum_{k=1}^m c(L_k) P(L_k) \quad\text{ and }\quad P(L') = \sum_{k=1}^m c(L'_k) P(L'_k),
$$
for some $m \in \N$, where $L_k$ and $L'_k$ are the resulting links and $c(L_k), c(L'_k) \in \mathcal{A}$ for every $k$. By the application of the skein relation on the exact same crossings at both links, we have that $c(L_k) = c(L'_k)$. Since $L$ and $L'$ are mutants, the links $L_k$ and $L'_k$ for every $k=1,\ldots,m$ are mutants too, so $P(L_k)=P(L'_k)$. Further, every link pair $L_k$ and $L'_k$ is a disjoint union of the same knots and a closed $2$-tangle, where the closing strands outside of $T$ may be knotted in the same manner. Therefore, due to the Theorem~\ref{disj_links} we can restrict ourselves to the case of closed $2$-tangles as illustrated in Figure~\ref{fig_tangle}, where $K_1$ and $K_2$ are (possibly knotted) $1$-component $1$-tangles. 

\begin{figure}[H]
\begin{center}
\begin{picture}(160,80)
\qbezier(60,25)(80,25)(100,25)
\qbezier(60,55)(80,55)(100,55)
\qbezier(60,25)(60,40)(60,55)
\qbezier(100,25)(100,40)(100,55)

\qbezier(10,30)(20,30)(30,30)
\qbezier(10,50)(20,50)(30,50)
\qbezier(10,30)(10,40)(10,50)
\qbezier(30,30)(30,40)(30,50)

\qbezier(130,30)(140,30)(150,30)
\qbezier(130,50)(140,50)(150,50)
\qbezier(130,30)(130,40)(130,50)
\qbezier(150,30)(150,40)(150,50)

\qbezier(25,50)(45,90)(65,55)
\qbezier(25,30)(45,-10)(65,25)
\qbezier(95,55)(115,90)(135,50)
\qbezier(95,25)(115,-10)(135,30)

\put(77,35){\Large $T$}
\put(13,35){$K_1$}
\put(133,35){$K_2$}

\end{picture}
\caption{A closed $2$-tangle where $K_1$ and $K_2$ are possibly knotted $1$-component $1$-tangles.}
\label{fig_tangle}
\end{center}
\end{figure}
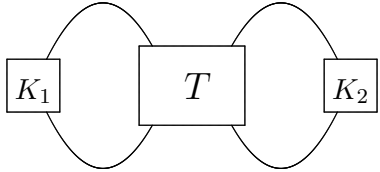

Let $L_T$ be a closed $2$-tangle described as above and let $L'_T$ be the mutant link obtained by applying the mutation operation on $T$, considering the appropriate diagrams. Inside the $2$-tangle $T$ there may be some closed components (\emph{i.e.,} not involving the $1$-tangles $K_1$ and $K_2$). We unlink these components by applying the special skein relation on both diagrams simultaneously
and we then apply Theorem~\ref{disj_links}. 
Note that the operation of mutation does not change the sign of a crossing (positive or negative), thus by the above procedure we get a decomposition of the links $L_T$ and $L'_T$ into pairs of mutant links with the same coefficients. So, as previously, we can reduce to the case where no extra components exist inside the $2$-tangle $T$, since the decomposition by the skein relation will involve the same coefficients for both diagrams.

Now there are only two possibilities for $L_T$ and $L'_T$: either they are both $2$-component links, or they are knots. If they are both knots, then by Theorem~\ref{conjHomflypt}, $\Theta_d(L_T) = P(L_T) = P(L'_T) = \Theta_d(L'_T)$, whence the proposition holds. If they are $2$-component links, using Theorem~\ref{twocomp}, we obtain:
$$
\Theta_d(L_T) = P(L_T) + \lambda_D^{\xi(L_T)} (E_D^{-1} - 1) P(I_T),
$$
and
$$
\Theta_d(L'_T) = P(L'_T) + \lambda_D^{\xi(L'_T)} (E_D^{-1} - 1) P(I'_T),
$$
where $I_T$ and $I_T'$ are the split versions of $L_T$ and $L'_T$ respectively. The links $L_T$ and $L'_T$ have the same value on the polynomial $P$ as mutants;  they also have the same number of clasps, so $\xi(L_T) = \xi(L'_T)$. Further, the links $I_T$ and $I'_T$ are also mutants, since applying the skein relation on a crossing inside $T$ results in two new mutant links, and so $P(I_T) = P(I'_T)$. Consequently, $\Theta_d(L_T) = \Theta_d(L'_T)$.
\end{proof}

\begin{rem} \rm \label{oldquadr8}
Clearly, the diagrammatic analysis made for the invariants $\Theta_d$ on pairs of $P$-equivalent links cannot be implements for the invariants $\Delta_d$. Nevertheless, there are computational indications that the invariants $\Delta_d$ are not topologically equivalent to $P$, and neither to $\Theta_d$; for more details see Section~\ref{threeVar}. Concerning now the properties studied in Subsection \ref{properties}, $\Delta_d$ has the same behaviour as $\Theta_d$ on links with reversed orientation, on split links, on connected sums and on mirror images. However, behaviour of $\Delta_d$ under mutation cannot be checked using the methods of Proposition~\ref{mutants}.
\end{rem}

\section{A new 3-variable skein link invariant}\label{threeVar}
In this section we construct a new 3-variable skein link invariant $\Theta$ generalizing both the invariants $\Theta_d$ and the Homflypt polynomial.

\subsection{The invariant $\Theta$}
Following Theorem~\ref{skein_defined}, the invariants $\Theta_d$ can be completely defined via the skein relation \eqref{new skein delta} and the value of $\Theta_d$ on disjoint unions of knots given by Theorem~\ref{disj_links}, which is in turn a generalization of  the property $\Theta_d(K) = P(K)$,  for $K$ a knot (Theorem \ref{conjHomflypt}). 
It is evident that when the invariants are defined via the skein relation, the value $E_D$ can be generalized to an indeterminate, since the E--system is not involved. Consequently, a 3-variable skein invariant $\Theta(q,\lambda,E)$, generalizing the invariants $\Theta_d(q,\lambda_D)$, can de defined as follows:

\begin{thm}\label{invE}
Let $q,\,\lambda,\, E$ be indeterminates.
There exists a unique isotopy invariant  of classical oriented links $\Theta: \mathcal{L} \rightarrow \C[q^{\pm 1}, \lambda^{\pm 1}, E^{\pm 1}]$ defined by the following rules:
\begin{enumerate}
\item For a disjoint union $L$ of $k$ knots, with $k \geq 1$, it holds that:
$$\Theta(L) =  E^{1-k} \,P(L).$$
\item On crossings involving different components the following skein relation holds:
$$
\frac{1}{\sqrt{\lambda}} \, \Theta(L_+) - \sqrt{\lambda} \, \Theta(L_-) = (q-q^{-1}) \, \Theta(L_0),
$$
where $L_+$, $L_-$, $L_0$ is a Conway triple. \smallbreak
\end{enumerate}
\end{thm}

It is clear that the invariant $\Theta$ satisfies all the results and properties proved in the previous sections for $\Theta_d$. In particular, $\Theta$ satisfies all properties of the invariants $\Theta_d$ discussed in Subsection~\ref{properties}. Further, by definition, $\Theta$ coincides with the Homflypt polynomial on knots. Finally, the pairs of links of Table~\ref{pairs} are distinguished by all the invariants $\Theta_d$ for $d\geq 2$, and the difference of the values of the invariants on any one of the pairs of links is divisible by the factor $(E_D - 1)$. When computing the invariant $\Theta$, this factor simply changes to $(E-1)$, and so $\Theta$  distinguishes these six pairs too. Hence, Theorem~\ref{ThetaNotP} also holds for the invariant $\Theta$. Since the invariant $\Theta$ contains both the Homflypt polynomial $P$ (for $E=1$) and the invariants $\Theta_d$, we conclude the following:

\begin{thm}\label{DeltaStrongerP}
The invariant $\Theta(q, \lambda, E)$ is stronger than the Homflypt polynomial.
\end{thm}

Moreover, by Remark~\ref{kauffman polynomial}, the invariant $\Theta$ is different from the Kauffman polynomial. However, despite the generalization of the family $\Theta_d$ to the invariant $\Theta$ we do not expect that $\Theta$ is any stronger than the invariants $\Theta_d$ when $d$ is taken to be generic. 

\begin{rem} \rm
The computer program of Remark~\ref{new_program} computes in reality the invariant $\Theta$, since it treats the value $E_D$ as a parameter.
\end{rem}

We shall now proceed with the proof of Theorem~\ref{invE}. For $E = 1/d$, where $d \in \N$, the invariant $\Theta$ is well-defined, since $\Theta(q,\lambda,1/d) = \Theta_d(q,\lambda_D)$ by Theorem ~\ref{skein_defined}. 
We need to show that $\Theta$ is well-defined for any $E$. A way to prove this is by adapting the diagrammatic method of Lickorish--Millett \cite{LM}, as described in Remark~\ref{theta LM} for the invariants $\Theta_d$. 
Such a proof is given in \cite{kala}, where the invariants $\Theta_d$ and $\Theta$ are placed in a more general context. 
However, we opt here for a different method, namely, by comparing $\Theta$ with a 3-variable invariant $\overline{\Theta}$ for tied links and by showing that $\Theta$ coincides with $\overline{\Theta}$ on classical links.

\subsection{Tied links}
Tied links were introduced and studied by F.~Aicardi and J.~Juyumaya in \cite{AJ,AJ2}.
A \textit{tied link} is a classical link $L$ endowed with a set of ties, containing unordered pairs of points belonging to the components of $L$ \cite[Definition 1]{AJ2}. Diagrammatically, one can visualize a tie as a spring connecting two (not necessarily different) components of $L$. 
The endpoints of a tie are allowed to slide along the components that they are attached to. If two ties join the same two components, one of them can be removed, and any tie on a single component can be also removed. A tie that cannot be removed is called \textit{essential}. 

The invariant $\overline{\Theta}$ of tied links is constructed with the use of a Markov trace on the \emph{algebra of braids and ties} $\mathcal{E}_n(q)$, which in Remark~\ref{conj braids and ties} is conjectured to be isomorphic to the subalgebra $\YH(q)^{\rm (br)}$ of $\YH(q)$ for $d \geq n$. The algebra of braids and ties $\mathcal{E}_n(q)$ is defined as the algebra generated by $g_1, \ldots, g_{n-1}, e_1, \ldots, e_{n-1}$ satisfying the following relations (cf.~\cite[Definition 1]{AJ}): 
\begin{center}
$\begin{array}{rcll}
g_i g_j g_i &=& g_j g_i g_j&  \text{for } \vert i - j \vert = 1\\
g_i g_j &=& g_j g_i& \text{for } \vert i - j \vert > 1\\
e_i e_j &=& e_j e_i&\\
e_i^2 &=& e_i&\\
e_i g_i &=& g_i e_i&\\
e_i g_j &=& g_j e_i& \text{for } \vert i -j \vert > 1\\
e_i e_j g_i &=& g_i e_i e_j&  \text{for } \vert i - j \vert = 1\\
e_i g_j g_i &=& g_j g_i e_j&  \text{for } \vert i - j \vert = 1\\
g_i^2 &=& 1 + (q-q^{-1}) \,e_i g_i \, .
\end{array}$
\end{center}

Diagrammatically, the generators $g_i$ correspond to the classical braiding generators and the elements $e_i$ correspond to ties connecting the $i$-th and the $(i+1)$-th strands. Note that the cancellation properties of ties mentioned above are reflected in the fact that the elements $e_i$ are idempotents. 

A Markov trace $\rho: \bigcup_{n \geq 0} \mathcal{E}_n(q) \rightarrow \C[q^{\pm 1}, z^{\pm 1}, E^{\pm 1}]$ can be defined satisfying the following rules (cf.~\cite[Theorem 3]{AJ}):
$$
\begin{array}{crcll}
{\rm (i)} & \rho(a b) & = & \rho(b a) & \qquad a, b\in \mathcal{E}_n(q) \\
{\rm (ii)} & \rho(1) &  = & 1 & \qquad 1 \in \mathcal{E}_n(q) \\
{\rm (iii)} & \rho(a g_n)  & = & z\, \rho( a) & \qquad a\in \mathcal{E}_n(q) \quad (\textit{Markov  property} ) \\
{\rm (iv)} & \rho(a e_n) & = & E\, \rho( a) & \qquad a \in \mathcal{E}_n(q) \\
{\rm (v)} & \rho(a e_n g_n)  & = & z\, \rho( a) & \qquad a\in \mathcal{E}_n(q) \,.
\end{array}
$$
\noindent
Notice the resemblance of the above rules with the five rules of $\Jtrs$ in Theorem~\ref{specialtrace}.  Further, the Markov trace $\rho$  satisfies Lemma~\ref{lem and conj}, Proposition~\ref{trace_ei} and Corollary~\ref{trace_ei_cor}. The proof lies in the fact that if the $i$-th and $k$-th strands of a braid belong to the same component, then $e_{i,k}$ is a tie connecting these two strands, so it can be removed.

Now, in \cite{AJ2} the \emph{tied braid monoid} $TB_n$ is defined; it is generated by the braiding generators $\s_1, \ldots, \s_{n-1}$ and the generating ties $\eta_1, \ldots, \eta_{n-1}$, where $\eta_i$ connects the $i$-th and the $i+1$-th strands of a tied braid. Denote by $\bar{\pi}: \C TB_n \rightarrow \mathcal{E}_n(q)$ the natural surjection defined by $\s_i \mapsto g_i$ and $\eta_i \mapsto e_i$. Then the invariant $\overline{\Theta}$ is defined as:

\begin{thm}\label{tied_inv}
For any tied braid $\a \in TB_n$, we define
$$
\overline{\Theta}(\widehat{\a}) := \left( \frac{1}{z \sqrt{\lambda}} \right)^{n-1} \sqrt{\lambda}^{\epsilon(\a)} (\rho \circ \bar{\pi})(\a) \, ,
$$
 where $\lambda = \frac{z-(q-q^{-1}) E}{z}$ and $\epsilon(\a)$ is the sum of the exponents of the braiding generators $\s_i$ in the word $\a$. Then the map
 $\overline{\Theta}$ is a $3$-variable isotopy invariant of oriented tied links.
\end{thm}

Note that, for $E=1$, $\overline{\Theta}$ specializes to the Homflypt polynomial when restricted to classical links.

\begin{rem}\label{tiedinvs_nosame}  \rm
In \cite{AJ, AJ2}, F.~Aicardi and J.~Juyumaya work on the algebra of braids and ties  $\mathcal{E}_n(u)$, generated by 
elements $\wt{g}_1,\ldots,\wt{g}_{n-1},e_1,\ldots,e_{n-1}$, with the braiding generators satisfying the old quadratic relations \eqref{THE oldquadr}. They define a Markov trace $\wt{\rho}$ on $\cup_{n\geq 0} \mathcal{E}_n(u)$\cite[Theorem 3]{AJ}, which also satisfies Lemma~\ref{lem and conj}, Proposition~\ref{trace_ei} and Corollary~\ref{trace_ei_cor}. This Markov trace gives rise 
to a  $3$-variable isotopy invariant of  tied links, which is denoted by  $\overline{\Delta}$. Our construction of $\overline{\Theta}$ is completely analogous to the construction of $\overline{\Delta}$. For $E=1$, $\overline{\Delta}$ also specializes to the Homflypt polynomial when restricted to classical links.

In \cite{AJ2}, $\overline{\Delta}$ is re-defined diagrammatically via a skein relation, which applies to any crossing in the link diagram. It is proved to be well-defined via the standard Lickorish--Millett method. Note that in \cite{AJ} $\overline{\Delta}$ is defined only on classical links, while in \cite{AJ2} $\overline{\Delta}$ is extended to the class of tied links.

The invariant $\overline{\Delta}$ has not been identified topologically. One obstruction to this is the fact that the old quadratic relation for the algebra of braids and ties is used. Therefore, it was impossible to derive a special skein relation which only involves classical links (with no ties).

Despite the fact that the algebras $\mathcal{E}_n(u)$ and $\mathcal{E}_n(q)$ are isomorphic, 
the invariants $\overline{\Delta}$ and $\overline{\Theta}$ are not necessarily equivalent (as we have already observed about the invariants $\Delta_d$ and $\Theta_d$ in Remark~\ref{invs_nosame}). Indeed, we have computational indications \cite{Ai} that the invariant $\overline{\Delta}$ is stronger than the Homflypt polynomial and that there are pairs of $\overline{\Theta}$-equivalent links which are distinguished by $\overline{\Delta}$ and vice versa. These pairs exhibit the same behaviour when replacing $\overline{\Theta}$ by  $\Theta_d$ and $\overline{\Delta}$ by $\Delta_d$.
\end{rem}

Now we proceed to the proof of Theorem~\ref{invE}.

\subsection{Proof of Theorem~\ref{invE}} \label{invE_proof}
We will prove that $\Theta$ and $\overline{\Theta}$ coincide on classical links and that they are defined by the same rules. 
In \cite{AJ2} it is shown that a skein relation holds for the invariant $\overline{\Delta}$, where the crossing involved is \textit{any crossing in the link diagram}. Let $L_+, L_-, L_0$ be a Conway triple of tied links. With the old quadratic relation, this skein relation involves the links $L_+, L_-$ and the links $L_{0,\sim}$ and $L_{+,\sim}$, which are the links $L_0$ and $L_+$ with a tie connecting the two arcs involved.
However, with the use of the new quadratic relation, the skein relation transforms to:
\begin{equation}\label{tied_skein}
\frac{1}{\sqrt{\lambda}}\overline{\Theta}(L_+) - \sqrt{\lambda} \, \overline{\Theta}(L_-) = (q-q^{-1})\, \overline{\Theta}(L_{0,\sim}).
\end{equation}

If we apply the skein relation \eqref{tied_skein} on a crossing of a tied link $L$ involving different components, then the link $L_{0,\sim}$ will have a tie connecting the same component. Such a tie is not essential, and thus it can be removed. Hence, we obtain $L_{0,\sim} = L_0$. So, if the original link $L$ contains no ties, then the skein relation \eqref{tied_skein} involves only classical links when applied on crossings between \textit{different components}. Thus, \eqref{tied_skein} is satisfied by the invariant $\overline{\Theta}$ when restricted on classical links, and so it coincides with the special skein relation satisfied by the invariant $\Theta$.

We will now look into the behaviour of $\overline{\Theta}$ on knots and on disjoint unions of knots.
For this, we need the following observation: when every 
 two components of a link $L$ are connected by a tie, then the skein relation \eqref{tied_skein} can be equivalently transformed to the following:
\begin{equation}\label{tied_skein_2}
\frac{1}{\sqrt{\lambda}}\overline{\Theta}(L_{+,\sim}) - \sqrt{\lambda} \, \overline{\Theta}(L_{-,\sim}) = (q-q^{-1})\, \overline{\Theta}(L_{0,\sim}).
\end{equation}
This is due to the fact that on any crossing we can introduce a non-essential tie, since any two components are already tied. Relation~\eqref{tied_skein_2} holds in particular for any classical \textit{knot} $K$, since ties can be introduced freely on knots.
Trying to resolve the skein tree for $K$ using relation~\eqref{tied_skein_2}, in order to compute $\overline{\Theta}(K)$, the process will terminate to a sequence of tied unlinks, whose components are all tied together. 
The invariant $\overline{\Theta}$ on any of these tied unlinks can be computed by \eqref{tied_skein} inductively on the number of components. During all this process for $\overline{\Theta}$ we have only used coefficients of the Homflypt polynomial, hence $\overline{\Theta}(K) = P(K)$ for any knot $K$, exactly as our invariant $\Theta$. Using now this property and the definition of $\overline{\Theta}$  given by Theorem~\ref{tied_inv}, we can repeat the proof of Theorem \ref{disj_links} and obtain that, for any disjoint union $L$ of $k$ knots, we have
\begin{equation}\label{tied_disj_link}
\overline{\Theta}(L) =  E^{1-k} \,P(L).
\end{equation}

We conclude that the invariant $\overline{\Theta}$ satisfies rules (1) and (2) of Theorem~\ref{invE}. Following the procedure described in Subsection~\ref{subs_skein} (replacing
everywhere $\Theta_d$ with $\overline{\Theta}$), these two rules suffice for the computation of the value of $\overline{\Theta}$ on any classical link. Since $\Theta$ is defined using the exact same properties, it must coincide with the invariant $\overline{\Theta}$, and thus it is a well-defined invariant of classical links.
\qed

\begin{rem} \rm
The invariant $\Theta$ does not involve complicated constructions such as the E--system, even though it contains the invariants $\Theta_d$ where the E-system is needed. Apart from the proof above another way of proving that $\Theta$ is well-defined would be to take the algebraic approach and show that the subalgebra ${\rm Y}_{d,n}^{\rm (br)}(q)$ of $\YH(q)$ is isomorphic to the algebra of braids and ties $\mathcal{E}_n(q)$ for $d \geq n$ (recall Remark~\ref{E as a param}). Note that the E-system would not appear via this method either, since $E$ would be taken to be a parameter. We also note that, a diagrammatic skein theoretic proof is given in \cite{kala}.
\end{rem}

\begin{rem} \rm \label{oldquadr9}
Similarly to the proof of Theorem~\ref{invE}, one can prove \eqref{tied_disj_link} for the invariant $\overline{\Delta}$. Given the isomorphism between the subalgebra ${\rm Y}_{d,n}^{\rm (br)}(u)$ of $\YH(u)$ and the algebra of braids and ties $\mathcal{E}_n(u)$ for $d \geq n$, the invariant $\overline{\Delta}$, when restricted to classical links, contains the invariants $\Delta_d  = \Delta_{d',D'}$ for any $d' \geq d$ and $|D'|=d$  (recall Remark~\ref{E as a param u}).
Consequently, also the invariants $\Delta_d$  are equivalent to the Homflypt polynomial on knots and on disjoint unions of knots. 
\end{rem}

\subsection{A closed formula for $\Theta$}
In Appendix~\ref{lickorish}, W.B.R. Lickorish proves in Theorem~\ref{theta_linking_P} that the 3-variable invariant $\Theta$ is a sum of Homflypt polynomials of sublinks of an oriented link, which involves also linking numbers between different components of the link. More precisely, for an $n$-component oriented link $L$:
\begin{equation}\label{lickorish_forumla}
\Theta (L) = \sum_{k=1}^n \mu^{k-1}E_k \sum_\pi \lambda^{\nu(\pi)}P(\pi L),
\end{equation}
where the second summation is over all partitions $\pi$ of the components of $L$ into $k$ (unordered) subsets and $P(\pi L)$ denotes the product of the Homflypt polynomials of the $k$ sublinks of $L$ defined by $\pi$. Furthermore, $\nu(\pi)$ is the sum of all linking numbers of pairs of components of $L$ that are in distinct sets of $\pi$, $E_k = (E^{-1} - 1)(E^{-1} - 2) \cdots (E^{-1} - k + 1)$, with $E_1 =1$, and $\mu = \frac{\lambda^{-{1/2}} - \lambda^{{1/ 2}}}{q - q^{-1}}$.

\smallbreak
Theorem~\ref{theta_linking_P} gives a better topological interpretation of the invariant $\Theta$ than the one obtained by the special skein relation. \textit{The invariant $\Theta$ is completely determined by the linking matrix of a link $L$ and the values of $P$ on each sublink of $L$}. Hence, pairs of $P$-equivalent links containing sublinks which are not $P$-equivalent are prime candidates to be distinguished by $\Theta$. Such an example is the pair of links of Theorem~\ref{ThetaNotP}; although both links consist of linked unknots with same linking numbers, there are 2-component sublinks of the pair that are not $P$-equivalent. Indeed, the link $L11n358\{0,1\}$ contains a disjoint union of two unknots as a sublink, whereas $L11n418\{0,0\}$ does not.

\begin{rem} \rm
Theorem~\ref{theta_linking_P} has been proved independently in \cite{pawa} using representation theory techniques.
\end{rem}

Further, Theorem~\ref{theta_linking_P} answers the question of how the invariants $\Theta_d$ compare among themselves for different values of $d$. In detail:

\begin{prop}\label{theta_d_n}
Let $L$ and $L'$ be two $n$-component links which are not $\Theta$-equivalent. Then they are not $\Theta_d$-equivalent for $d \geq n$.
\end{prop}

\begin{proof}
Since $L$ and $L'$ are not $\Theta$-equivalent, by Theorem~\ref{theta_linking_P} we have that:
$$
\sum_{k=1}^n \mu^{k-1}E_k \sum_\pi \left[ \lambda^{\nu(\pi)}P(\pi L) - \lambda^{\nu'(\pi)}P(\pi L') \right] \neq 0,
$$
where the notation is as in Theorem~\ref{theta_linking_P} and $\nu'(\pi)$ denotes the corresponding sum of linking numbers (as in $\nu(\pi)$) for the link $\pi L'$. Due to the above relation, there exists a $k \in \{1, \ldots, n\}$ and a partition $\pi$ such that $\nu(\pi) \neq \nu'(\pi)$ and/or $P(\pi L) \neq P(\pi L')$.

Denote now by $E_{k,d}$ the quantity $E_k$ evaluated at $E = 1/d$, that is:
$$
E_{k,d} = (d - 1) \cdots (d - k + 1).
$$
Substituting in Theorem~\ref{theta_linking_P} the quantity $E_k$ by $E_{k,d}$ we get the corresponding formula for the invariant $\Theta_d$. Since for $d \geq n$ it holds that $E_{k,d} \neq 0$ for all $k \in \{1, \ldots, n\}$ and from the above it follows that:
\begin{equation}\label{proof_theta_d_n}
\sum_{k=1}^n \mu^{k-1}E_{k,d} \sum_\pi \left[ \lambda^{\nu(\pi)}P(\pi L) - \lambda^{\nu'(\pi)}P(\pi L') \right] \neq 0, \qquad \text{ for } d \geq n,
\end{equation}
that is $\Theta_d(L) \neq \Theta_d(L')$ for $d \geq n$.
\end{proof}

Proposition~\ref{theta_d_n} tells us that, for $d \geq n$ we do not lose any topological information, since $E_{k,d} \neq 0$. Indeed, for a $d < n$ and for a pair of $P$-equivalent links $L$ and $L'$, which are not $\Theta$-equivalent, in the formula \eqref{proof_theta_d_n} the coefficients $E_{k,d}$, for $k \geq d$, will be equal to zero. Furthermore, for the remaining summands, for which $E_{k,d}$ is not zero, it could well happen that the corresponding sublinks $\pi L$ and $\pi L'$ are $P$-equivalent sublinks and that $\nu(\pi)=\nu'(\pi)$. In that case, the invariant $\Theta_d$ will not distinguish this pair for this value of $d < n$.

\begin{rem} \rm
The above result has also been proved in \cite{pa} using representation theory techniques. It is worth noting, that the isomorphism of $\YH^{\rm{(br)}}$ with the algebra of braids and ties \cite{EsRy} holds for $d \geq n$, which is the same condition as in the statement of Proposition~\ref{theta_d_n}.
\end{rem}

Proposition~\ref{theta_linking_P} enables us to re-define using in yet a new way the invariants $\Theta_d$ and $\Theta$. Namely, one could define the invariant $\Theta$ combinatorially by Equation~\ref{lickorish_forumla} and then prove that the definition is equivalent to the algebraic one or to the skein-theoretic one. More precisely:

\begin{thm} \label{theta_new_def}
Suppose that the invariant $\Theta$ is defined by Equation~\ref{lickorish_forumla}. Then, this definition is equivalent to the definition of Theorem~\ref{invE}.
\end{thm}

\begin{proof}
The formula~\ref{lickorish_forumla} clearly defines uniquely an oriented link invariant. The proof of Theormem,~\ref{theta_linking_P} uses precisely the defining rules of $\Theta$ as given in the statement of Theorem~\ref{invE}. So, one direction of the statement is done. For the other direction, we shall show that the defining equation~\eqref{lickorish_forumla} implies the two rules of Theorem~\ref{invE}. The proof uses similar arguments as those in the proof of Theorem~\ref{theta_linking_P}. Indeed, let $L$ be a disjoint union of $n$ knots $K_1, \ldots, K_n$ with $n \geq 1$. Then, since for every partition $\pi$ associated to some $k \in \{1,\ldots,n\}$ we have that $P(\pi L) = \mu^{1-k} P(L)$ and $\nu(\pi) = 0$, Equation~\eqref{lickorish_forumla} yields:
$$
\Theta (L) =  P(L)  \sum_{k=1}^n S(n,k) E_k,
$$
where $S(n,k)$ is the Stirling number of second kind. By the relation~\eqref{stirling} used in the proof of Theorem~\ref{theta_linking_P}, we obtain:
$$
\Theta (L) = E^{1-n} \, P(L) ,
$$
which is exactly the first rule of Theorem~\ref{invE}.

\smallbreak
Now, let $L_+, L_-, L_0$ be a Conway triple, where the crossing of $L_+$ involves two different components, denoted by $L_1$ and $L_2$. Using Equation~\eqref{lickorish_forumla}, it holds that:
\begin{equation}\label{theta_L_plus}
\Theta (L_+) = \sum_{k=1}^n \mu^{k-1}E_k \left[ \sum_\pi \lambda^{\nu(\pi)}P(\pi L_+) + \sum_\rho \lambda^{\nu(\rho)} P(\rho L_+) \right],
\end{equation}
where $\pi$ are the partitions where $L_1$ and $L_2$ belong to the same subset and $\rho$ are those were the two components belong to different subsets. For the sublinks corresponding to the partitions $\rho$ it holds that $P(\rho L_+) = P(\rho L_-)$, since the selected crossing is not there neither in $\rho L_+$ nor in $\rho L_-$, and $\nu(\rho)(\rho L_+)  = \nu(\rho)(\rho L_-) +1$, hence, $\lambda^{\nu(\rho)} P(\rho L_+) = \lambda^{\nu(\rho)+1} P(\rho L_-)$. For a partition $\pi$ the selected crossing is present both in $\pi L_+$ and in $\pi L_-$. Now, a partition $\pi$ on $L_+$ induces a corresponding partition on $L_-$ and vice versa and it also induces a partition $\pi'$ on $L_0$, since the components $L_1$ and $L_2$ belong to the same subset. Now, we apply the Homflypt skein relation \eqref{skein of hom} on the link $\pi L_+$ and we substitute in \eqref{theta_L_plus}. Hence, by the above facts Equation~\ref{theta_L_plus} becomes:
$$
\Theta(L_+) = \sum_{k=1}^n \mu^{k-1}E_k \left[ \sum_\pi \lambda^{\nu(\pi)} \lambda P(\pi L_-) + \sum_{\pi'} \lambda^{\nu(\pi')} \lambda^{1/2} (q-q^{-1}) P(\pi' L_0) + \sum_\rho \lambda^{\nu(\rho)+1} P(\rho L_-) \right].
$$
The partitions $\pi$ and $\rho$ run over all the partitions of the link $L_-$. Using now Equation~\eqref{lickorish_forumla} for the link diagrams $L_-$ and $L_0$ we obtain:
$$
\Theta(L_+) = \lambda \, \Theta(L_-) + \lambda^{1/2} (q-q^{-1}) \, \Theta(L_0),
$$
which is exactly the special skein relation in the second rule of Theorem~\ref{invE}.
\end{proof}

\begin{rem} \rm The proof of Theorem~\ref{theta_new_def} gives a new insight for Corollary~\ref{no_skein}. Indeed, ff we tried to repeat the same proof for a crossing of $L_+$ involving the same component, then the partitions of $L_0$ induced by the partitions of $L_+$ would not run over all partitions of $L_0$, since $L_0$ has one more component (compare with the proof of Corollary~\ref{no_skein}). Hence, we would not obtain the Homflypt skein relation on a crossing involving the same component.
\end{rem}

To summarize, in this paper we defined the invariant $\Theta$ as a generalization of the invariant $\Theta_d$, \textit{skein-theoretically} (Theorem~\ref{invE}) and we proved its well-definedness via the theory of tied braids and tied links. A direct skein-theoretic proof of the well-definedness of $\Theta$ is given in \cite{kala}. Alternatively, the invariant $\Theta$ could be established purely \textit{algebraically} using the trace $\Jtrs$ on the algebra $\YH$, the isomorphism of $\YH^{\rm{(br)}}$ with the algebra of braids and ties \cite{EsRy}, and then considering the value $E_D$ as a parameter $E$. Finally, the invariant $\Theta$ could be established purely \textit{combinatorially} using via the closed formula  of Theorem~\ref{theta_linking_P}.

\section{Further research directions}\label{new research directions}
The results of this paper open to several research directions. To point out a few:

\subsection{Exploring further the invariants $\Theta_d$, $\Theta$ and $\Delta_d$, $\overline{\Delta}$}
We believe and we hope that the six pairs of links of Homflypt-equivalent links that are detected by the invariants $\Theta_d$ and $\Theta$, and possibly some more still to be found, will shed light on some of the geometric properties and topological limitations of the Homflypt polynomial. Further, it would be interesting to know whether Theorem~\ref{2comp} generalizes to $P$-equivalent links with more than $2$ components. Another important question is whether and when $\Theta_d$-equivalence implies $P$-equivalence. In \cite{kala} the invariants $\Theta_d$ and $\Theta$ are placed in a more general skein-theoretic context and associated state sum models for the new invariants are discussed.
Finally, it is worth  exploring further  the relation between the invariants $\Theta_d$ and $\Delta_d$, and between $\overline{\Delta}$ and $\Theta$.

\subsection{Revisiting the invariants for framed links}
Under the light of our new results, one could investigate further the framed link invariants $\Phi_{d,D}(q,z)$, in particular with respect to the special skein relation of Proposition~\ref{skein_ei}. In the same spirit one can investigate further the classical and framed link invariants derived from the Framization of the Temperley--Lieb algebra, ${\rm FTL}_{d,n}(q)$ \cite{gojukolaf, chpo2}. In \cite{gojukolaf} it has been checked that the related invariants also distinguish the six pairs of $P$-equivalent links. So, it would make sense to construct invariants for 3--manifolds, in analogy to the Witten invariants. Finally, it might also be meaningful to explore further the transverse link invariants from the Yokonuma--Hecke algebras defined in \cite{chjakala}.

\subsection{Exploring further the $p$-adic framed link invariants}In \cite{jula1} the $p$-adic framed braid group ${\mathcal F}_{\infty ,n}$ and the $p$-adic Yokonuma-Hecke algebra ${\rm Y}_{\infty,n}(u)$ were constructed as inverse limits of the modular framed braid groups and the classical Yokonuma-Hecke algebras respectively. Topologically dense sub-structures were explored and approximations of $p$-adic elements through the dense sub-structures were found \cite{jula2}. Thus, a $p$-adic Markov trace was constructed on the algebras  ${\rm Y}_{\infty,n}(u)$ using the traces $\Jtru$.  Moreover, solutions of the {\rm E}--system lift to solutions on the $p$-adic level \cite{jula1}. Therefore, the specialized traces $\Jtrsu$ also yield  Markov traces on the $p$-adic Yokonuma--Hecke algebras.  Consequently, a $p$-adic  invariant for oriented framed links, $\Gamma_{p^{\infty}}(u,\wt{z})$, and a $p$-adic invariant for oriented classical links, $\Delta_{p^{\infty}}(u,\wt{z})$, were constructed through the invariants $\Gamma_{d,D}(u,\wt{z})$ and $\Delta_d(u,\wt{z})$. Further, in \cite{jula3} the above were adapted to the adelic setting.  For more details, see \cite{jula2,jula3}. One could check whether Theorem~\ref{specialtrace} carries through to the specialized $p$-adic (respectively adelic) trace.
Moreover, one could adapt the above construction to the new presentation of the algebra $\YH(q)$ and obtain $p$-adic (respectively adelic) link invariants $\Phi_{p^{\infty}}(q,z)$ and $\Theta_{p^{\infty}}(q,z)$. Then one could investigate whether our main results apply to these invariants, especially when restricted to the topologically dense sub-structures that are isomorphic to ${\mathcal F}_n$, which are explored in \cite{jula2}, and for which there is no modular restriction on the framings.

\subsection{Invariants from other framization algebras}
Apart from the framization algebras related to the Temperley--Lieb algebra studied in \cite{gojukola, gojukolaf, chpo1,chpo2},
which are related to the Jones polynomial, one could study further the new algebras defined in \cite{jula5,chpa2}, which are framizations of known knot algebras: of the BMW algebra, of the cyclotomic $B$--type Hecke algebras and the affine Hecke algebra of type $A$, and of the singular Hecke algebra.
These algebras are related respectively to: the  Kauffman polynomial, the Lambropoulou invariants for the solid torus, and the  Kauffman--Vogel and Paris--Rabenda invariants for singular links. For each of these framization algebras, it is worth exploring the possibility of constructing link invariants via Markov traces, since, by our results, we are likely to obtain new link invariants different from the above. Further, one could look into the direction of constructing appropriate quotients of these algebras, which would lead to unoriented link invariants. Adapting the skein methods described in this paper to appropriate presentations of the framization algebras, one would most likely obtain stronger invariants than the above.

\vfill
\pagebreak
\appendix

\section{Computational data}\label{computations}
For every pair of links of Table~\ref{pairs} we give the corresponding braid words and the value of the invariants $\Theta_d$ on them. For the braid words we use the notation $\{ a_1, \ldots, a_n \}$ which corresponds to the word $\s_{a_1}^{{\rm sign}(a_1)} \cdots \s_{a_n}^{{\rm sign}(a_n)}$. For example, $\{1, -2, 1, -2\}$ corresponds to the braid word $\s_1 \s_2^{-1} \s_1 \s_2^{-1}$. The braid words have been obtained from {\it LinkInfo} \cite{chli}.

\small
\smallbreak
\begin{center}
\def\arraystretch{1.5}
\begin{tabular}{|c|c|}\hline
$L11n358\{0,1\}$ & $\{1, -2, -3, -4, 3, 3, -5, 4, -3, 2, -1, -3, -2, -4, 3, -2, -2, -2, 5, 4, -3\}$\\\hline
\multicolumn{2}{|l|}{\rule{0pt}{7ex} $\begin{aligned}
	&\left(E_D^2 \lambda_D^4 (q-1)^2 q^6 (q+1)^2 \right)^{-1} \left(E_D \lambda_D +E_D \lambda_D q^4-2 E_D \lambda_D q^2+\lambda_D 
   q^2-q^2\right)\\
   &(E_D \lambda_D + E_D \lambda_D q^{12}-E_D \lambda_D^2
   q^{10}-2 E_D \lambda_D q^{10}-E_D q^{10}+3 E_D \lambda_D ^2 q^8
   +4 E_D \lambda_D q^8 +2 E_D q^8-4 E_D \lambda_D ^2 q^6\\
   &-6 E_D \lambda_D q^6 -2 E_D q^6+3 E_D \lambda_D^2 q^4
   +4 E_D \lambda_D q^4 + 2 E_D q^4-E_D \lambda_D ^2
   q^2-2 E_D \lambda_D q^2-E_D q^2+\lambda_D q^6-q^6)
   \end{aligned}$}\\\hline
$L11n418\{0,0\}$ & $\{-1, -2, 3, -2, -3, 2, -1, -3, -3, 2, -3\}$\\\hline
\multicolumn{2}{|l|}{\rule{0pt}{8ex} $\begin{aligned}&\left( E_D^2 \lambda_D^4 q^6
   \left(q^2-1\right)^2 \right)^{-1} \Big[ E_D^2 \left(q^2-1\right)^4 \Big(\lambda_D^2+\lambda_D^2
   q^8+\lambda_D \left(-\lambda_D^2+\lambda_D -2\right) q^6
   +\left(4\lambda_D^2-\lambda_D +1\right) q^4\\
   &+\lambda_D \left(-\lambda_D
   ^2+\lambda_D -2\right) q^2\Big)
   +E_D \left(q^2-1\right)^2 q^4
   \big(2 (\lambda_D -1) \lambda_D +2 (\lambda_D -1) \lambda_D
   q^4\\
   &-\left(\lambda_D^3-3 \lambda_D^2+4 \lambda_D-2\right)
   q^2\big)+(\lambda_D -1)^2 q^8 \Big]
\end{aligned}$}\\\hline
\end{tabular}
\smallbreak
\begin{tabular}{|c|c|}\hline
$L11n467\{0,1\}$ & $\{1, -2, -3, 4, 3, -2, 3, 3, -2, -4, 5, 4, 3, -2, -1, -2, -3, -2, -2, -4, 3, -2, -5\}$\\\hline
\multicolumn{2}{|l|}{\rule{0pt}{8ex} $\begin{aligned}&-\left(E_D^2 \lambda_D ^4 \left(q^2-1\right)^5 \right)^{-1} \left(q^{-2}-1\right)^3 \Big(E_D^2 \lambda_D 
   \left(q^2-1\right)^4 \big(\lambda_D  (\lambda_D +2) +\lambda_D (\lambda_D +2) q^8\\
   &-\left(\lambda_D ^3+3 \lambda_D ^2+2 \lambda_D
   +2\right) q^6+\left(\lambda_D ^3+4 \lambda_D ^2+6 \lambda_D +1\right)
   q^4 -\left(\lambda_D ^3+3 \lambda_D ^2+2 \lambda_D +2\right)
   q^2\big)\\
   &+E_D (\lambda_D -1) \left(q^2-1\right)^2 q^2 \Big(\lambda_D
   +\lambda_D  q^8 -\left(\lambda_D ^2+1\right) q^6+\lambda_D  (\lambda_D
   +4) q^4-\left(\lambda_D ^2+1\right) q^2\Big) +(\lambda_D -1)^2
   q^8\Big)
   \end{aligned}$}\\\hline
$L11n527\{0,0\}$ & $\{1, 2, -3, -4, -3, 5, 4, -3, -2, -1, -3, -4, -3, 2, -3, -3, -5, 4, -3, 2, 2\}$\\\hline
\multicolumn{2}{|l|}{\rule{0pt}{8ex} $\begin{aligned}&-\left( E_D^2 \lambda_D ^4
   \left(q^2-1\right)^5 \right)^{-1} \left(q^{-2}-1\right)^3 \Big(E_D^2
   \left(q^2-1\right)^4 \big(\lambda_D ^2 (\lambda_D +2) +\lambda_D ^2
   (\lambda_D +2) q^8\\
   &-\lambda_D  \left(\lambda_D ^3+3 \lambda_D ^2+\lambda_D
   +3\right) q^6 +\left(\lambda_D ^4+3 \lambda_D ^3+7 \lambda_D
   ^2+1\right) q^4 \lambda_D  \left(\lambda_D ^3+3 \lambda_D ^2+\lambda_D
   +3\right) q^2\big)\\
   &+E_D (\lambda_D -1) \left(q^2-1\right)^2 q^4
   \left(2 \lambda_D +2 \lambda_D  q^4-\left(\lambda_D ^2-2 \lambda_D
   +2\right) q^2\right) +(\lambda_D -1)^2 q^8\Big)
\end{aligned}$}\\\hline
\end{tabular}
\smallbreak
\begin{tabular}{|c|c|}\hline
$L11n325\{1,1\}$ & $\{-1, 2, -1, 2, -1, -2, -2, 3, -2, 3, -2\}$\\\hline
\multicolumn{2}{|l|}{\rule{0pt}{6ex} $\begin{aligned}
   &\Big(\lambda_D^{-3}E_D^2 \left(q^2-1\right)^4 \left(q^2-\lambda_D \right)
   \left(\lambda_D  q^2-1\right) \left(q^4-(\lambda_D +1)
   q^2+1\right)-\lambda_D^{-3} \big(E_D (\lambda_D -1)
   \left(q^2-1\right)^2 q^2 (\lambda_D\\
   &+\lambda_D q^8 -\left(\lambda_D ^2+1\right) q^6 -(\lambda_D -2) \lambda_D 
   q^4-\left(\lambda_D ^2+1\right) q^2)\big)+\left(\lambda_D^{-1}-1\right)^2 q^8\Big)
   \left(E_D^2
   \left(q^{-2}-1\right)^2 q^{10}\right)^{-1}
   \end{aligned}$}\\\hline
$L11n424\{0,0\}$ & $\{-1, 2, -1, -2, 3, -2, -2, 1, -2, 3, -2\}$\\\hline
\multicolumn{2}{|l|}{\rule{0pt}{6ex} $\begin{aligned}&\left(E_D^2
   \left(q^{-2}-1\right)^2 q^{10}\right)^{-1}
   \Big(\lambda_D^{-3} E_D^2 \left(q^2-1\right)^4 \left(q^2-\lambda_D \right)
   \left(q^4-2 \lambda_D  q^2+1\right) \left(\lambda_D 
   q^2-1\right)\\
   &-\lambda_D^{-3}E_D (\lambda_D -1)
   \left(q^2-1\right)^2 q^4 \left(2 \lambda_D +2 \lambda_D  q^4-\left(3
   \lambda_D ^2+2\right) q^2\right)+\left(\lambda_D^{-1}-1\right)^2 q^8\Big)
\end{aligned}$}\\\hline
\end{tabular}
\smallbreak
\begin{tabular}{|c|c|}\hline
$L10n79\{1,1\}$ & $\{-1, 2, -1, 2, -1, -2, -2, -2, -2, -2\}$\\\hline
\multicolumn{2}{|l|}{\rule{0pt}{6ex} $\begin{aligned}&\left(E_D^2 \lambda_D ^2 \left(q^{-2}-1\right)^2 q^{10}\right)^{-1} \Big(-\lambda_D^{-2} E_D^2 \left(q^2-1\right)^4 \left(q^4+1\right)
   \left(q^2-\lambda_D \right) \left(\lambda_D  q^2-1\right)\\
   &+\lambda_D^{-2} E_D (\lambda_D -1) \left(q^2-1\right)^2
   \left(q^4+q^2+1\right) q^2 \left(\lambda_D +\lambda_D 
   q^4-q^2\right)+\left(\lambda_D^{-1}-1\right)^2
   q^8\Big)
   \end{aligned}$}\\\hline
$L10n95\{1,0\}$ & $\{-1, 2, -1, -2, -2, -2, 1, -2, -2, -2\}$\\\hline
\multicolumn{2}{|l|}{\rule{0pt}{8ex} $\begin{aligned}&\left(E_D^2 \lambda_D ^2 \left(q^{-	2}-1\right)^2
   q^{10}\right)^{-1} \Big(\Big(E_D^2 \left(q^2-1\right)^4 (\lambda_D +\lambda_D 
   q^8+\left(-2 \lambda_D ^2+\lambda_D -1\right) q^6-\left(\lambda_D ^2-4\lambda_D +1\right) q^4\\
   &+\left(-2 \lambda_D ^2+\lambda_D -1\right)
   q^2)-\lambda_D^{-2} \Big) +\lambda_D ^{-2} E_D (\lambda_D -1)
   \left(q^2-1\right)^2 q^4 \left(2 \lambda_D +2 \lambda_D  q^4+(2
   \lambda_D -3) q^2\right)\\
   &+\left(\lambda_D^{-1}-1\right)^2 q^8\Big)
\end{aligned}$}\\\hline
\end{tabular}
\smallbreak
\begin{tabular}{|c|c|}\hline
$L11a404\{1,1\}$ & $\{-1, -1, 2, 2, -1, 3, 2, 2, -1, 2, 2, -3, 2\}$\\\hline
\multicolumn{2}{|l|}{\rule{0pt}{8ex} $\begin{aligned}&\left(E_D^2 q^8 \left(q^2-1\right)^2\right)^{-1} \Big(-E_D^2 \left(q^2-1\right)^4 \left(q^4-\lambda_D  q^2+1\right)
   \Big(\lambda_D +\lambda_D  q^8-\left(\lambda_D ^2+1\right)
   q^6+\lambda_D  (\lambda_D +4) q^4\\
   & -\left((\lambda_D ^2+1\right)
   q^2\Big) +E_D (\lambda_D -1) \left(q^2-1\right)^2 q^4 \Big(\lambda_D
   +\lambda_D  q^8-\left(\lambda_D ^2+3\right) q^6\\
   &+(5 \lambda_D +1\Big) q^4-\left(\lambda_D ^2+3\right) q^2)+(\lambda_D -1)^2 q^8
   \left(q^4-\lambda_D  q^2+1\right)\Big)
   \end{aligned}$}\\\hline
$L11a428\{0,1\}$ & $\{1, -2, 3, -2, 1, 1, 1, -2, 3, -2, 1\}$\\\hline
\multicolumn{2}{|l|}{\rule{0pt}{8ex} $\begin{aligned}&\left(E_D^2 q^8 \left(q^2-1\right)^2\right)^{-1} \Big(-E_D^2 \left(q^2-1\right)^4 \left(q^4-\lambda_D  q^2+1\right)
   \Big(\lambda_D +\lambda_D  q^8 -\left(\lambda_D ^2+1\right) q^6+(4
   \lambda_D +1) q^4\\
   &-\left(\lambda_D ^2+1\right) q^2\Big) -E_D (\lambda_D
   -1) \left(q^2-1\right)^2 q^4 \Big(q^8+(1-3 \lambda_D )
   q^6\\
   &+(\lambda_D  (2 \lambda_D -1)+1) q^4+(1-3 \lambda_D )
   q^2+1\Big)+(\lambda_D -1)^2 q^8 \left(q^4-\lambda_D 
   q^2+1\right)\Big)
\end{aligned}$}\\\hline
\end{tabular}
\smallbreak
\begin{tabular}{|c|c|}\hline
$L10n76\{1,1\}$ & $\{1, 2, -3, 4, -3, -2, -1, -3, -2, 3, -2, 3, -2, -3, -4, -3\}$\\\hline
\multicolumn{2}{|l|}{\rule{0pt}{6ex} $\begin{aligned}&\left(E_D^2 \lambda_D ^4
   \left(q^{-2}-1\right)^2 q^8\right)^{-1} \Big(-E_D^2 (\lambda_D +1) \left(q^2-1\right)^4 \left(q^2-\lambda_D
   \right) \left(\lambda_D  q^2-1\right)\\
   &+E_D (\lambda_D -1)
   \left(q^2-1\right)^2 q^2 \left(\lambda_D  (\lambda_D +1)+\lambda_D 
   (\lambda_D +1) q^4-q^2\right) +(\lambda_D -1)^2 q^6\Big)
   \end{aligned}$}\\\hline
$L11n425\{1,0\}$ & $\{-1, 2, -1, -3, -3, -2, 1, 3, -2, -3, -3\}$\\\hline
\multicolumn{2}{|l|}{\rule{0pt}{6ex} $\begin{aligned}&\left(E_D^2 \lambda_D ^4 \left(q^{-2}-1\right)^2 q^8\right)^{-1} \Big(-E_D^2 (\lambda_D +1) \left(q^2-1\right)^4 (\lambda_D +\lambda_D
    q^4+\left(-2 \lambda_D ^2+\lambda_D -1\right) q^2)\\
    &+E_D \left(2 \lambda_D ^3-3 \lambda_D +1\right) \left(q^2-1\right)^2 q^4+(\lambda_D
   -1)^2 q^6\Big)
\end{aligned}$}\\\hline
\end{tabular}
\end{center}
\normalsize

\vfill
\pagebreak

\section{A formula for the invariant $\Theta$ by W.B.R.~Lickorish}\label{lickorish}

Recall that Theorem~\ref{invE} asserts the following. An invariant $\Theta$ of oriented links, with indeterminates $q, \lambda, E$ is given by:

\begin{enumerate}
\item $\Theta (L) = E^{1-k}P(L)$ when $L$ is the union of $k$ {\it unlinked} knots;
\item $\lambda^{-{1/2}}\Theta (L_+) - \lambda^{{1/ 2}}\Theta (L_-) = (q - q^{-1})\Theta (L_0)$, for any Conway triple $L_+, L_-, L_0$ in which the arcs of $L_+$ are in {\it different} components of the link;
\item $P$, the Homflypt polynomial, is defined by
\begin{center}
$\lambda^{-{1/2}}P (L_+) - \lambda^{{1/ 2}}P(L_-) = (q - q^{-1})P(L_0),$
\end{center}
with no restriction on the arcs in $L_+$ and $P({\rm unknot}) = 1$.
\end{enumerate}

\smallbreak

\begin{thm}\label{theta_linking_P}
Let $L$ be an oriented link with $n$ components.  Then
\begin{equation}
\Theta (L) = \sum_{k=1}^n \mu^{k-1}E_k \sum_\pi \lambda^{\nu(\pi)}P(\pi L)
\end{equation}
where the second summation is over all partitions $\pi$ of the components of $L$ into $k$ (unordered) subsets and $P(\pi L)$ denotes the product of the Homflypt polynomials of the $k$ sublinks of $L$ defined by $\pi$. Furthermore, $\nu(\pi)$ is the sum of all linking numbers of pairs of components of $L$ that are in distinct sets of $\pi$, $E_k = (E^{-1} - 1)(E^{-1} - 2) \cdots (E^{-1} - k + 1)$, with $E_1 =1$, and $\mu = \frac{\lambda^{-{1/2}} - \lambda^{{1/ 2}}}{q - q^{-1}}$.
\end{thm}

\begin{proof}
Suppose that a diagram of $L$ is given. The proof is by induction on $n$ and on the number, $u$, of crossing changes between distinct components required to change $L$ to $n$ unlinked knots. If $n=1$ there is nothing to prove. So assume the result true for $n-1$ components and $u-1$ crossing changes and prove it true for $n$ and $u$.

The induction starts when $u = 0$. Then $L$ is the unlinked union of $n$ components $L_1, L_2, \dots, L_n$ and all linking numbers are zero. A classic elementary result concerning the Homflypt polynomial shows that $P(L) = \mu^{n-1}P(L_1)P(L_2)\cdots P(L_n)$. Furthermore, in this situation, for any $k$ and $\pi$, $P(\pi L) = \mu^{n-k}P(L_1)P(L_2)\cdots P(L_n)$. So it is required to prove that
\begin{equation}\label{stirling}
E^{1-n} = \sum_{k=1}^n S(n,k)(E^{-1} - 1)(E^{-1} - 2) \cdots (E^{-1} - k + 1)
\end{equation}
where $S(n,k)$ is the number of partitions of a set of $n$ elements into $k$ subsets. However, in the theory of combinatorics, $S(n,k)$ is known as a Stirling number of the second kind and this required formula is a well known result about such numbers.

Now suppose that $u > 0$. Suppose that in a sequence of $u$ crossing changes that changes $L$, as above, into unlinked knots, the first change is to a crossing $c$ of sign $\epsilon$ between components $L_1$ and $L_2$. Let $L^\prime$ be $L$ with the crossing changed and $L^0$ be $L$ with the crossing annulled. Now, from the definition of $\Theta$,
$$\Theta (L) = \lambda ^\epsilon \Theta (L^\prime) + \epsilon \lambda^{\epsilon/2}(q - q^{-1})\Theta (L^0).$$

\noindent The induction hypotheses imply that the result is already proved for $L^\prime$ and $L^0$ so
\begin{equation}\label{star}
\Theta (L) = \lambda ^\epsilon \sum_{k=1}^n \mu^{k-1}E_k \sum_{\pi^\prime} \lambda^{\nu(\pi^\prime)}P(\pi^\prime L^\prime) + \epsilon \lambda^{\epsilon/2}(q - q^{-1})\sum_{k=1}^{n-1} \mu^{k-1}E_k \sum_{\pi^0} \lambda^{\nu(\pi^0)}P(\pi^0 L^0),
\end{equation}
where $\pi^\prime$ runs through the partitions of the components of $L^\prime$ and $\pi^0$ those of $L^0$.

A sublink $X^0$ of $L^0$ can be regarded as a sublink $X$ of $L$ containing $L_1$ and $L_2$ but with $L_1$ and $L_2$ fused together by annulling the crossing at $c$. Let $X^\prime$ be the sublink of $L^\prime$ obtained from $X$ by changing the crossing at $c$. Then
$$
P (X) = \lambda ^\epsilon P (X^\prime) + \epsilon \lambda^{\epsilon/2}(q - q^{-1}) P(X^0).
$$
This means that the second (big) term in \eqref{star} is
\begin{equation}\label{big_term}
\sum_{k=1}^{n-1} \mu^{k-1}E_k \sum_{\rho} \lambda^{\nu(\rho)} \biggl( P(\rho L) - \lambda^\epsilon P(\rho^\prime L^\prime)\biggr),
\end{equation}
where the summation is over all partitions $\rho$ of the components of $L$ for which $L_1$ and $L_2$ are in the same subset and $\rho^\prime$ is the corresponding partition of the components of $L^\prime$.

Note that, for any partition $\pi$ of the components of $L$ inducing partition $\pi^\prime$ of $L^\prime$, if $L_1$ and $L_2$ are in the same subset then $\nu (\pi) = \nu (\pi^\prime)$, otherwise $\nu (\pi) = \nu (\pi^\prime) + \epsilon$. So, when $L_1$ and $L_2$ are in different subsets
\begin{equation}\label{diff_subsets}
\lambda^{\nu(\pi^\prime) + \epsilon} P(\pi^\prime L^\prime) = \lambda^{\nu(\pi)}P(\pi L). 
\end{equation}
Thus, substituting \eqref{big_term} in \eqref{star} we obtain:
$$
\Theta (L) = \sum_{k=1}^n \mu^{k-1}E_k \biggl( \sum_{\pi^\prime} \lambda^{\nu(\pi^\prime) + \epsilon}P(\pi^\prime L^\prime) + \sum_\rho\bigl(\lambda^{\nu(\rho)}P(\rho L) - \lambda^{\nu(\rho )+ \epsilon}P(\rho^\prime L^\prime)\bigr) \biggr),
$$
where $\pi^\prime$ runs through all partitions of $L^\prime$ and $\rho$ through partitions of $L$ for which $L_1$ and $L_2$ are in the same subset. Note that, for $k=n$ the second sum is zero. Hence, using also \eqref{diff_subsets}, we obtain:
$$
\Theta (L) = \sum_{k=1}^n \mu^{k-1}E_k \sum_\pi \lambda^{\nu(\pi)}P(\pi L)
$$
and the induction is complete.

\end{proof}

\end{document}